\documentclass[a4paper, 11pt]{article}
\usepackage [a4paper,left=2.5cm,bottom=2.5cm,right=2.5cm,top=2.5cm]{geometry}

\usepackage[english]{babel}

\usepackage[toc,page]{appendix}

\usepackage[utf8]{inputenc} 	
\usepackage[T1]{fontenc}
\usepackage{lmodern}
\usepackage{yhmath}
\usepackage{comment}
\usepackage{bm}
\usepackage{enumerate}   
\usepackage{comment}

\usepackage{amsthm,amsmath,amssymb,amsbsy,bbm,mathrsfs,supertabular,
eurosym,graphicx,enumitem,xcolor, dsfont}
\usepackage{mathtools}
\usepackage{ulem}

\newtheoremstyle{BBstyle0}  {}{}{\itshape}{}{\bfseries}{}{6pt}{}
\newtheoremstyle{BBstyle1}  {3pt}{3pt}{\rmfamily}{}{\itshape}{: }{3pt}{}
\newtheoremstyle{BBstyle2}  {3pt}{3pt}{\itshape}{}{\bfseries\large}{}{0pt}{}
\newtheoremstyle{BBstyle3}  {}{}{\itshape}{}{\bfseries}{: }{3pt}{}
\newtheoremstyle{BBstyle4}  {}{}{\rmfamily}{}{\bfseries}{}{6pt}{}
  
\newtheorem{thm}{Theorem}
\newtheorem{lem}{Lemma}

\newtheorem{prop}{Proposition}
\newtheorem{df}{Definition}

\newtheorem{rem}{Remark}

\theoremstyle{definition}

\usepackage[english]{babel}

\usepackage{enumitem}
\usepackage{hyperref}  

\newcommand{\norm}[1]{\left\|{#1}\right\|}

\newcommand{\argmin}{\mathop{\rm argmin}}

\newcommand{\esssup}{\mathop{\rm ess\ sup}}
\newcommand{\essinf}{\mathop{\rm ess\ inf}}

\newcommand{\Card}{\mathop{\rm Card}\nolimits}

\newcommand{\E}{{\mathbb{E}}}

\newcommand{\R}{{\mathbb{R}}}

\newcommand{\cB}{{\mathcal{B}}}

\newcommand{\cE}{{\mathcal{E}}}

\newcommand{\cI}{{\mathcal{I}}}
\newcommand{\cJ}{{\mathcal{J}}}

\newcommand{\cL}{{\mathcal{L}}}

\newcommand{\cP}{{\mathcal{P}}}
\newcommand{\cQ}{{\mathcal{Q}}} 
\newcommand{\cR}{{\mathcal{R}}}

\newcommand{\cX}{{\mathcal{X}}}
 
\newcommand{\cZ}{{\mathcal{Z}}}

\newcommand{\modar}{\color{black}}

\newcommand{\revar}{\color{black}}
\newcommand{\revarn}{\color{black}}
\newcommand{\revarna}{\color{black}}

\newcommand{\revis}{\color{black}}

\newcommand{\theconvex}{C}

\usepackage[textwidth=1.8cm, textsize=tiny]{todonotes}

\newcommand{\indi}[1]{\mathds{1}_{#1}}
\DeclarePairedDelimiter{\abs}{\lvert}{\rvert}

\newcommand{\proba}{\mathfrak{P}}
\newcommand{\Borelof}[1]{\mathcal{B}_{#1}}
\newcommand{\CompleteBorelof}[1]{\widetilde{\mathcal{B}}_{#1}}

\title{Factorization by extremal privacy mechanisms: new insights into efficiency}
\author{Chiara Amorino\thanks{ Universitat Pompeu Fabra and Barcelona School of Economics, Department of Economics and Business, Ram\'on Trias Fargas 25-27, 08005, Barcelona, Spain. The author gratefully acknowledges financial support of PID2022-138268NB-I00/AEI/10.13039/501100011033.} \qquad Arnaud Gloter \thanks{Laboratoire de Math\'ematiques et Mod\'elisation d'Evry, CNRS, Univ Evry, Universit\'e Paris-Saclay, 91037, Evry, France.}} 

\begin{document}
\maketitle




\begin{abstract}
	\noindent We study the problem of efficiency under $\alpha$-local differential privacy ($\alpha$-LDP) in both discrete and continuous settings. Building on a factorization result—which shows that any privacy mechanism can be decomposed into an extremal mechanism followed by additional randomization—we reduce the Fisher information maximization problem to a search over extremal mechanisms. 
	The representation of extremal mechanisms requires working in infinite-dimensional spaces and invokes advanced tools from convex and functional analysis, such as Choquet’s theorem. 
	{\revarna Using this factorization result, we establish} 
	matching upper and lower bounds on the Fisher information in the high-privacy regime ($\alpha \to 0$)
	{\revarna for regular models. We also prove}
	that the maximization problem always admits a solution for any $\alpha$.
	{\revarna As other application,} 
	we consider the problem of estimating the parameter of a uniform distribution on $[0, \theta]$ under $\alpha$-LDP. Guided by our theoretical findings, we design an extremal mechanism that yields a consistent and asymptotically efficient estimator in high privacy regime. Numerical experiments confirm our theoretical results.
\end{abstract}
\noindent
\textit{MSC2020 subject classifications: Primary 62F12, 68P27; secondary 62B15, 46A55}

\setcounter{tocdepth}{3} 
\tableofcontents

\section{Introduction}{\label{s: intro}
	With the generalization of large-scale data collection and the ever-increasing computational power of modern computers, statisticians now have access to unprecedented amounts of data. This revolution is not merely quantitative but also qualitative, as it fundamentally changes the way knowledge is built. Statisticians constructs models based on data arising from the natural and social sciences (e.g., physics, chemistry, sociology). However, this data is often sensitive. For example, several recent studies have relied on sensitive medical data, 
	where particular care must be taken to ensure privacy, leading to a necessary trade-off between statistical utility and privacy protection.
	
	Classical privacy-preserving techniques, such as data permutation and traditional anonymization procedures, have proven to be largely ineffective. In fact, numerous privacy breaches have been documented using these methods (see, for example, \cite{Nar08}). One of the most famous examples is the case of the Netflix Prize competition, where anonymized movie ratings were successfully de-anonymized, forcing Netflix to cancel its second competition in 2010 \cite{KeaRot19}. For a detailed and engaging account of this story, we also recommend the introduction of the PhD thesis \cite{Lal23}, which provides an in-depth, well-written, and slightly romanticized (but very enjoyable) description of the events.
	
	In response to such challenges, researchers in statistics, databases, and computer science have worked to formalize privacy as a mathematical framework for limiting disclosure risks. This formalization aims to protect individual data while still allowing for the statistical analysis of aggregated information. The foundational work in this direction is that of Dwork et al. \cite{Dwo06}, who introduced the notion of global differential privacy (also known as central differential privacy). In this setting, a trusted aggregator has access to the entire dataset and produces privatized outputs—while the raw data remains hidden from external users. A practical example of this model would be a centralized data center that stores all user information but only releases privatized statistics.
	
	Research on central privacy has since flourished, but alternative models have also emerged. Notably, local differential privacy (LDP) does not assume a trusted aggregator. Instead, each user privatizes their own data before sending it to a central aggregator. At first glance, local privacy may appear more desirable since it offers stronger privacy guarantees. However, this increased protection often comes at a significant cost to statistical utility, leading to lower-quality estimates compared to those obtained under central privacy.
	
	In recent years, several important extensions and refinements of differential privacy have been developed. These include concentrated differential privacy \cite{BunSte16, DwoRot16}, pufferfish privacy \cite{puff}, and componentwise local differential privacy \cite{AmoGlo23}, among others.

	Major technology companies such as Apple \cite{5LDP, 49LDP} and Google \cite{1LDP, 23LDP} have already adopted local differential privacy mechanisms, reflecting the significant real-world impact of privacy-preserving measures on billions of devices. As a result, understanding the fundamental trade-off between local differential privacy and statistical utility has become crucial, leading to a surge of academic research focused on statistical inference under privacy constraints.
	
	A key feature of local differential privacy, as introduced in Section \ref{s: pb formulation and main} and specifically formalized in Equation 
	{\revarna \eqref{Eq : formalism LDP usuel NI}, is} that the privacy level is directly controlled by the parameter $\alpha \in [0, \infty)$. When $\alpha = 0$, perfect privacy is achieved, while as $\alpha \to \infty$, the privacy constraint progressively disappears. A particularly active area of current research focuses on fixing a desired privacy level $\alpha$ and, among all random mechanisms that satisfy this privacy constraint, identifying those that optimize statistical utility according to a specified criterion. This is precisely the approach adopted in the present manuscript, which specifically aims to investigate privacy mechanisms that are optimal in the sense of maximizing the Fisher information.

	The foundational steps in the study of statistical inference under local differential privacy constraints can be found in early works such as \cite{Martin} and \cite{WassermanZhou10}. Since then, the field has experienced rapid and extensive development, with an increasing number of contributions spanning various areas of statistical methodology. Recent advancements include investigations into hypothesis testing \cite{4CLDP, 29CLDP}, M-estimation \cite{Avella}, robustness \cite{31CLDP}, change-point detection \cite{6CLDP}, as well as estimation of means and medians \cite{Martin}. Moreover, significant progress has been made in nonparametric estimation \cite{10CLDP, 11CLDP, Kro24} and parametric inference \cite{AmoGloHal, Asi22, DucRog19, Jos19, KalSte}. Particularly within the parametric setting, considerable attention has been devoted to the pursuit of efficient estimators under privacy constraints—a central theme in recent research.
	
	In the centralized (global) differential privacy paradigm, Smith’s early work \cite{Smi08, Smi11} laid the foundations for efficient estimation under privacy constraints. By contrast, in the local model, \cite{BarChe20} derived upper bounds on the Fisher information achievable under LDP but did not construct mechanisms that attain them. Nam and Lee \cite{NamLee22} then investigated the maximization of Fisher information in the LDP setting, albeit restricted to one-bit communication schemes.
	
	More recently, \cite{DucRua24} introduced the notion of $L_1$-information to precisely characterize the local minimax risk of regular parametric models under LDP, up to universal constants. The closest precursor to our work is the elegant treatment by \cite{steinbergerEfficiencyLocalDifferent}, which establishes both the existence of asymptotically efficient locally private estimators and explicit constructions achieving minimum variance. 
	
	\subsection{Our contribution}
	Following the line described above, we pose the question: among all mechanisms in \(\mathcal{Q}_\alpha:= \{ q : q \text{ satisfies } \alpha\text{-LDP} \}\), which one maximizes 
	{\revar the statistical utility contained in the public data?
		A common feature of many statistical utilities is that they decrease when the data are randomized. An example of such statistical utility is the Fisher information. Due to its importance, we will focus on it in this work. However, considering other utilities is possible (see Remark \ref{rem : utility Kairouz}).

		Our first result is a factorization lemma in a discrete probability setting (Lemma~\ref{lem : factorisation cas fini}):
		it shows that any \(\alpha\)-LDP mechanism can be decomposed into two consecutive mechanisms, where the first one is extremal—that is, it saturates the privacy constraint (see Eq.~\eqref{eq : extremal constrainte discret}).
		This result shows that when maximizing some statistical utility among all $\alpha$-LDP mechanism, it is sufficient to consider extremal mechanisms only. Indeed, applying the first extremal mechanism {\revarna provides the} 
		$\alpha$-LDP guarantee on the data. The second step of the factorization reduces the statistical utility without increasing the level of privacy over the $\alpha$-LDP constraint. Hence, it can be removed when searching for an optimal privacy mechanism.
		
		We also refer to these extremal mechanisms as staircase mechanisms.
		Staircase mechanisms have become a central object of study in the finite-alphabet setting. Let us mention some key references on the topic. Geng and Viswanath \cite{GenVis12} introduced the one-dimensional staircase mechanism, showing both theoretically and empirically that its piecewise-constant noise outperforms the Laplace mechanism. They later extended it to \(d\) dimensions \cite{Gen15}, again proving superiority over Gaussian or Laplace noise for \(d\ge2\). More recently, \cite{Kul23} proposed a staircase-like scheme for one-dimensional mean estimation, while \cite{kairouzExtremalMechanismsLocal2015} formulated the privacy–utility trade-off as a constrained optimization and proved that, on finite spaces, the family of staircase mechanisms contains the exact optima for a wide class of information-theoretic utilities.
		
		{\revarna In this work, we focus on the maximization of the Fisher information.}
		Given a regular parametric model {\revar \(\mathcal{P}=(p_\theta)_{\theta\in\Theta}\), with $\Theta \subset \mathbb{R}$, the randomized} 
		model \(q\circ\mathcal{P}\) remains regular (see \cite{steinbergerEfficiencyLocalDifferent} and recap in Section ~\ref{ss:Parametric models and Fisher information}), with Fisher information \(\mathcal{I}_{\theta_0}(q\circ\mathcal{P})\). Thus our core problem becomes looking for solutions of the optimization problem
		\begin{equation}\label{eq: max info intro}
			\sup_{q\in\mathcal{Q}_\alpha}\; \mathcal{I}_{\theta_0}\bigl(q\circ\mathcal{P}\bigr).
		\end{equation}
		Leveraging our factorization result, we reduce the Fisher‐information maximization in \eqref{eq: max info intro} to searching over extremal (or staircase) mechanisms.  This additional structure lets us show that the optimum admits the closed form
		\[
		\max_{q\in\mathcal{Q}_\alpha}\mathcal{I}_{\theta_0}(q\circ\mathcal{P})
		=(e^\alpha -1)^2\,M^*,
		\]
		where \(M^*\) is the solution to the linear program
		\begin{equation*}
			M^*=\max_{w\in\mathbb{R}_+^{2^d}}w^\top i
			\quad\text{s.t.}\quad R\,w=\mathbf1.
	\end{equation*}}
	Here \(i\) is explicit and given {\revar in \eqref{eq : def i beta}} and \(R\) is the {\revar \(d\times2^d\)} staircase‐pattern matrix (see Definition~\ref{def: staircase matrix}). This constrained optimization problem parallels the privacy–utility trade‐off of Kairouz et al.\ \cite{kairouzExtremalMechanismsLocal2015} (see their Eq.~(8) and Theorem 4). 
	
	By establishing matching upper and lower bounds (Theorem~\ref{th : info cas fini}), we not only characterize \(M^*\) but also explicitly construct a staircase mechanism that attains it.
	
	{\revarn  Then,} to address non discrete spaces,
	we introduce a continuous analogue of the staircase mechanism (Section~\ref{ss: infinite}). 

	It is worth noting that  
	\cite{KalSte} analyzes the Fisher information in a continuous space \(\mathcal{X}\). However, his solution to the (continuous) optimization problem is obtained by approximating it through a discretization procedure, whereas our approach provides an exact continuous solution, without the need for any discrete approximation.
	
	Extending our findings to the continuous case is a highly challenging task from a purely theoretical and mathematical perspective. It requires several tools from functional and convex analysis, 
	which make the proofs rather involved and interdisciplinary. 
	The main tools of functional analysis used in the paper are recalled in Section \ref{S: tools} of the Appendix.

	The key idea is as follows. In the finite case, proving the factorization result relies on the fact that every point in the hyperrectangle \([1, e^\alpha]^d\), where \(d = |\mathcal{X}|\), can be expressed as a convex combination of its extreme points \(\{1, e^\alpha\}^d\). This property is guaranteed by Carathéodory’s theorem, which applies because the hyperrectangle is a compact convex subset of \(\mathbb{R}^d\).
	
	When extending this result to the continuous case, the hyperrectangle is replaced by the set
	\[
	\theconvex := \left\{ v : \mathcal{X} \rightarrow \mathbb{R} \ \text{measurable} \ \middle| \ 1 \leq v(x) \leq e^\alpha \ \text{a.e.} \right\} \subset L^\infty(\mathcal{X}, \mathbb{R}).
	\]
	The extremal points of the finite-dimensional hyperrectangle translate, in the continuous setting, to the set
	\[
	\mathcal{E} := \left\{ v : \mathcal{X} \rightarrow \mathbb{R} \ \text{measurable} \ \middle| \ v(x) \in \{1, e^\alpha\} \ \text{a.e.} \right\}.
	\]
	In this framework, Carathéodory’s theorem is replaced by Choquet’s theorem. To apply it, we equip \(\theconvex\) with the distance \(d_\star\), which metrizes the weak-\(\star\) topology, and we prove that \(\theconvex\) is convex and compact with respect to this topology.
	
	Another significant challenge arises from the fact that \(\theconvex \subset L^\infty(\mathcal{X}, \mathbb{R})\). For our results, it is crucial to evaluate elements of \(\theconvex\) at specific points \(x \in \mathcal{X}\), which leads to the introduction of the evaluation operator. The key difficulty is that this operator is not continuous with respect to the weak-\(\star\) topology. Providing a rigorous interpretation of pointwise evaluation of integrals defined over the domain \(\mathcal{E}\) is therefore particularly delicate, as detailed in Subsection~\ref{ss:point eval} of {\revar the} Appendix. The solution to this problem is provided in Proposition~\ref{p: evaluation}, whose proof is quite technical and relies on the introduction of Bochner integrals and on proving the separability of \(\theconvex\) when endowed with the weak-\(\star\) topology.
	
	These results allow us to extend the proof of the factorization result to the continuous setting. Consequently, we are able to establish, even in the continuous 
	case, both upper and lower bounds for the Fisher information and these bounds coincide as \(\alpha \to 0\), which corresponds to the regime with a significant level of privacy. More precisely, we prove in Theorem \ref{thm : equivalent alpha petit ctn} that the best Fisher information under $\alpha$-LDP constraint in a regular model 
	$(p_\theta)_{\revarna \theta \in \mathbb{R}}$ is equivalent to $\frac{\alpha^2}{4} \left( \int \abs{\dot{p}_{\theta}(x)} dx \right)^2$ as $\alpha\to0$. 
	{\revar This result precises how the $L^1$ information, introduced in \cite{DucRua24} governs the statistical estimation difficulty under the $\alpha$-LDP constraint {\revarna in the regime $\alpha\to0$. It also provides an explicit description of the optimal mechanism.}}
	
	More generally, we show that, without imposing any restriction on \(\alpha\), the maximization problem introduced in \eqref{eq: max info intro} admits a solution. Specifically, there exists a Radon sub-probability measure \(\bar{\mu}\) on \(\mathcal{E}\) such that the corresponding (continuous version of the) staircase privacy mechanism \(q^{(\bar{\mu})}\) satisfies
	\[
	\mathcal{I}_{\theta}(q^{(\bar{\mu})} \circ \mathcal{P}) = \max_{q \in \mathcal{Q}_\alpha} \mathcal{I}_{\theta}(q \circ \mathcal{P}),
	\]
	see Theorem~\ref{thm: main opt cont} for details.
	
	It is important to emphasize that the proof of this result is far from trivial. Although the compactness of \(\theconvex\) in the weak-\(\star\) topology guarantees the existence of a solution in \(\theconvex\), the fact that \(\mathcal{E}\) is not closed prevents us from directly concluding that the solution is supported on \(\mathcal{E}\). Instead, the proof relies on more advanced techniques, combining convex analysis and set-theory arguments, 
	to establish the existence of an extremal solution supported on \(\mathcal{E}\).

	It is also important to highlight that our theoretical framework provides valuable intuition about the structure of optimal mechanisms, which enables us to effectively apply our results even in models that are not regular. This is precisely the focus of Section~\ref{s: unif}, where we consider the case of a private variable following a uniform distribution on the interval \({\revar (}0, \theta]\), with unknown \(\theta\). In this scenario, we naturally operate in the continuous setting, since \(\mathcal{X} = {\revar(}0, \infty)\). The factorization result continues to apply in this context, allowing us to establish an upper bound on the Fisher information that closely resembles the one obtained for regular models. What is quite intriguing is that, even if we start from a model that is not regular, once privacy is enforced we effectively recover a model that becomes regular again (see Proposition \ref{prop : Fisher unif}).

	In Section~\ref{Ss : efficient est unif}, we further assume the availability of a preliminary estimator of the parameter, which does not need to be particularly accurate or close to the true value (note that two-stage estimators are commonly employed in this setting, as illustrated, for instance, in \cite{KalSte} and \cite{steinbergerEfficiencyLocalDifferent}). Guided by the intuition developed through our earlier theoretical results, we introduce an \(\alpha\)-LDP extremal mechanism corresponding to a measure supported on two points in \(\mathcal{E}\). We show that, when the preliminary estimate is smaller than the true parameter value \(\theta_0\), the resulting estimator is consistent, asymptotically Gaussian, and, as \(\alpha \to 0\), its variance matches the previously derived upper bound on the Fisher information. This establishes the efficiency of the proposed estimator in this privacy regime.
	We {\revar pursue} our analysis with numerical experiments to assess the practical performance of the estimation procedure. The simulations confirm our theoretical findings: when the preliminary estimate exceeds \(\theta_0\), the estimator is significantly biased, while when the preliminary estimate is below \(\theta_0\), the estimator performs well and aligns with the theoretical predictions. {\revis Although our primary focus is establishing exact theoretical foundations, we conclude by mapping out concrete pathways toward practical implementation. We discuss tractable symmetric approximations in high-privacy regimes and establish a semi-infinite dual formulation to lay the groundwork for numerical computation.}

\vspace{0.3 cm}

{\revis
	The structure of the paper is as follows. In Section~\ref{s: pb formulation and main}, we rigorously introduce the notion of $\alpha$-LDP and formulate the estimation problem. Our main theoretical results are presented in Section~\ref{ss: finite}, which treats the finite setting, and Section~\ref{ss: infinite}, which extends the analysis to continuous sample spaces. Section~\ref{s: Example of app} presents applications to concrete statistical models, including a regular parametric regression model and the non-regular uniform model. Section~\ref{S : implementation pathways} discusses pathways towards the practical implementation and numerical approximation of the optimal mechanisms. The functional-analytic tools required for the continuous framework, the proofs of the main results, and further technical arguments are provided in the Appendices \ref{S: tools}--\ref{s: proof technical}. 
}

\section{Problem formulation}{\label{s: pb formulation and main}}
We begin by describing the mathematical formulation of the notion of privacy that we will work with. The process of privatizing raw data and transforming it into public data is modeled through a conditional distribution, known as a privacy mechanism or channel distribution. Let us now formalize this setup.

Let $\mathcal{X}$ and $\mathcal{Z}$ be two separable, complete metric spaces. Equipped with their respective Borel $\sigma$-algebras, they define the measurable spaces $(\mathcal{X}, {\revarn \Borelof{\cX}})$ and $(\mathcal{Z}, {\revarn \Borelof{\cZ}})$. In particular, $(\mathcal{X}, {\revarn \Borelof{\cX}})$ represents the space of sensitive (raw) data, while $(\mathcal{Z}, {\revarn \Borelof{\cZ}})$ corresponds to the space of privatized (public) data.

Let us formally describe the mechanism that transforms raw data into its public counterpart. 
	{\revarn {\revarna We} consider the so-called non-interactive mechanisms, where the privatization mechanism release a public data $Z_i$ using only the value of $X_i$.   
		Formally, for each $i \in \{1, \dots, n\}$, the privatized output is given by
		\begin{equation}{\label{eq: drawn Z nonint}}
			Z_i \sim Q(\cdot \mid X_i = x),   
		\end{equation}
		where $Q : \Borelof{\cZ} \times \mathcal{X} \rightarrow [0,1]$ is some given Markov kernel.
		
		{\revis Interactive mechanisms can incorporate additional context about the private dataset into the public observations $Z_i$, for instance by integrating preliminary estimates into the randomization process. Consequently, they can yield enhanced statistical performance compared to non-interactive approaches (see, e.g.,~\cite{ButKlo25, Kro24}). Nevertheless, characterizing optimality in the non-interactive regime remains fundamental. In this respect, \cite{steinbergerEfficiencyLocalDifferent} demonstrated that restricting attention to non-interactive mechanisms is sufficient to maximize Fisher information. We therefore focus on non-interactive mechanisms in this work.}
	}

	Later, we will search for privacy mechanisms that are optimal in a sense that will be extensively detailed. 
	To study such notions of optimality, it is crucial to quantify the level of privacy provided by a mechanism, allowing for meaningful comparisons between different mechanisms that guarantee the same degree of privacy.
	
	Privacy is measured using the notion of local differential privacy. Given a parameter $\alpha \ge 0$, we say that a random variable $Z_i$ is an $\alpha$-locally differentially private version of $X_i$ if,
	{\revarn
		\begin{equation}
			\label{Eq : formalism LDP usuel NI}
			\forall(x,x') \in \cX^2, \quad \sup_{A \in \Borelof{\cZ}} \frac{Q(A \mid X = x)}{Q(A \mid X = x')} \le \exp(\alpha).
	\end{equation}}
	%
	%
	%
	%
	As mentioned earlier, the parameter $\alpha$ quantifies the strength of privacy: the smaller the value, the harder it is to infer sensitive information from the released data. In particular, $\alpha = 0$ corresponds to perfect privacy, while letting $\alpha \to \infty$ gradually removes the privacy constraint.
			Under the local differential privacy condition, the kernels $Q(\cdot \mid X = x)$ are mutually absolutely continuous for all $x \in \mathcal{X}$. This allows us to assume the existence of a dominating measure {\revarn $\vartheta$} on $(\mathcal{Z}, {\revarn \Borelof{\cZ}})$ with respect to which each kernel $Q$ admits a density, denoted $q_x(z)$. The $\alpha$-local differential privacy constraint can then be rewritten in terms of these densities as
			{\revarn 
				\begin{equation}{\label{eq: def LDP}}
					\forall(x,x')\in\cX^2, \quad \frac{q_x(z)}{q_{x'}(z)} \le \exp(\alpha)\quad \text{for $\vartheta(dz)$-almost every $z$.}
				\end{equation}
				We refer to Lemma B3 in \cite{11CLDP} and to Lemma \ref{eq : extension of BRS and measurable} in the Appendix \ref{ss: factorization Polish Z} for more precise properties on the Radon-Nikodym derivatives $q_x(z)$. It is for instance possible to choose $q_x(z)>0$ 
				for all $(x,z)\in\cX\times \cZ$ and $\vartheta$ as a probability.
				From now on, we will write $q_x(z)$ to denote the density $q(z \mid X = x)$ with respect to $\vartheta$.}
			
			With this background in place, we can now present our main results, which crucially depend on whether the spaces under consideration are discrete or continuous. We begin with the finite setting in Section \ref{ss: finite}, as it allows us to fix the notation and highlight the key ideas in a simpler framework. We then extend our analysis in Section \ref{ss: infinite}, focusing on the scenario where the space $\mathcal{X}$ is an open subset of $\mathbb{R}^d$. 

			\section{Main results: Discrete case}{\label{ss: finite}}
			Let us begin by considering the case where both $\mathcal{X}$ and $\mathcal{Z}$ are finite sets. Specifically, we take $\mathcal{X} = \{x_1, \dots, x_d\}$ and $\mathcal{Z} = \{z_1, \dots, z_l\}$.
			{\revis In this discrete case, we assume that $\vartheta$ is the counting measure on $\cZ$. Then, a randomization from $\cX$ to $\cZ$ is a map $q : \cX \to \proba(\cZ)$, where $\proba(\cZ)$ denotes the set of probability measures on $\cZ$. We identify these measures with their densities with respect to $\vartheta$ using the notation $q_x(z) = P(Z=z \mid X=x)$.
			}
			
			In this setting, one of the most studied and widely used privacy mechanisms is the staircase mechanism. A key reference on this topic is \cite{kairouzExtremalMechanismsLocal2015}, where it is shown that, when $\mathcal{X}$ is finite, the optimal mechanism that maximizes utility functions expressible as sums of sublinear functions has an extremal structure.
			
			Following \cite{kairouzExtremalMechanismsLocal2015}, we say that a randomization mechanism is extremal if, for all $z \in \mathcal{Z}$ and all $(x, x') \in \mathcal{X}^2$, the log-likelihood ratios satisfy
			$$
			\left| \ln \frac{q_{x'}(z)}{q_x(z)} \right| \in \{0, \alpha\},
			$$
			which is equivalent to the condition:
			\begin{equation}\label{eq : extremal constrainte discret}
				\forall z \in \mathcal{Z}, \ \forall (x,x') \in \mathcal{X}^2, \quad \frac{q_{x'}(z)}{q_x(z)} \in \{e^{-\alpha}, 1, e^\alpha\}.
			\end{equation}
			{\revarn 
				If} $q_x(z) = 0$ for some $z \in \mathcal{Z}$ and $x \in \mathcal{X}$, this constraint is interpreted as requiring $q_{x'}(z) = 0$ for all $x' \in \mathcal{X}$.
			
			We refer to any mechanism satisfying this constraint as a staircase mechanism. In the context of global differential privacy, a similar definition can be extended to apply to all neighboring database queries $x, x'$ (or, equivalently, pairs differing within a given sensitivity), thus recovering many known optimal mechanisms. Notably, the mechanisms studied in \cite{GenVis12, GenVis13} and \cite{Ghosh12} are particular instances of the staircase mechanism defined in \eqref{eq : extremal constrainte discret}.
			
			There are several reasons why staircase mechanisms deserve particular attention. First, as noted above, they arise naturally in the context of optimal mechanisms and have been extensively studied across different notions of privacy and utility. Second, extremal mechanisms saturate the $\alpha$-LDP constraint. Third, and most importantly for our purposes, one of our main results (see Lemma \ref{lem : factorisation cas fini}) shows that any privacy mechanism can be decomposed into the composition of two mechanisms, one of which is an extremal mechanism.
			
			Given the central role staircase mechanisms play in our analysis, and the frequency with which we will refer to them throughout the paper, we now provide a more detailed description.
			
			\begin{df}\label{def: staircase matrix}
				For $\beta \in \{0, \dots, 2^d - 1\}$, let $d_j(\beta)$ denote the $j$-th component of the $d$-dimensional binary vector corresponding to the binary representation of $\beta$ in the dyadic basis, i.e., $\beta = \sum_{j=0}^{d-1} d_j(\beta) 2^j$, with $d_j(\beta) \in \{0, 1\}$.
				A matrix is called a staircase pattern matrix if its $\beta$-th column is a vector $r_\beta \in \{1, e^\alpha\}^d$ for each $\beta \in \{0, \dots, 2^d - 1\}$, where the $j$-th component of $r_\beta$ is defined as
				$(r_\beta)_j = \indi{\{d_j(\beta) = 0\}} + e^\alpha \, \indi{\{d_j(\beta) = 1\}}.$
				Each column vector $r_\beta$ is called a staircase pattern.
			\end{df}
			
			To clarify the construction, consider the case $d = 2$. There are $2^d = 4$ staircase patterns, and the corresponding staircase pattern matrix is:
			$$
			\begin{bmatrix}
				1 & 1 & e^\alpha & e^\alpha \\
				1 & e^\alpha & 1 & e^\alpha
			\end{bmatrix}.
			$$
			In general, for arbitrary $d$, the staircase patterns are given by:
			\begin{equation}\label{eq : R matrix by slice}
				r_0=\begin{bmatrix} 1\\1\\ 1\\ \vdots\\ 1\end{bmatrix},~
				r_1=\begin{bmatrix} e^\alpha\\1 \\ 1\\ \vdots\\ 1\end{bmatrix},~
				r_2=\begin{bmatrix} 1 \\ e^\alpha\\ 1\\ \vdots\\ 1\end{bmatrix},~\dots~,~
				r_{2^d-1}=\begin{bmatrix} e^\alpha \\ e^\alpha\\ e^\alpha\\ \vdots\\ e^\alpha\end{bmatrix}.
			\end{equation}
			Equivalently, the vectors $(r_\beta)_{\beta \in \{0, \dots, 2^d - 1\}}$ are the extremal points of the hyper-rectangle $[1, e^\alpha]^d$.
			We denote by $\mathcal{E} = \{0, \dots, 2^d - 1\}$ the index set of these extremal points ${r_0, \dots, r_{2^d - 1}}$.
			Interpreting each $r_\beta$ as a function on $\mathcal{X}$, we write $r_\beta(x_j) = (r_\beta)_j$ for $x_j \in \mathcal{X}= \{ x_1, ... , x_d \}$. 
			
			For each $\beta \in \{0, \dots, 2^d - 1\}$, we define the level sets:
			\begin{equation}\label{eq: Fbeta+ 4.8}
				F^+_\beta = \{ x \in \mathcal{X} \mid r_\beta(x) = e^\alpha \}, \quad
				F^-_\beta = \{ x \in \mathcal{X} \mid r_\beta(x) = 1 \}.
			\end{equation}
			This implies that for all $x \in \mathcal{X}$, the function $r_\beta$ satisfies
			\begin{equation}{\label{eq: r beta x}}
				r_\beta(x) = e^\alpha \indi{F^+_\beta}(x) + \indi{F^-_\beta}(x) = 1 + (e^\alpha - 1) \indi{F^+_\beta}(x),
			\end{equation}
			i.e., the function $r_\beta$ takes values in $\{1, e^\alpha\}$, with the structure fully determined by $F^+_\beta$.
			
			Note that as $\beta$ ranges from $0$ to $2^d - 1$, the sets $F^+_\beta$ range over all possible subsets of $\mathcal{X}$. For instance, $F^+_0 = \emptyset$ because $r_0 = (1, \dots, 1)^T$ by \eqref{eq : R matrix by slice}, and $F^+_{2^d - 1} = \mathcal{X}$ since $r_{2^d - 1} = (e^\alpha, \dots, e^\alpha)^T$.

			\subsection{A factorization result}
			
			We are now ready to present our first main result: the factorization lemma, which shows that any privacy mechanism can be decomposed into the composition of two mechanisms, one of which is extremal in the sense of Definition \ref{def: staircase matrix}.
			
			In order to do so, let us introduce some notation.
			{\revis  Let $q^{(1)}: \cX \to \proba(\cE)$ and $q^{(2)}: \cE \to \proba(\cZ)$ be privacy mechanisms between finite sets. We denote their composition by $q = q^{(2)} \circ q^{(1)}$, which represents the sequential application of the two randomization steps. The resulting mechanism is a randomization from $\cX$ to $\cZ$ satisfying:}
			$$
			q_x(z) = \sum_{y \in \mathcal{E}} q^{(1)}_x(y) \, q^{(2)}_y(z).
			$$
			\begin{lem}\label{lem : factorisation cas fini}
				Let $q$ be an $\alpha$-LDP mechanism between the finite spaces $\cX=\{x_1,\dots,x_d\}$ and $\cZ=\{z_1,\dots,z_l\}$. 
				There exist $q^{(1)}$ an extremal mechanism from $\cX$ to $\cE$, and $q^{(2)}$ a randomization mechanism for $\cE$ to $\cZ$ such that $q=q^{(2)}\circ q^{(1)}$.
				
				Moreover, we have the representation $q^{(1)}_x(\beta)= \omega_\beta r_\beta(x)$, where
				{\revarna $(\omega_\beta)_{\beta\in\{0,\dots,2^d-1\}}$} defines a sub-probability measure on $\cE$, satisfying $e^{-\alpha} \le \sum_{\beta=0}^{2^d-1} \omega_{\beta}\le 1$.
			\end{lem}
			\begin{proof}
				We define the minimal mass function
				$$
				\underline{q}(z) = \min_{x \in \cX} q_x(z), {\revarna\quad \forall z\in\cZ.}
				$$
				Note that, from the definition of $\alpha$-LDP in \eqref{eq: def LDP} and the following remark, it follows that if $\underline{q}(z) = 0$, then $q_x(z) = 0$ for all $x \in \cX$.
				
				Let us now introduce a normalized version of the mechanism $q$. For each $(x, z) \in \cX \times \cZ$, define
				$$
				u_x(z) = 
				\begin{cases}
					\displaystyle \frac{q_x(z)}{\underline{q}(z)}, & \text{if } \underline{q}(z) > 0, \\
					1, & \text{if } \underline{q}(z) = 0.
				\end{cases}
				$$
				By construction and by the $\alpha$-LDP constraint in \eqref{eq: def LDP}, we have that for all $z \in \cZ$, the vector $u(z) := (u_x(z))_{x \in \cX} \in [1, e^\alpha]^d$, where we recall that $d = |\cX|$.
				
				The hyperrectangle $[1, e^\alpha]^d \subset \mathbb{R}^d$ is compact and convex, so by the Carathéodory theorem, the vector $u(z)$ belongs to the convex hull of the extremal points of $[1, e^\alpha]^d$. Therefore, for all $z \in \cZ$, we may write
				$$
				u(z) = \sum_{\beta=0}^{2^d - 1} c_{z, \beta} r_\beta,
				$$
				where $c_{z, \beta} \ge 0$, $\sum_{\beta=0}^{2^d - 1} c_{z, \beta} = 1$, and the $r_\beta$'s are the staircase pattern vectors defined in \eqref{eq : R matrix by slice}. Accordingly, for all $x \in \cX$,
				$$
				u_x(z) = \sum_{\beta=0}^{2^d - 1} c_{z, \beta} r_\beta(x),
				$$
				and thus
				\begin{equation} \label{eq : qxz convex hull}
					q_x(z) = \underline{q}(z) u_x(z) = \sum_{\beta=0}^{2^d - 1} \underline{q}(z) c_{z, \beta} r_\beta(x).
				\end{equation}
				Define, for each $\beta \in \mathcal{E} = \{0, \dots, 2^d - 1\}$,
				\begin{equation} \label{eq: def omega beta 4.5}
					\omega_\beta := \sum_{z \in \cZ} \underline{q}(z) c_{z, \beta} \ge 0.
				\end{equation}
				Since $q_x(z)$ is a probability distribution over $\cZ$, we have $\sum_{z \in \cZ} q_x(z) = 1$, which yields
				\begin{equation} \label{eq: constraint factorisation 4.55}
					1 = \sum_{\beta=0}^{2^d - 1} \sum_{z \in \cZ} \underline{q}(z) c_{z, \beta} r_\beta(x) = \sum_{\beta=0}^{2^d - 1} \omega_\beta r_\beta(x).
				\end{equation}
				This suggests defining the first-stage mechanism $q^{(1)} : \cX \to {\revarna \proba\left(\cE\right)}$ as
				\begin{equation} \label{eq: cond q1 4.6}
					q^{(1)}_x(\beta) := \omega_\beta r_\beta(x),
				\end{equation}
				which forms a probability distribution over $\cE = \{0, \dots, 2^d - 1\}$. Moreover, it satisfies the extremality constraint in \eqref{eq : extremal constrainte discret}, since
				$$
				\frac{q^{(1)}_x(\beta)}{q^{(1)}_{x'}(\beta)} = \frac{r_\beta(x)}{r_\beta(x')} \in \{e^{-\alpha}, 1, e^\alpha\},
				$$
				using the fact that each $r_\beta$ takes values in $\{1, e^\alpha\}$.
				Now fix any $x \in \cX$. Since $r_\beta(x) \in \{1, e^\alpha\}$, we obtain the inequalities
				\begin{equation}{\label{eq: mass omega beta}}
					1 = \sum_{\beta} \omega_\beta r_\beta(x) \le \sum_\beta \omega_\beta e^\alpha, \quad \text{and} \quad 1 = \sum_\beta \omega_\beta r_\beta(x) \ge \sum_\beta \omega_\beta,   
				\end{equation}
				implying that the total mass $\sum_\beta \omega_\beta$ lies in $[e^{-\alpha}, 1]$. Thus, $\omega$ is a sub-probability measure.
				
				We now define the second-stage mechanism $q^{(2)} : \cE \to {\revarna \proba\left(\cZ\right)}$. For all $\beta$ such that $\omega_\beta > 0$, set
				$$
				q^{(2)}_\beta(z) := \frac{\underline{q}(z) c_{z, \beta}}{\omega_\beta},
				$$
				which defines a probability distribution on $\cZ$, due to the normalization in \eqref{eq: def omega beta 4.5}. If instead $\omega_\beta = 0$, we may define $q^{(2)}_\beta(z)$ arbitrarily as any probability distribution on $\cZ$ — for example, set
				$$
				q^{(2)}_\beta(z) := \frac{1}{|\cZ|} \quad \text{for all } z \in \cZ.
				$$
				This definition is consistent, since for such $\beta$, $\underline{q}(z) c_{z, \beta} = 0$ for all $z$.
				
				In all cases, we have the identity
				\begin{equation} \label{eq: 4.75}
					\omega_\beta q^{(2)}_\beta(z) = \underline{q}(z) c_{z, \beta}, \quad \text{for all } (\beta, z) \in \cE \times \cZ.
				\end{equation}
				Substituting \eqref{eq : qxz convex hull} and \eqref{eq: 4.75} into the expression for $q_x(z)$, and using \eqref{eq: cond q1 4.6}, we obtain the desired factorization:
				$$
				q_x(z) = \sum_{\beta=0}^{2^d - 1} \omega_\beta q^{(2)}_\beta(z) r_\beta(x) = \sum_{\beta=0}^{2^d - 1} q^{(1)}_x(\beta) q^{(2)}_\beta(z),
				$$
				which shows that $q = q^{(2)} \circ q^{(1)}$.
			\end{proof}
			
			\begin{rem}
				The factorization result above shows that any $\alpha$-LDP mechanism can be decomposed into an extremal $\alpha$-LDP mechanism, followed by an additional randomization step. Since any such randomization further reduces the statistical information available from the data—without contributing to the satisfaction of the $\alpha$-LDP constraint— applying a second randomization is, from a statistical perspective, detrimental. The first (extremal) step alone suffices to ensure privacy, so only extremal mechanisms need to be considered in theory.
				
				Building on this insight, we will derive an expression for the maximal Fisher information achievable by any $\alpha$-LDP mechanism, under suitable regularity assumptions on a parametric model.
			\end{rem}
			
			\begin{rem} \label{rem : utility Kairouz} {\revar
					More generally, the factorization result is useful if the statistical utility of a randomization $q$ can be written as
						$		U(q)=\sum_{z\in\cZ} u\left( (q_x(z))_{x\in\cX} \right)$,
					where $u : \mathbb{R}^\cX \to \mathbb{R}$ is sublinear (see Definition 1 in \cite{kairouzExtremalMechanismsLocal2015}). {\revarna In this case, one} can verify that $U(q^{(2)}\circ q^{(1)}) \le U(q^{(1)})$ and, consequently, it is sufficient to maximize the utility over extremal mechanisms.
					The Fisher information is a special case, which we treat in details in the next section. Other examples of such utilities, including the mutual information between the public and private data, can be found in \cite{kairouzExtremalMechanismsLocal2015}. }
			\end{rem}
			
			\begin{rem}
				The factorization result \ref{lem : factorisation cas fini} applies to the case where the input space $\mathcal{X}$ is finite. In this setting, the extremal mechanism takes values in the set $\mathcal{E}$, whose cardinality is $2^{\Card(\mathcal{X})}$—significantly larger than $\Card(\mathcal{X})$ itself. However, the actual support of the randomization $q^{(1)}$ is determined by the support of the sub-probability measure $\omega$, which can be much smaller than the full set $\mathcal{E}$ (see Point (a) of Theorem 2 in \cite{kairouzExtremalMechanismsLocal2015}).
				
				One of the main findings of this paper consists in extending Lemma \ref{lem : factorisation cas fini} to the case where the input space $\mathcal{X}$ is continuous, for example $\mathcal{X} = \mathbb{R}$. As already mentioned in the introduction and as we will see in Section \ref{ss: infinite}, in such a setting the set of extremal points $\mathcal{E}$ becomes infinite-dimensional.
			\end{rem}

			\subsection{Parametric models and Fisher information}\label{ss:Parametric models and Fisher information}
			In this subsection, we explain how the factorization result can be used to derive an expression for the maximal Fisher information achievable by any privacy mechanism satisfying the $\alpha$-LDP constraint, in the case where the data space $\mathcal{X}$ is discrete. To do so, we first introduce some background.
			
			Consider i.i.d. data $X_1, \ldots, X_n$ drawn from a distribution $P_\theta$ on the finite space $\mathcal{X} = \{x_1, \ldots, x_d\}$, where $P_\theta$ belongs to a parametric model $\mathcal{P} = (P_\theta)_{\theta \in \Theta}$. {\revarna We assume that $\Theta \subset \mathbb{R}$ and that the model is regular in a sense to be better specified later (see in particular Definition \ref{def: DQM}). We denote the corresponding Fisher information by
				$\mathcal{I}_\theta(\mathcal{P}) \in \mathbb{R}$.}
			We write $p_\theta(x) = P_\theta(X = x)$ for the probability mass function.
			
			If data must be protected through local differential privacy, the observed variables $Z_1, \ldots, Z_n$ are sanitized via a Markov kernel, as described in Section \ref{s: pb formulation and main}. Let $q : \mathcal{X} \to {\revarna \proba\left(\mathcal{Z}\right)}$ be such an $\alpha$-LDP mechanism with output space $\mathcal{Z} = \{z_1, \ldots, z_l\}$. The distribution of the public variable $Z$ is then given by $\widetilde{p}_\theta(z) = P_\theta(Z = z) = \int_{\mathcal{X}} p_\theta(x) q_x(z) {\nu_\cX(dx)}$,
			{\revar where $\nu_\cX$ denotes the counting measure on $\mathcal{X}$.}

			If the map $\theta \mapsto p_\theta(x)$ is continuously differentiable for all $x \in \mathcal{X}$, then $\theta \mapsto \widetilde{p}_\theta(z)$ is also continuously differentiable for all $z \in \mathcal{Z}$. We will recall a more precise result on the preservation of regularity under privacy constraints, as established in \cite{steinbergerEfficiencyLocalDifferent} (see \eqref{eq : score public} below). Thus, the resulting statistical model under the privacy mechanism $q$ remains regular and has Fisher information denoted by $\mathcal{I}_\theta(q \circ \mathcal{P})$.
			
			Under standard regularity assumptions, the maximum likelihood estimator in the privatized model $q \circ \mathcal{P}$ is asymptotically normal with distribution $N(0, \mathcal{I}_\theta(q \circ \mathcal{P})^{-1})$. It is also efficient in the sense that no regular estimator in $q \circ \mathcal{P}$ can have smaller asymptotic variance. However, since the mechanism $q$ is freely chosen subject to the $\alpha$-LDP constraint \eqref{eq: def LDP}, it is natural to ask which mechanism maximizes the Fisher information. This leads to the optimization problem
			\begin{equation}{\label{eq: opt Fisher}}
				\sup_{q \in \mathcal{Q}_\alpha} \mathcal{I}_\theta(q \circ \mathcal{P}),
			\end{equation}
			where we recall that $\mathcal{Q}_\alpha$ denotes the set of $\alpha$-LDP channels from $\mathcal{X}$ to a finite set $\mathcal{Z}$.
			This optimization problem is challenging and has been thoroughly investigated in the elegant work \cite{steinbergerEfficiencyLocalDifferent}, which develops a general theory of statistical efficiency under local differential privacy.
			
			The factorization result proven earlier allows us to revisit the problem from a novel perspective. It implies that we can restrict attention to the class of extremal (staircase) mechanisms, significantly reducing the search space. This additional structure enables us to derive an explicit expression for the maximal Fisher information under $\alpha$-LDP constraints, as presented in Theorem \ref{th : info cas fini}.
			
			To formalize the notion of regularity and rigorously study asymptotic efficiency, we adopt the classical framework of differentiability in quadratic mean (DQM) at an interior point $\theta_0 \in \overset{\circ}{\Theta}$.
			
			\begin{df}\label{def: DQM}
				A statistical model $\mathcal{P} = (P_\theta)_{\theta \in \Theta}$, with 
				{\revar $\Theta \subset \R$,}
				sample space $(\mathcal{X}, \mathcal{F})$, and dominating measure {\revar $\nu$,} is said to be \textit{differentiable in quadratic mean (DQM) at $\theta_0 \in \overset{\circ}{\Theta}$} if the $\mu$-densities $p_\theta = \frac{dP_\theta}{d{\revar\nu}}$ satisfy, for $h \in {\revar \mathbb{R}}$,
				$$
				\int_\cX \left(\sqrt{p_{\theta_0+h}(x)} - \sqrt{p_{\theta_0}(x)} - {\revar \frac{1}{2} h} s_{\theta_0}(x) \sqrt{p_{\theta_0}(x)} \right)^2 {\revar\nu}(dx) = o(\|h\|^2)
				\quad \text{as } h \to 0,
				$$
				for some measurable 
				{\revar function $s_{\theta_0} : \cX \to \mathbb{R}$.} This function is called the score at $\theta_0$.
			\end{df}
			
			This definition is independent of the choice of the dominating measure  ${\revar\nu}$. To see this, consider two dominating measures ${\revar\nu_1}$ and ${\revar\nu_2}$, and note that both corresponding integrals can be written with respect to the measure ${\revar\nu = \nu_1 + \nu_2}$.
			
			We say that a model is DQM (or simply regular) if it is DQM at every point $\theta \in \overset{\circ}{\Theta}$. Let us also recall some properties of the score function, which will be crucial in the sequel. According to Theorem 7.2 in \cite{Van07}, if the model $\mathcal{P}$ is DQM at $\theta_0$, then the score satisfies 
			\begin{equation}{\label{eq: score centered}}
				\mathbb{E}_{\theta_0}[{\revar s_{\theta_0}(X)}] = 0,
			\end{equation}
			and the Fisher information
			$$
			\mathcal{I}_{\theta_0}(\mathcal{P}) = 
			{\revar \mathbb{E}_{\theta_0}[s_{\theta_0}(X)^2]}
			$$
			exists and is finite. Moreover, in a $\mathcal{C}^1$ model, the score function is explicitly given by
			$$
			s_{\theta_0}(x) = \frac{{\revar \partial_\theta} p_{\theta_0}(x)}{p_{\theta_0}(x)}.
			$$
			Regarding transformations of DQM models, Proposition A.5 in \cite{Bic93} states that if $\mathcal{P} = (P_\theta)_{\theta \in \Theta}$ is DQM and $T: \mathcal{X} \to \mathcal{Z}$ is measurable, then the transformed model $T\mathcal{P} = (P_\theta \circ T^{-1})_{\theta \in \Theta}$ is also DQM, with score function
			$$
			t_{\theta_0}(z) = \mathbb{E}[s_{\theta_0}(X) \mid Z = z],
			$$
			where $(X, Z) \sim (X, T(X))$ under $P_{\theta_0}$.
			
			For local differential privacy, one must consider randomized transformations—i.e., Markov kernels. An extension of the result in \cite{Bic93} to this context is provided in Lemma 3.1 of \cite{steinbergerEfficiencyLocalDifferent}. It shows that if the original model $(P_\theta)_\theta$ is DQM at $\theta_0$, then the privatized model $(q \circ P_\theta)_\theta$ is also DQM at $\theta_0$, with score function
			\begin{equation} \label{eq : score public}
				t_{\theta_0}(z) = \mathbb{E}\left[s_{\theta_0}(X) \mid Z = z\right] =
				\frac{\int_\cX s_{\theta_0}(x) q_x(z) p_{\theta_0}(x) {\revarna \nu_\cX}(dx)}{\int_\cX q_x(z) p_{\theta_0}(x) {\revarna \nu_\cX}(dx)}.
			\end{equation}
			
			In what follows, given a parametric model $\mathcal{P} = (P_\theta)_\theta$ on $\mathcal{X}$, we denote by $q \circ \mathcal{P} = (\tilde{p}_\theta)_\theta$ the model induced on $\mathcal{Z}$ via the channel $q$. We denote by
			$$
			\mathcal{I}_{\theta_0}(\mathcal{P}) = \mathbb{E}\left[ {\revar s_{\theta_0}(X)^2} \right]
			$$
			the Fisher information of the private variable $X$, and by
			$$
			\mathcal{I}_{\theta_0}(q \circ \mathcal{P}) = \mathbb{E}\left[{\revar t_{\theta_0}(Z)^2} \right]
			%
			$$
			that of the public variable $Z$. {\revis Here, `private'' refers to the original data $X$, which are not released, whereas `public'' refers to the privatized data $Z$ produced by the mechanism and made available to the statistician.
			}
			
			From Lemma 3.1 in \cite{steinbergerEfficiencyLocalDifferent}, we know that the Fisher information of the public variable is always smaller than or equal to that of the private one. More precisely,
			\begin{equation} \label{eq : decroissance Fisher}
				\mathcal{I}_{\theta_0}(q \circ \mathcal{P}) = \mathcal{I}_{\theta_0}(\mathcal{P}) -
				\mathbb{E}\left[{\revar \left(s_{\theta_0}(X) - t_{\theta_0}(Z)\right)^2} \right].
			\end{equation}
			
			The factorization result provides a valuable structural insight into any randomized channel $q$. Specifically, it guarantees the existence of a decomposition of the form $q = q^{(2)} \circ q^{(1)}$, where $q^{(1)}$ is an extremal $\alpha$-LDP mechanism. This yields the chain of inequalities:
			\begin{equation}{\label{eq: fact on Fisher}}
				\mathcal{I}_{\theta_0}\big((\widetilde{p}_\theta)_\theta\big)
				= \mathcal{I}_{\theta_0}(q \circ \mathcal{P}) 
				= \mathcal{I}_{\theta_0}((q^{(2)} \circ q^{(1)}) \circ \mathcal{P}) 
				= \mathcal{I}_{\theta_0}(q^{(2)} \circ (q^{(1)} \circ \mathcal{P})) 
				\le \mathcal{I}_{\theta_0}(q^{(1)} \circ \mathcal{P}). 
			\end{equation}
			Therefore, to study the optimization problem in \eqref{eq: opt Fisher}, it suffices to restrict our attention to extremal mechanisms of the form $q = q^{(1)}$.
			
			Moreover, the fact that \( q^{(1)} \) is an extremal staircase mechanism introduces additional structure that enables us to reformulate the optimization problem \eqref{eq: opt Fisher}. In the proof of Proposition \ref{prop : borne sup Fisher fini}, we show that if the parameter $\theta$ is one dimensional, the optimization problem \eqref{eq: opt Fisher} is the same as
			\[
			\max_{q \in \mathcal{Q}_\alpha} \, \mathcal{I}_{\theta_0}(q \circ \mathcal{P}) = (e^\alpha - 1)^2 M^\star,
			\]
			where \( M^\star \) denotes the optimal value of the following linear program:
			\begin{equation} \label{eq: opt 14.75}
				M^\star = \max_{\omega \in \mathbb{R}_+^{2^d}} \, {\revis \omega}^T i \quad \text{subject to} \quad R {\revis \omega} = \bm{1},
			\end{equation}
			with \( i = (i_\beta)_{\beta \in \mathcal{E}} \) defined by
			\begin{equation} \label{eq : def i beta}
				i_\beta
				:=
				\frac{
					\left(
					\displaystyle\int_{{ \revis F_\beta^+}}
					s_\theta(x)p_\theta(x)\,\nu_\cX(dx)
					\right)^2
				}{
					1+(e^\alpha-1)
					\displaystyle\int_{F_\beta^+}
					p_{\theta_0}(x)\,\nu_\cX(dx)
				},
			\end{equation}
			with $F_\beta^+$ as in \eqref{eq: Fbeta+ 4.8}. The expression of $i_\beta$ has been derived explicitly in our proof, starting from the use of staircase mechanisms. Here, \( R \) is the staircase pattern matrix introduced in Definition~\ref{def: staircase matrix}.
			
			{\revis The linear program~\eqref{eq: opt 14.75} is the same as the privacy-utility optimization problem studied by Kairouz et al.~\cite{kairouzExtremalMechanismsLocal2015}, see their Equation~(8) and Theorem~4, specialized to the Fisher information criterion considered here.}
			
			
			This leads directly to the statement of the next theorem. 
			
			\begin{thm}\label{th : info cas fini}
				Assume that $\Theta \subset \R$ and that $\cP=(p_{\theta})_\theta$ 
				is DQM at $\theta_0 \in \overset{\circ}{\Theta}$ with score function $s_{\theta_0}$.
				Assume that $s_{\theta_0}(x)\neq 0$ for all $x\in \cX$.
				Then, there exists $\overline{\alpha}$ such that if $\alpha \le \overline{\alpha}$,
				\begin{equation*}
					\max_{q \in \cQ_\alpha}  \cI_{\theta_0}(q \circ \cP) = 	\cI^{\text{max},\alpha}_{\theta_0}((p_\theta)_\theta)		
				\end{equation*}
				where
				\begin{equation}\label{eq : def I max caf fini}
					\cI^{\text{max},\alpha}_{\theta_0}((p_\theta)_\theta)
					:=\frac{(e^\alpha-1)^2}{4}%
					\times \frac{\E\left[\abs{s_{\theta_0}(X)}\right]^2}{\left[ (1-n_\text{max})+e^\alpha n_\text{max}  \right]
						\left[n_\text{max}+(1-n_\text{max})e^\alpha   \right]}
				\end{equation}
				with 
				\begin{equation}{\label{eq: def nmax}}
					n_\text{max}:= \int_{x: s_{\theta_0}(x)>0} p_{\theta_0}(x) {\revarna \nu_\cX}(dx)
					\in (0,1).  
				\end{equation}
				The value of $\overline{\alpha}$ is independent of the channel $q$, but depends on the model $(p_{\theta})_\theta$.
			\end{thm}
			
			The proof of Theorem \ref{th : info cas fini} is provided in Section \ref{s: proof main} of {\revar the} Appendix. It follows directly from Propositions \ref{prop : borne sup Fisher fini} and \ref{prop : borne inf Fisher fini} in Section \ref{s: proof main}, which establish upper and lower bounds, respectively. We emphasize that in the lower bound, we explicitly construct an $\alpha$-LDP mechanism that attains the Fisher information $\cI^{\text{max},\alpha}_{\theta_0}((p_\theta)_\theta)$. Hence, this proposition gives an exact characterization of the maximal Fisher information achievable under the $\alpha$-LDP constraint, at least for small values of $\alpha$.

			\begin{rem}{\label{rk: choice Fmax}}
				The key idea in the proof of the upper bound is to consider
				$$
				F_{\max}^+ := \left\{ x \in \mathcal{X} : s_{\theta_0}(x) > 0 \right\}
				\quad \text{and} \quad
				F_{\max}^{\prime +} := \left\{ x \in \mathcal{X} : s_{\theta_0}(x) < 0 \right\},
				$$
				which form a partition of \(\mathcal{X}\) by the assumption \(s_{\theta_0}(x)\neq0\). For any measurable set \(F_\beta^+\subset\mathcal{X}\),
				$$
				\left( \int_\mathcal{X} s_\theta(x) \indi{F^+_\beta}(x) p_\theta(x) \, \mu(dx) \right)^2
				\le
				\left( \frac{1}{2}\int_\mathcal{X}|s_\theta(x)|p_\theta(x)\,\mu(dx)\right)^2,
				$$
				with equality if and only if \(F_\beta^+=F_{\max}^+\) or \(F_\beta^+=F_{\max}^{\prime+}\). Hence, denoting by \(\beta_{\max}\) and \(\beta_{\max}'\) the corresponding indices in~\eqref{eq: Fbeta+ 4.8}, these two extremal patterns provide natural candidates for the support of an optimal solution to the linear program~\eqref{eq: opt 14.75}. The remainder of the proof shows that the optimum is indeed attained by a sub-probability measure supported on \(\{\beta_{\max},\beta_{\max}'\}\).
			\end{rem}
			
			\begin{rem}
				{\revis It is worth noting that, in the derivation of the upper bound, a preliminary and simple estimate shows that
					$
					\cI_{\theta_0}((\widetilde{p}_\theta)_\theta)
					\le
					\frac{(e^\alpha-1)^2}{4}
					\,\E\left[\abs{s_{\theta_0}(X)}\right]^2,
					$
					and this bound holds even without assuming the additional condition $\alpha\le\bar{\alpha}$. In the exact expression~\eqref{eq : def I max caf fini}, the quantity
					$
					n_{\max}
					=
					\mathbb{P}_{\theta_0}\!\left(s_{\theta_0}(X)>0\right)
					$
					measures the probability mass of the positive-score region, while $1-n_{\max}$ is the mass of the negative-score region. The denominator
					$
					\left[(1-n_{\max})+e^\alpha n_{\max}\right]
					\left[n_{\max}+(1-n_{\max})e^\alpha\right]
					$
					therefore reflects the relative masses of these two regions in the normalization of the optimal extremal mechanism. When $\alpha$ is close to zero, this denominator converges to $1$, so the exact expression in~\eqref{eq : def I max caf fini} coincides asymptotically with the simpler bound above. We will see in the next section that this high privacy behavior extends to more general settings, in particular when the sample space $\mathcal{X}$ is continuous (see Theorem~\ref{thm : equivalent alpha petit ctn}).}
			\end{rem}

			\section{Main results: Extension to the Continuous Case}{\label{ss: infinite}}
			The goal of this section is to extend the results established in Section~\ref{ss: finite} to the setting where \(\mathcal{X}\) is no longer a finite set. Specifically, we now consider \(\mathcal{X} \subset \mathbb{R}^d\) to be an open subset for some \(d \geq 1\), with \(\Borelof{\mathcal{X}}\) denoting the Borel \(\sigma\)-algebra on \(\mathcal{X}\). {\revis The finite-dimensional arguments developed in Section~3 will be extended to an infinite-dimensional setting, where the extremal mechanisms are represented by Radon measures on the corresponding space of extremal functions.
				
				Thus, the proofs in Sections~3 and~4 should be viewed as complementary rather than as two instances of the same argument. Readers primarily interested in the finite setting may restrict their attention to Section~3 and the corresponding finite-case proofs in the Section \ref{Ss: appendix discrete} of the Appendix, whereas Section~4 develops the additional arguments required for the continuous setting.}


			
			
			
			\subsection{ {\revar LDP and extremal mechanism in the continuous case}}\label{ss : extremal mecha cont}
			Note that the problem formulation introduced at the beginning of Section~\ref{s: pb formulation and main} was not restricted to the discrete case, and the notation used here remains fully consistent with that section. 
			In particular, we denote by $\mathcal{Q}_\alpha$ the set of randomization mechanisms from $\mathcal{X}$ to $\mathcal{Z}$
			{\revarn that is satisfying the privacy constraint \eqref{Eq : formalism LDP usuel NI}.}
			
			. {\revarn We assume that $\cZ$ is a locally compact Polish space (see Definition 6.1 in Chapter XI of \cite{dugundjiTopology1966}).
				This space is endowed with its Borel sigma-field $\Borelof{\cZ}$.
				Using Lemma \ref{eq : extension of BRS and measurable} in the Appendix, there exists a finite measure $\vartheta$ on $\cZ$, such that
				%
				the randomization from $\mathcal{X}$ to $\mathcal{Z}$ is given by
				\begin{equation*}
					Q(dz \mid X=x) = q_x(z) \vartheta(dz)
				\end{equation*} 
				where 
				$(x,z) \mapsto q_x(z)$ is a measurable function. 
				Moreover, for $z\in \mathcal{Z}$, we have 
				\begin{equation} \label{eq : alpha LDP weaker}
					e^{-\alpha} \le \frac{q_x(z)}{q_{x'}(z)} \le e^\alpha,~ \text{$dxdx'$-almost everywhere.} 	
				\end{equation} 
			}
			
			The key point now lies in extending the definition of extremal mechanisms to the continuous setting. As in the discrete case, such randomizations will play a central role in the study of optimal privacy mechanisms.
			
			To this end, we introduce notation that enables us to generalize the space of extremal vectors introduced in~\eqref{eq : R matrix by slice} (or equivalently, the staircase mechanisms defined in Definition~\ref{def: staircase matrix}) to the continuous case. From this, we prove the extension of the results from Section~\ref{ss: finite}.
			
			We define the set
			\begin{align*}
				\theconvex = \left\{ v : \mathcal{X} \rightarrow \mathbb{R} \text{ measurable} \, \middle| \, 1 \leq v(x) \leq e^\alpha \text{ a.e.} \right\} \subset L^\infty(\mathcal{X}, \mathbb{R}),
			\end{align*}
			which plays the role of the hyperrectangle $[1, e^\alpha]^d$ in the discrete case (where $d = |\mathcal{X}|$). As the extremal points of $[1, e^\alpha]^d$ were key in the discrete setting, the extremal points of $\theconvex$ will now play an analogous role. This gives rise to the set
			\begin{align*}
				\mathcal{E} = \left\{ r : \mathcal{X} \rightarrow \mathbb{R} \text{ measurable} \, \middle| \, r(x) \in \{1, e^\alpha\} \text{ a.e.} \right\}.
			\end{align*}
			For any $r \in \mathcal{E}$, define the sets
			\[
			F^+_r := r^{-1}(\{e^\alpha\}) \quad \text{and} \quad F^-_r := r^{-1}(\{1\}),
			\]
			so that, $dx$-almost everywhere, we can write
			\begin{equation}{\label{eq: def r(x)}}
				r(x) = e^\alpha \indi{F^+_r}(x) + \indi{F^-_r}(x) = 1 + (e^\alpha - 1)\indi{F^+_r}(x).
			\end{equation}
			These sets and expressions are the continuous analogues of the discrete objects introduced in \eqref{eq: Fbeta+ 4.8} and \eqref{eq: r beta x}.

			As detailed in Section~\ref{S: tools} of {\revar the Appendix} on functional analytic tools, we endow $\theconvex$ with a distance $d_\star$ that metrizes the weak-$\star$ topology. The set $\theconvex$ is convex, compact, and separable (see Lemmas~\ref{l: B closed} and~\ref{l: B separable}). Moreover, $\mathcal{E}$ coincides with the set of extremal points of $\theconvex$ (see Lemma \ref{lem : extremal points}). By Theorem~\ref{th: Choquet}, $\mathcal{E}$ is a $G_\delta$ set (i.e., a countable intersection of open sets), and since we equip $\theconvex$ with the Borel $\sigma$-algebra $\Borelof{\theconvex}$, it follows that $\mathcal{E}$ is measurable. 
			
			Note that $\theconvex \subset \widehat{B} \subset L^\infty(\mathcal{X}, \mathbb{R})$, 
			where $\widehat{B}$ is the closed ball $B_{\norm{\cdot}_\infty}(0,e^\alpha)$,
			so evaluating an element of $\theconvex$ at a point $x \in \mathcal{X}$ is essential. The evaluation operator $e$ defined and extensively studied in Section \ref{ss:point eval} plays a key role in this context:
			\begin{equation*}
				e : 
				\begin{cases}
					(\mathcal{X} \times  \widehat{B}, \Borelof{\cX} \otimes  \Borelof{\widehat{B}}) \to \mathbb{R} \\
					(x, r) \mapsto e_x(r).
				\end{cases}
			\end{equation*}
			We show in Subsection~\ref{ss:point eval} it is measurable and satisfies, for all $r \in  \widehat{B}$, $e_x(r) = r(x)$ for a.e.\ $x$. {\revarn Moreover, we will use that $e_x(r)\in[1,e^\alpha]$ for all $(x,r)\in\cX\times\theconvex$, which is immediate from the definition of $e_x(r)$ given
				in \eqref{eq : def evaluation operator} in the Appendix.
				
				We are now ready to introduce extremal randomizations in the continuous setting. For any non-negative Borel measure $\mu$ on $\mathcal{E}$ satisfying
				\begin{equation} \label{eq : cond norm mu continue}
					\text{for a.e.\ } x \in \mathcal{X}, \quad \int_{\mathcal{E}} e_x(r) \, \mu(dr) = 1,
				\end{equation}
				we define a randomization mechanism {\revarn $q^{(\mu)}_x(dr)$} from $\mathcal{X}$ to $\mathcal{E}$ given by the kernel
				\begin{equation} \label{eq : def random generic extr ctn}
					q^{(\mu)}_x(dr) =
					\begin{cases}
						& e_x(r) \mu(dr), \text{ if $x\in\cX$ satisfies  $\int_{\mathcal{E}} e_x(r) \, \mu(dr) = 1$,}
						\\& e_{x_0}(r) \mu(dr), \text{ otherwise,}
					\end{cases}
				\end{equation}
				where $x_0\in \cX$ is some fixed and arbitrary point satisfying $\int_{\mathcal{E}} e_{x_0}(r) \, \mu(dr) = 1$. 
				The arbitrary definition of the randomization outside the exceptional set of \eqref{eq : cond norm mu continue} is just to guarantee that $q^{(\mu)}$ is well defined for all $x\in\cX$. However, it will play no role in statistical applications where the law of $X$ is absolutely continuous with respect to the Lebesgue measure. 
			}
			
			{\revarn Using $e_x(r) \in [1, e^\alpha]$, one can check that 
				$e^{-\alpha}\leq\frac{q^{(\mu)}_x(dr)}{q^{(\mu)}_{x'}(dr)} \leq e^{\alpha}$ is valid for all $(r,x,x')\in\cE\times\cX^2$. This proves that $q^{(\mu)}_x(dr)$ is an $\alpha$-LDP mechanism.}
			
			{\revarn Moreover, this} randomization is extremal in the sense that {\revarn for $dxdx'$-almost all $(x,x')$}
			\begin{equation} \label{Eq : def extremal cond ctn}
				\frac{q^{(\mu)}_x(dr)}{q^{(\mu)}_{x'}(dr)} 
				= \frac{e_x(r)}{e_{x'}(r)} 
				{\revarn = \frac{r(x)}{r(x')} \in \{e^{-\alpha}, 1, e^{\alpha}\} .} 
			\end{equation}
			Since $e_x(r) \in [1, e^\alpha]$, condition~\eqref{eq : cond norm mu continue} implies the bounds
			\begin{equation}{\label{eq: bound mu E}}
				e^{-\alpha} \leq \mu(\mathcal{E}) \leq 1.
			\end{equation}
			
			\subsection{A factorization result}{\label{ss: Fisher}}
			
			We emphasize that the randomizations just introduced constitute a natural generalization of the extremal randomizations appearing in Lemma~\ref{lem : factorisation cas fini}. In the finite case, the randomization takes values in the discrete set $\{0, \dots, 2^{\Card(\mathcal{X})} - 1\}$, which indexes the extremal vectors appearing in~\eqref{eq : R matrix by slice}. In contrast, in the continuous setting, the randomization takes values in the set $\mathcal{E}$ of measurable functions with values in $\{1, e^\alpha\}$—the analogous of the extremal vectors from the discrete case.
			
			The normalization condition~\eqref{eq : cond norm mu continue} is the continuous counterpart of the constraint~\eqref{eq: constraint factorisation 4.55} in the finite case. Similarly, the sub-probability measure $\mu$ on $\mathcal{E}$ plays the role of the coefficient vector $(\omega_\beta)_{\beta = 0, \dots, 2^{|\mathcal{X}|} - 1}$ used in the discrete setting.
			
			\medskip
			With this analogy in mind, we are now ready to extend the factorization result of Lemma~\ref{lem : factorisation cas fini} to the case where $\mathcal{X}$ is a continuous space.
			
			\begin{prop} \label{prop : factor ctn}
				{\revar Let {\revarn $Q(dz\mid X=x)=q_x(dz) \vartheta(dz)$} be as in Section \ref{ss : extremal mecha cont} an $\alpha$-LDP mechanism from an open subset $\mathcal{X} \subset \mathbb{R}^d$ to  a Polish and locally compact space $\mathcal{Z}$.
					Then, there exist a non-negative Radon measure $\mu$ on $\mathcal{E}$ satisfying \eqref{eq : cond norm mu continue} and a randomization mechanism $q^{(2)}$ from $\mathcal{E}$ to $\mathcal{Z}$,
					such that the mechanism $q$ admits the factorization $q = q^{(2)} \circ q^{(\mu)}$ :
					\begin{equation*}
						{\revarn Q(dz\mid X=x)}=q_x(z)\vartheta(dz)=\int_\mathcal{E} q^{(2)}_r(dz)q^{(\mu)}_x(dr), \quad \text{ for $dx$-almost every $x\in\cX$.}
					\end{equation*}
					We recall that $q^{(\mu)}$ is defined by {\revarn $q^{(\mu)}_x(dr) = e_x(r)\mu(dr)$, for almost all $x\in\cX$}, and that the measure $\mu$ satisfies
					$e^{-\alpha} \le \mu(\mathcal{E}) \le 1$.}
			\end{prop}
			\begin{proof}
				{\revis For simplicity, we present the proof here for the case where $\cZ$ is finite or countable; the proof for general $\cZ$ is provided in Section~\ref{ss: factorization Polish Z} of the Appendix.
					
					Since $\cZ$ is discrete, we may, without loss of generality, restrict $\cZ$ to the support of $\vartheta$.} As {\revarn $Q$} satisfies the $\alpha$-LDP constraint, \eqref{eq : alpha LDP weaker} gives that for all $z \in \mathcal{Z}$:
				\[
				\frac{q_x(z)}{q_{x'}(z)} \le e^\alpha, \quad \text{for a.e.  $(x, x') \in \mathcal{X} \times \mathcal{X}$.}
				\]
				This implies that the function
				\begin{equation}{\label{eq: 4.4.5}}
					\underline{q}(z) := \essinf_{x \in \mathcal{X}} q_x(z),    
				\end{equation}
				defined analogously to the minimal mass function in Lemma~\ref{lem : factorisation cas fini}, satisfies $\underline{q}(z) > 0$. Define for each $z \in \mathcal{Z}$ and $x \in \mathcal{X}$ the ratio
				\[
				u_x(z) := \frac{q_x(z)}{\underline{q}(z)}.
				\]
				By construction, for all $z \in \mathcal{Z}$, the function $x \mapsto u_x(z)$ satisfies $1 \le u_x(z) \le e^\alpha$ for almost every $x \in \mathcal{X}$. 
				
				Let us denote $v_z : x \mapsto u_x(z)$, which defines a function in the compact convex set $\theconvex$ introduced earlier. We apply Choquet's Theorem (see Theorem~\ref{th: Choquet}), whose applicability to the set $\theconvex$ is justified in Section~\ref{S: tools}. This yields the existence of a probability measure $\mu_z$ on $\mathcal{E}$ such that
				\[
				v_z = \int_{\mathcal{E}} r \, \mu_z(dr),
				\]
				where the integral is taken in the Bochner sense. 
				By Proposition~\ref{p: evaluation}, applying the evaluation operator $e_x$ gives
				\[
				e_x(v_z) = \int_{\mathcal{E}} e_x(r) \mu_z(dr).
				\]
				Note that for a.e.\ $x \in \mathcal{X}$, we have $e_x(v_z) = v_z(x) = u_x(z)$, so that
				\[
				u_x(z) = \int_{\mathcal{E}} e_x(r) \mu_z(dr), \quad \text{for a.e.\ } x \in \mathcal{X},\ \forall z \in \mathcal{Z}.
				\]
				Recalling that $q_x(z) = u_x(z) \underline{q}(z)$, we deduce
				\begin{equation}
					\label{eq : q x with mu r in proof}
					q_x(z) = \int_{\mathcal{E}} e_x(r) \mu_z(dr)\underline{q}(z), \quad \text{for a.e.\ } x \in \mathcal{X},\ \forall z \in \mathcal{Z}.
				\end{equation}
				Since $q$ is a randomization mechanism, {\revar and recalling that $\vartheta$ is the reference measure on $\cZ$,} we have
				\[
				{\revar \int_\cZ q_x(z) \vartheta(dz) =1,  \quad \text{for a.e.\ } x \in \mathcal{X}.}
				\]
				Using this and~\eqref{eq : q x with mu r in proof}, we find for a.e.\ $x \in \mathcal{X}$:
				{\revar
					\begin{equation} \label{eq : normalisation 4-65}
						1 = 
						\int_\cZ  \left( \int_{\mathcal{E}} e_x(r) \mu_z(dr) \underline{q}(z) \right) \vartheta(dz)
						=	\int_{\mathcal{E}} e_x(r) \left( \int_\cZ  \underline{q}(z) \vartheta(dz) \mu_z(dr) \right),
				\end{equation}}
				where the second line uses Fubini-Tonelli. We now define the measure
				\begin{equation}
					\label{eq : def mu proof fact ctn}
					{\revar \mu(dr) := \int_{\mathcal{Z}} \underline{q}(z) \vartheta(dz) \mu_z(dr)
						=\sum_{z\in\cZ}\underline{q}(z)\vartheta(\{z\}) \mu_z(dr)}
				\end{equation}
				which is a non-negative measure on $\mathcal{E}$. With this definition, we have {\revar shown in \eqref{eq : normalisation 4-65}:}
				\[
				\int_{\mathcal{E}} e_x(r) \mu(dr) = 1, \quad \text{for a.e.\ } x \in \mathcal{X},
				\]
				which is exactly the normalization condition~\eqref{eq : cond norm mu continue}. Consequently, the measure $\mu$ defines a randomization mechanism with values in $\mathcal{E}$, 
				{\revarn via the formula \eqref{eq : def random generic extr ctn}.}
				
				It remains to define a randomization $q^{(2)}$ from $\mathcal{E}$ to $\mathcal{Z}$ such that $q = q^{(2)} \circ q^{(\mu)}$. {\revar From~\eqref{eq : def mu proof fact ctn} and $\vartheta(\{z\})>0$,} we know that for each $z \in \mathcal{Z}$, the measure $\underline{q}(z) \mu_z$ is absolutely continuous with respect to $\mu$. By the Radon-Nikodym theorem, there exists a non-negative measurable function $\tilde{q}^{(2)}_z : \mathcal{E} \to \mathbb{R}_+$ such that
				\begin{equation}
					\label{eq : Rad Nyk in proof}
					\underline{q}(z) \mu_z(dr) = \tilde{q}^{(2)}_z(r) \mu(dr).
				\end{equation}
				{\revar Now, we write the following equalities between measures on $\cE$,
					\begin{align*}
						\left( \int_{\mathcal{Z}} \tilde{q}^{(2)}_z(r) \vartheta(dz) \right) \mu(dr) &= 
						\sum_{z\in\cZ}  \tilde{q}^{(2)}_z(r)\mu(dr) \vartheta(\{z\}) 
						\\
						&=	\sum_{z\in\cZ} \underline{q}(z)  \mu_z(dr)\vartheta(\{z\}), \text{ by \eqref{eq : Rad Nyk in proof},}
						\\
						&=\int_{\cZ} \underline{q}(z) \mu_z(dr)  \vartheta(dz) = \mu(dr), \text{ by \eqref{eq : def mu proof fact ctn}.}
					\end{align*}
				}
				{\revar This implies that 
					$\int_\cZ 	\tilde{q}^{(2)}_z(r) \vartheta(dz)=1$
				} for $\mu$-a.e.\ $r \in \mathcal{E}$.
				
				We now define the randomization {$q^{(2)}_r(z)\vartheta(dz)$} from $\mathcal{E}$ to $\mathcal{Z}$ by:
				\[
				q^{(2)}_r({z}) := 
				\begin{cases}
					\tilde{q}^{(2)}_z(r) & \text{if } 
					{\revar \int_\cZ 	\tilde{q}^{(2)}_z(r) \vartheta(dz)=1,} \\ 
					{\revar \indi{z_1}(z)/\vartheta(\{z_1\})} & \text{otherwise},
				\end{cases}
				\]
				for some fixed $z_1 \in \mathcal{Z}$. Note that $\mu(dr)$-a.e.,
				$q^{(2)}_r(z) = \tilde{q}^{(2)}_z(r)$ for all $z \in \mathcal{Z}$, and the definition outside the support of $\mu$ is irrelevant.
				
				We can now return to~\eqref{eq : q x with mu r in proof}, and using~\eqref{eq : Rad Nyk in proof}, we write for all $z\in\cZ$ and $dx$-a.e. :
				\begin{align*}
					q_x(z) &= \int_{\mathcal{E}} e_x(r) \tilde{q}^{(2)}_z(r) \mu(dr) = \int_{\mathcal{E}} e_x(r) q^{(2)}_r(z) \mu(dr) = \int_{\mathcal{E}} q^{(2)}_r(z) q^{(\mu)}_x(dr),
				\end{align*}
				which establishes that
				\begin{equation}
					\label{eq : composition in the proof ctn}
					q = q^{(2)} \circ q^{(\mu)}.
				\end{equation}
				This completes the factorization construction in the continuous case {\revar for $\cX$ and with $\cZ$ at most countable.}
			\end{proof}

			\begin{rem}
				\begin{enumerate}
					\item The proof of Proposition \ref{prop : factor ctn} follows the general scheme of the proof of Lemma \ref{lem : factorisation cas fini}, but in a more involved setting. A key aspect for the application of Choquet's theorem is the compactness of the set $\theconvex$, which motivates the introduction of the weak-$\star$ topology. However, as pointed out in Remark \ref{rk: E not closed} below, the set of extremal points $\cE$ is not closed in the weak-$\star$ topology, and thus not compact. This represents a fundamental difference with the finite case, where the set of extremal points is finite, and hence compact.
					\item {\revar In the case of a non countable $\cZ$, the proofs contains several additional technical difficulties, such as the impossibility to use the Radon-Nikodym theorem in the construction of $q^{(2)}$. The proof is postponed to the Section \ref{ss: factorization Polish Z} of the Appendix. }   
					\item The proposition above shows that any $\alpha$-LDP randomization can be factorized as the composition of two randomizations, one of which is an explicit extremal randomization $q^{(\mu)}$, fully described by a measure $\mu$. As in the discrete case, this representation will be instrumental in deriving upper bounds on the Fisher information of $\alpha$-LDP randomizations. Indeed, we can consider maximizing the Fisher information over all extremal randomizations, that is, over all possible measures $\mu$. However, since $\cE$ is not compact in general (in the weak-$\star$ topology), the set of sub-probability measures on $\cE$ is not compact either . This presents a major challenge in establishing the existence of a solution to the corresponding optimization problem, as discussed in detail {\revar after Proposition \ref{prop : maximisation pb convex multidim}}.
				\end{enumerate}
			\end{rem}

			{\revar
				\subsection{Existence of optimal measures}
				One main consequence of the factorization result is that it allows us to reduce the search for an optimal randomization mechanism to the smaller family of extremal mechanisms. This methodology is used 
				in Section \ref{ss:Parametric models and Fisher information}  for the discrete setting. It enables 
				to compute the optimal Fisher information, among all $\alpha$-LDP mechanisms, for a one-dimensional parameter and for small $\alpha$. In that case, the reduction to extremal mechanisms yields the finite dimensional optimization problem \eqref{eq: opt 14.75}. 
				
				In the continuous setting, the same methodology reduces the search of optimal randomization mechanisms to the search of an optimal measure $\mu$ describing the extremal mechanism $q^{(\mu)}$.   
				
				Our next result gives sufficient conditions that guarantee existence of an optimal measure $\mu$ in such optimization problems. This result will be used in Section \ref{Ss : Fisher cont} to study the optimal Fisher information for a continuous private variable with a $1$ dimensional parameter.  Let us stress that the factorization method and the proposition below can be applied to other statistical problems, {\revis as mentioned in Remark \ref{rem : utility Kairouz}}.
				
				Let $\mathrm{d} \ge 1$ and $\mathcal{S}_+$ be the cone of non-negative symmetric matrices of dimensions $\mathrm{d}\times \mathrm{d}$. We endowed the set $\mathcal{S}_+$ with the partial Loewner order and consider
				$j : \begin{cases} \theconvex \to \mathcal{S}_+ \\ r \mapsto j(r) \end{cases}$ a bounded, continuous, and convex function, and $\ell : \mathcal{S}_+ \to \mathbb{R}$ an increasing continuous function.
				\begin{prop} \label{prop : maximisation pb convex multidim}
					Consider the optimization problem
					\begin{equation} \label{eq : optimisation pb ctn J}
						J^* := \sup_{\mu} ~\ell \left( \int_\cE j(r) \, \mu(dr) \right), \quad
						\text{subject to } 
						 \int_\cE e_x(r) \, \mu(dr) = 1 \quad \text{for $dx$-almost every } x,
					\end{equation}
					where the supremum is taken over all Radon sub-probability measures $\mu$ on $\cE$. 
					\begin{enumerate}
						\item Then, there exists a Radon sub-probability measure $\overline{\mu}$ on $\cE$ such that
						\[
						J^* = \ell \left( \int_\cE j(r) \, \overline{\mu}(dr) \right)
						\quad \text{and} \quad
						\int_\cE e_x(r) \, \overline{\mu}(dr) = 1 \quad \text{for $dx$-almost every } x.
						\]
						\item Moreover, if in \eqref{eq : optimisation pb ctn J} we replace the occurrences of $\mathcal{E}$ by $C$ and maximize over any $\mu$ Radon sub-probabilities on $\theconvex$, we do not modify the value of $J^*$.
					\end{enumerate}

				\end{prop}
				The proof of Proposition \ref{prop : maximisation pb convex multidim} is provided in Section~\ref{ss: proof cont} of the Appendix. }
			\begin{rem}{\label{rk: E not closed}} {\revis 
					The proof of the above proposition relies on the compactness of $\theconvex$ under the weak-$\star$ topology (see Lemma~\ref{l: B closed} in the Appendix), which allows one to establish the existence of a solution over $\theconvex$. However, the set $\mathcal{E}$ is not closed in general (see a counterexample in Remark~\ref{rem : E not closed}), which prevents us from directly concluding that the solution to the maximization problem is supported in $\mathcal{E}$ via this route.
					
					Nevertheless, we are able to construct an optimal measure supported on $\mathcal{E}$ by using convexity arguments and several tools from set theory, including techniques akin to those used in the proof of Choquet’s theorem, as developed in~\cite{Choquet}, whose main tool to derive the existence of measures supported on $\cE$ is the 
					Zorn's lemma. }
			\end{rem}

			\subsection{Fisher information in the continuous case}{\label{Ss : Fisher cont}}
			We now investigate how the factorization result can be leveraged to derive an expression for the maximal Fisher information. Again, addressing this problem requires some regularity assumptions on the statistical model.
			
			Assume that $\cX$ is an open subset of $\R^d$ and that $(P_{\theta})_{\theta \in \Theta}$ is a family of probability measures on $\cX$. {\revar  For simplicity we consider a one dimensional parameter and
				assume $\Theta \subset \R$.}
			We suppose that this family is dominated by the Lebesgue measure and denote its density by $p_{\theta}(x) := \frac{dP_\theta(x)}{dx}$.
			
			Let $\alpha > 0$ and let $\mu$ be a Radon sub-probability measure on $\cE$ satisfying \eqref{eq : cond norm mu continue}. We consider the privacy mechanism $q^{(\mu)}$ as introduced in Section \ref{ss: infinite}, given by $q^{(\mu)}_x(dr) = e_x(r)\mu(dr)$. Denote by $Z$ the associated public data, and by $(\widetilde{P}_{\theta})_{\theta \in \Theta}$ the induced model on $\cE$, i.e., the law of $Z$.
			
			The following lemma shows that this model admits a continuous density and is differentiable in quadratic mean, with an explicit and continuous score function.
			\begin{lem} \label{lem : expression tilde p et score ctn}
				Let $\theta_0 \in \overset{\circ}{\Theta}$. The following statements hold:
				\begin{enumerate}
					\item The probability measure $\widetilde{P}_{\theta}$ is absolutely continuous with respect to $\mu$ and its density is given by
					\begin{equation} \label{eq : tilde p ctn}
						\widetilde{p}_\theta(r) := \frac{d\widetilde{P}_{\theta}}{d\mu}(r)
						= \int_\cX r(x) p_\theta(x) dx.
					\end{equation}
					\item Assume that the model $(P_{\theta})_{\theta \in \Theta}$ is DQM at $\theta_0$ with score function $s_{\theta_0}$ such that 
					${\revar \E_{\theta_0}\left[\abs{s_{\theta_0}(X)}^2\right] < \infty}$. Then the model 
					$(\widetilde{P}_{\theta})_{\theta \in \Theta}$ is also DQM at $\theta_0$, and its score function is given by
					\begin{equation} \label{eq : score t ctn}
						t_{\theta_0}(r) = \frac{\int_\cX s_{\theta_0}(x) r(x)p_{\theta_0}(x) dx}{\widetilde{p}_{\theta_0}(r)} = (e^\alpha - 1) \frac{\int_{F_r^+} s_{\theta_0}(x) p_{\theta_0}(x) dx}{\widetilde{p}_{\theta_0}(r)}.
					\end{equation}
					\item The functions $r \mapsto \widetilde{p}_{\theta_0}(r)$ and $r \mapsto t_{\theta_0}(r)$ are continuous on $\cE$.
				\end{enumerate}
			\end{lem}
			\begin{proof}
				\begin{enumerate}
					\item The law of $Z$ is given by
					\[
					\widetilde{P}_{\theta}(dr) = \int_\cX q^{(\mu)}_x(dr) p_\theta(x)dx = \int_\cX e_x(r) \mu(dr) p_\theta(x)dx = \left( \int_\cX r(x) p_\theta(x) dx \right) \mu(dr),
					\]
					where we used that $e_x(r) = r(x)$ $dx$-a.e. for $r \in \cE$. Hence, the density with respect to $\mu$ is as in \eqref{eq : tilde p ctn}.
					\item Applying Lemma 3.1 in \cite{steinbergerEfficiencyLocalDifferent}, we know that the model $(P_\theta)$ is DQM with score function $s_{\theta_0}$. From Lemma A1 in the same reference, we obtain:
					\[
					t_{\theta_0}(r) = \frac{1}{\widetilde{p}_{\theta_0}(r)} \int_\cX s_{\theta_0}(x) e_x(r) p_{\theta_0}(x) dx = \frac{\int_\cX s_{\theta_0}(x) r(x) p_{\theta_0}(x) dx}{\widetilde{p}_{\theta_0}(r)}.
					\]
					Recalling that $r(x)$ is as in \eqref{eq: def r(x)} and that $\int_{\mathcal{X}} s_{\theta_0}(x)p_{\theta_0}(x) dx = 0$ we deduce the second formula.
					\item The continuity of $r \mapsto \widetilde{p}_{\theta_0}(r)$ follows from \eqref{eq : tilde p ctn}, the integrability of $p_{\theta_0}$, and the definition of the weak-$\star$ topology on $\cE$. Analogously, the continuity of $r \mapsto t_{\theta_0}(r)$ follows from \eqref{eq : score t ctn} and the integrability of $s_{\theta_0}(x)p_{\theta_0}(x)$.
				\end{enumerate}
			\end{proof}
			We now aim to use the factorization result to study the 
			information $\cI_{\theta_0}(q\circ \cP)$ that can be extracted from an $\alpha$-LDP view of the private data through a randomisation mechanism $q$ from $\cX$ to some finite or countable output space $\cZ$.
			Using the factorization result in Proposition \ref{prop : factor ctn}, we obtain the inequality:
			\begin{align}{\label{eq: fact on Fisher cont}}
				\mathcal{I}_{\theta_0}(q \circ \mathcal{P}) 
				&= \mathcal{I}_{\theta_0}((q^{(2)} \circ q^{(\mu)}) \circ \mathcal{P}) \\
				&= \mathcal{I}_{\theta_0}(q^{(2)} \circ (q^{(\mu)} \circ \mathcal{P})) 
				\le \mathcal{I}_{\theta_0}(q^{(\mu)} \circ \mathcal{P}), \nonumber
			\end{align}
			where we used \eqref{eq : decroissance Fisher}.
			Therefore, to study the maximization of the Fisher information, we may restrict our attention to extremal mechanisms of the form $q = q^{(\mu)}$. Thanks to Lemma \ref{lem : expression tilde p et score ctn}, we can write the Fisher information explicitly:
			\begin{align}\nonumber
				\cI_{\theta_0}(q^{(\mu)}\circ \cP)
				&= \E_{\theta_0} \left[ {\revar  t_{\theta_0}(Z)^2 } \right] = \int_\cE {\revar t_{\theta_0}(r)^2}  \widetilde{p}_{\theta_0}(r) \mu(dr) \\
				&=\nonumber  \int_{\cE} 
				\frac{ {\revar  \left( \int_\cX s_{\theta_0}(x) r(x) p_{\theta_0}(x) dx \right)^2}}
				{\widetilde{p}_{\theta_0}(r)} \mu(dr) \\
				& =\label{eq : Fisher ctn general_preli}(e^\alpha - 1)^2 \int_{\cE} 
				\frac{{\revar \left( \int_{F_r^+} s_{\theta_0}(x) p_{\theta_0}(x) dx \right)^2 }}
				{\widetilde{p}_{\theta_0}(r)} \mu(dr).
			\end{align}
			
			This representation is a key technical tool to derive the next theorem, which provides general lower and upper bounds for the Fisher information under $\alpha$-LDP constraints, valid for any $\alpha > 0$, and characterizes its behavior in the asymptotic regime $\alpha \rightarrow 0$.
			
			{\revar As} the parameter space is one-dimensional, the Fisher information is a scalar and we can define
			\begin{equation*}
				\cJ^{\text{max},\alpha}_{\theta_0} := \sup_{q \in \cQ_\alpha} \cI_{\theta_0}(q \circ \cP),
			\end{equation*}
			where the supremum runs over all $\alpha$-LDP mechanisms $q$ from $\cX$ to {\revar some Polish output space $\cZ$.}

			\begin{thm} \label{thm : equivalent alpha petit ctn}
				Assume that $\Theta \subset \R$ and that the model $(P_{\theta})_{\theta \in \Theta}$ is DQM at $\theta_0$ with score function $s_{\theta_0}$ such that $\E_{\theta_0}\left[{\revar \abs{s_{\theta_0}(X)}^2}\right] < \infty$.
				Then we have, for any $\alpha>0$,
				\begin{equation} \label{eq :  encadrement Fisher ctn small alpha}
					\frac{(e^\alpha-1)^2}{2e^\alpha(1+e^\alpha)}
					\left( \int_\cX \abs{s_{\theta_0}(x)}p_{\theta_0}(x)dx \right)^2
					\le
					\cJ^{\text{max},\alpha}_{\theta_0} \le \frac{(e^\alpha-1)^2}{4}
					\left( \int_\cX \abs{s_{\theta_0}(x)}p_{\theta_0}(x)dx \right)^2.	
				\end{equation}
				Consequently,
				\begin{equation*}
					\cJ^{\text{max},\alpha}_{\theta_0} \sim_{\alpha\to0} \frac{\alpha^2}{4}
					\left( \int_\cX \abs{s_{\theta_0}(x)}p_{\theta_0}(x)dx \right)^2.		
				\end{equation*}	
			\end{thm}
			\begin{rem} \label{rem : optimal mech continous case alpha small}
				{\revis
					We can compare Theorem~\ref{th : info cas fini} in the finite case with Theorem~\ref{thm : equivalent alpha petit ctn} in the continuous setting. Theorem~\ref{th : info cas fini} provides the exact value of the optimal Fisher information in the discrete case, whereas Theorem~\ref{thm : equivalent alpha petit ctn} gives an asymptotic characterization in the regime $\alpha\to0$. The two results are consistent, but the sharper finite-case characterization relies on the finite-dimensional structure of the set of staircase patterns and is not recovered directly from the continuous upper and lower bounds. Thus, although the continuous argument extends the factorization principle, the discrete analysis contains additional information that is specific to finite sample spaces.
					
					In the proof of Theorem~\ref{thm : equivalent alpha petit ctn}, we consider the extremal mechanism associated with the measure
					$\mu^{\mathrm{opt}}
					=
					\frac{1}{1+e^\alpha}
					\left(
					\delta_{r_1}+\delta_{r_2}
					\right)$, where
					\[
					r_1
					=
					1+(e^\alpha-1)\indi{\{x:s_{\theta_0}(x)>0\}},
					\qquad
					r_2
					=
					1+(e^\alpha-1)\indi{\{x:s_{\theta_0}(x)\leq0\}}.
					\]
					We prove that this explicit mechanism is asymptotically optimal among all $\alpha$-LDP mechanisms as $\alpha\to0$. It is the continuous analogue of the extremal mechanism used in the lower-bound argument for the discrete case; see Proposition~\ref{prop : borne inf Fisher fini} in the Appendix.
				}

			\end{rem}
			\begin{rem} {\revis Theorems \ref{th : info cas fini} and \ref{thm : equivalent alpha petit ctn} detail how the $L^1$ information introduced in \cite{DucRua24}, namely $\int_\mathcal{X} \abs{s_{\theta_0}(x)}p_{\theta_0}(x)dx$, governs the Fisher information of the optimal privacy mechanism in the small-$\alpha$ regime. It was established in \cite{DucRua24} that, under the privacy constraint, the minimum dispersion of an estimator is proportional to the inverse of the $L^1$ information. Compared to \cite{DucRua24}, our result provides the exact asymptotic equivalent, as $\alpha \to 0$, for the minimum estimation variance under the $\alpha$-LDP setting.}
			\end{rem}
			\begin{rem}
				It is interesting to note that, in the case where the score function is given by $s = \frac{\dot{p}_{\theta_0}}{p_{\theta_0}}$ 
				and we consider a Gaussian distribution, the quantity $\int_{\mathcal{X}} \abs{s_{\theta_0}(x)} \, p_{\theta_0}(x) \, dx$
				is easily computable and equals \( \sqrt{\frac{2}{\pi}} \). We can compare this with the results in \cite{KalSte} concerning the efficient estimation of a Gaussian mean under local differential privacy. In that work, the optimal Fisher information is shown to be 
				$\frac{2}{\pi}\frac{(e^\alpha - 1)^2}{(e^\alpha + 1)^2}$.
				Notably, in the limit as \( \alpha \to 0 \), one recovers the same result as in our calculation. It is also interesting to compare this with the exact maximal Fisher information derived in the discrete case, in Equation \eqref{eq : def I max caf fini}. Specifically, for a Gaussian variable, it is \( n_{\max} = \frac{1}{2} \), which simplifies the denominator to $\left( \frac{e^\alpha + 1}{2} \right)^2,$
				leading to the same optimal Fisher information as reported in \cite{KalSte}. When \( \alpha \) is close to zero, this denominator is asymptotically equivalent to 1, which explains why neglecting it in \eqref{eq : def I max caf fini} still provides the correct asymptotic result. 
			\end{rem}

			We now move to the proof that, for any regular statistical model and given $\alpha > 0$, there exists an extremal mechanism $q^{(\mu)}$ which maximizes the Fisher information.
			
			The key observation is that the continuous version of the maximization problem given in~\eqref{eq: opt 14.75} admits a solution.

			To simplify the notation, let us define, for any $r \in \theconvex$,
			\begin{equation} \label{eq : def i r ctn convex}
				i(r) := \frac{ \left( \int_\cX  s_{\theta_0}(x) \, r(x) \, p_{\theta_0}(x) \, dx \right)^2 }{
					\int_{\cX} r(x) \, p_{\theta_0}(x) \, dx } .
			\end{equation}
			For $r \in \cE$, by using  \eqref{eq: def r(x)}, we have
			\begin{equation}	\label{eq : def i r ctn extremal}
				i(r)	= (e^\alpha - 1)^2 \frac{ \left( \int_{F_r^+}  s_{\theta_0}(x) \, p_{\theta_0}(x) \, dx \right)^2 }{ 1 + (e^\alpha - 1)
					\int_{F_{r}^+} p_{\theta_0}(x) \, dx }.
			\end{equation}
			Observe that the quantity multiplying \( (e^\alpha - 1)^2 \) above plays exactly the same role as \( i_\beta \) defined in \eqref{eq : def i beta}, in the finite case.
			{\revar From \eqref{eq : Fisher ctn general_preli} and recalling \eqref{eq : tilde p ctn}, the  maximization of the information over the extremal mechanisms} 
			led to consider the following optimization problem:
			\begin{equation} \label{eq : optimisation pb ctn}
				I^* := \sup_{\mu} \int_\cE i(r) \, \mu(dr), \quad
				\text{subject to } \int_\cE e_x(r) \, \mu(dr) = 1 \quad \text{for $dx$-almost every } x,
			\end{equation}
			where the supremum is taken over all Radon sub-probability measures $\mu$ on $\cE$.
			
			By Proposition~\ref{prop : maximisation pb convex multidim}, we deduce that the above optimization problem admits a solution. Its applicability is formalized in the following theorem, whose proof can be found in Subsection~\ref{ss: proof cont}. 
			
			\begin{thm}{\label{thm: main opt cont}}
				There exists a Radon sub-probability measure $\overline{\mu}$ on $\cE$ such that the mechanism defined by $q^{(\overline{\mu})}(dr) := e_x(r) \, \overline{\mu}(dr)$ satisfies
				\begin{equation*}
					\cJ^{\text{max},\alpha}_{\theta_0} = \cI_{\theta_0}\left( q^{(\overline{\mu})} \circ \cP \right).
				\end{equation*}
				{\revarn The measure $\overline{\mu}$ is a minimizer of \eqref{eq : optimisation pb ctn} and $	\cJ^{\text{max},\alpha}_{\theta_0} =I^*$.}
			\end{thm}
			We have therefore established that one can always find an extremal privacy mechanism that reaches the optimal Fisher information for $\Theta \subset \mathbb{R}$. Moreover, we have explicitly characterized this information in the regime where the privacy parameter $\alpha$ tends to zero. Let us stress that Proposition \ref{prop : maximisation pb convex multidim} is a general result which can be used to maximize other criteria than the Fisher information. 
			\begin{rem}
				{\revar
					A representation  analogous to \eqref{eq : Fisher ctn general_preli} is also valid for the Fisher information matrix in the case of multidimensional parameters. 
					An application of Proposition \ref{prop : maximisation pb convex multidim} would then yield the existence of extremal mechanisms minimizing some measure of dispersion, such as the variance, of the estimator. 
					However, it is not clear how results like Theorems \ref{th : info cas fini} or \ref{thm : equivalent alpha petit ctn} can be extended to the cases of multidimensional parameters. In particular, it is not straightforward to generalize the definition of $\mu^{opt}$ given
					in Remark \ref{rem : optimal mech continous case alpha small}.
				}
			\end{rem}
			
			
			{\revarn 
				We end this section by stating that the optimal Fisher information is also characterized as the solution of a minimization problem, dual to the linear maximization problem described in \eqref{eq : optimisation pb ctn}.
				By Proposition \ref{prop : maximisation pb convex multidim} 2), the maximization problem \eqref{eq : optimisation pb ctn} is unchanged if we replace the extremal set $\cE$ by the convex $\theconvex$. We introduce the dual program, which reads,
				\begin{equation} \label{eq : dual nos notations}
					\tilde{I}^*=\inf_{f \in L^1(\mathcal{X})} \int_{\mathcal{X}} f(x) dx, \quad \text{subject to}
					\int_{\mathcal{X}} r(x) f(x) dx \ge i(r), ~ \text{for all $r \in {C}$.}
				\end{equation}
				This dual program has also been considered for the discrete setting in \cite{KalSte}. In Section \ref{S : dual program supp} of the Appendix, we prove that the primal and dual programs are equivalent, as stated in the following proposition.
				\begin{prop}\label{prop : dual no gap}
					\begin{enumerate}
						\item We have $\tilde{I}^*=I^*=	\cJ^{\text{max},\alpha}_{\theta_0}$.
						\item There exists a minimizer $f^* \in L^1(\cX)$ of \eqref{eq : dual nos notations}.
					\end{enumerate}
				\end{prop}	
				In Section \ref{S : implementation pathways}, we discuss how this duality could be leveraged to numerically approximate the optimal Fisher information.
			}
			
			{\revar 
				\section{Example of application to regular and irregular models}
				\label{s: Example of app}
				To illustrate the results of Section \ref{ss: infinite}, we first consider a statistical model where Theorem \ref{thm : equivalent alpha petit ctn} directly applies.
				Then, we study the situation of a uniform distribution, for which the results of Section \ref{ss: infinite}, concerned with regular models, can not be  directly used.  
				
				\subsection{Application to a regression model}
				Let us consider the one dimensional regression model, where the private data is a sequence of i.i.d. pair $(X_i,Y_i)_{i=1,\dots,n}$ satisfying the 
				regression model $Y_i=\theta X_i + W_i$ where $W_i$ is centered Gaussian with variance $\sigma^2$. The aim is to estimate the true value  $\theta_0 \in \mathbb{R}$.  This model is studied in Example 9 of \cite{DucRua24}. The score function is $s_{\theta}(x,y)=\sigma^{-2}x(x\theta-y)$. Following \cite{DucRua24}, the squared $L^1$ information of the model 
				is \begin{equation*}
					\left(\int_\mathbb{R} \abs{s_{\theta_0}(x,y)}p_{\theta_0}(x,y) dxdy\right)^2=
					\frac{2}{\sigma^2\pi}\E\left[\abs{X}\right]^2.
				\end{equation*}
				A direct application of Theorem \ref{thm : equivalent alpha petit ctn} gives that the optimal Fisher information satisfies:
				\begin{equation*}
					\cJ^{\text{max},\alpha}_{\theta_0} \sim_{\alpha\to0} \frac{\alpha^2}{2\pi \sigma^2} \E\left[\abs{X}\right]^2.
				\end{equation*}
				Comparing with the estimate (23)  in \cite{DucRua24} for the estimation risk, we see that our result specifies the exact constant in the equivalent for this estimation risk in the regime $\alpha \to 0$.

				This asymptotically optimal information is obtained for a Bernoulli mechanism. Using the notation of the previous section, we set $F^+_\text{max}=\{(x,y) \in \mathbb{R}^2: s_{\theta_0}(x,y)>0 \}=\{(x,y) \in \mathbb{R}^2:  x(x\theta_0-y)>0 \}$, and
				$F^{\prime +}_\text{max}=\{(x,y) \in \mathbb{R}^2: s_{\theta_0}(x,y)\le 0 \}=\{(x,y) \in \mathbb{R}^2:  x(x\theta_0-y)\le0 \}$.
				We denote $r_1=1+(e^{\alpha}-1)\indi{F^+_\text{max}}$ and $r_2=1+(e^{\alpha}-1)\indi{F^{\prime+}_\text{max}}$. It is established in the proof of Theorem \ref{thm : equivalent alpha petit ctn} that an asymptotically optimal mechanism is given by
				$q^{(\mu)}_{(x,y)}(dr):=e_{(x,y)}(r) \mu(dr)$ with $\mu(dr)=\frac{1}{1+e^\alpha} \left(\delta_{r_1}+\delta_{r_2}\right)$. Thus, the public data is $Z \in \{r_1,r_2\} $,
				with
				\begin{align*}
					P(Z=r_1 \mid (X,Y)=(x,y))&= q^{(\mu)}_{(x,y)}(\{r_1\}) = \frac{e_{(x,y)}(r_1)}{1+e^\alpha} \delta_{r_1}(\{r_1\}) + 
					\frac{e_{(x,y)}(r_2)}{1+e^\alpha} \delta_{r_2}(\{r_1\}) 
					\\ &= \frac{r_1(x,y)}{1+e^\alpha},
					\text{~ $dxdy$ a.e., by $e_{(x,y)}(r_1)=r_1(x,y)$, }
					\\&=\begin{cases}
						e^{\alpha}/(1+e^\alpha),\quad& \text{ if $x(x\theta_0-y)>0$,}
						\\
						1/(1+e^\alpha),\quad &\text{ if $x(x\theta_0-y)\le 0$,}
					\end{cases}
				\end{align*}
				and $P(Z=r_2 \mid (X,Y)=(x,y))=1-	P(Z=r_1 \mid (X,Y)=(x,y))$.
				
				This mechanism uses the knowledge of $\theta_0$. Hence, in practice it is necessary to pre-estimate $\theta_0$ on a subset of the data, before using it.

				\subsection{Application to the uniform law}{\label{s: unif}}
				In this section, we consider the situation of a private variable following a uniform law on the interval ${\revar(}0,\theta]$ where $\theta>0$ is an unknown parameter.} We have $\cX={\revar(}0,\infty)$ and $p_{\theta}(x)=\frac{1}{\theta} \indi{[0,\theta]}(x)$. The model is not regular, and it is well known that the rate of estimation of $\theta$ is proportional to the number of data $n$ instead of the usual $\sqrt{n}$ corresponding to finite Fisher information situations.
			Consequently, the results of the Section \ref{Ss : Fisher cont} do not apply in this case. However, the factorization result holds true, and so we know that any $\alpha$-LDP mechanism can be factorized through a mechanism $q^{(\mu)}$ as described in Section \ref{ss: infinite}. 
			The next proposition shows that for any such channel $q^{(\mu)}$ the model is DQM for almost all values of the parameter. We also give an upper bound on the Fisher information.
			
			It shows that under $\alpha$-LDP constraint, the best rate of estimation for the range of a uniform variable is $\sqrt{n}$,
			with some upper bound on the Fisher information. 
			{\revarn Let us stress that the rate optimality in the uniform model was already obtained in \cite{RohSte20}.}
			
			The proof of {\revar our} proposition also {\revis elucidates} which $\alpha$-LDP mechanism is optimal in the regime $\alpha \to 0$. It leads to an estimation procedure that will be described in Section \ref{Ss : efficient est unif}.

			\subsubsection{{\revar Upper bound on the Fisher information}}
			
			{\revis Let us consider any Radon sub-probability measure $\mu$ on $\mathcal E$} satisfying the normalization constraint \eqref{eq : cond norm mu continue}. We denote $q^{(\mu)}$ the extremal mechanism from $\cX={\revar(}0,\infty)$ to $\cE$ given by $q^{(\mu)}=e_x(r) \mu(dr)$. Let us denote by $\tilde{p}_\theta(dr)$ the law of the public data on $\cE$ when the private one follows the uniform law on ${\revar(}0,\theta]$. We have
			\begin{equation*}
				\tilde{p}_\theta(dr)=\left(\int_\cX r(x)p_\theta(x)dx\right) \mu(dr)=
				\left(\frac{1}{\theta}\int_0^\theta r(x)dx \right)\mu(dr).
			\end{equation*}
			\begin{prop}\label{prop : Fisher unif}
				For almost every $\theta_0 \in (0,\infty)$ the model
				$\tilde{p}_\theta(dr)$ is DQM at $\theta_0$. The following upper bound on the Fisher information holds
				\begin{equation*}
					\cI_{\theta_0}\left((\widetilde{p}_\theta)_\theta\right) \le \frac{(e^\alpha-1)^2}{\theta_0^2}.
				\end{equation*}
			\end{prop}
			\begin{proof}
				The density of the model is given by 
				\begin{equation}
					\label{eq : tilde p unif}
					\tilde{p}_\theta(r) =\frac{\tilde{p}_\theta(dr)}{\mu(dr)}=\frac{1}{\theta} \int_0^\theta r(x)dx.
				\end{equation}
				We have to prove that this density admits a derivative with respect to the parameter $\theta$. Define 
				\begin{equation*}
					\cR=\left\{(r,x)\in {\revarn \widehat{B}} \times {\revarn (}0,\infty) \mid e_x^+(r):=\lim_{h\to0} h^{-1}\int_x^{x+h} r(y) dy \text{ exists} \right\}.
				\end{equation*}
				By {\revis a} similar argument as for the proof of the measurability of the operator $e_x$ defined by \eqref{eq : def evaluation operator}, we can check that {\revis the} set $\cR$ is measurable with respect to {\revarn $\Borelof{\widehat{B}}\otimes\Borelof{(0,\infty)}$} (see Step 4 in the proof of Proposition \ref{p: evaluation}). By  Fubini's theorem, we deduce
				$\int_{\cE \times {\revar(}0,\infty)} \indi{\cR^c}(r,\theta) \mu(dr)\times d\theta=
				\int_{\cE} \left(\int_0^\infty \indi{\cR^c}(r,\theta) d\theta \right)\mu(dr)=0$, as for any $r\in \cE$, $\lim_{h\to0} h^{-1}\int_\theta^{\theta+h} r(y) dy$ exists for almost every $\theta$ and is equal to $r(\theta)$. Now, a further application of Fubini theorem entails  
				$ \int_0^\infty\left(\int_{\cE} \indi{\cR^c}(r,\theta) \mu(dr) \right)d\theta=0$.
				Consequently, there exists $\overline{\Theta}$ a set of real {\revis numbers}, with full Lebesgue
				measure, such that $\int_{\cE} \indi{\cR^c}(r,\theta) \mu(dr)=0$ for all $\theta \in \overline{\Theta}$.
				
				Now, we consider any $\theta_0 \in \overline{\Theta}$. From the definition of $\cR$, we have that $\mu(dr)$ a.e., the function $\theta \mapsto \int_0^\theta r(x)dx$ admits a derivative at $\theta=\theta_0$, denoted as $e^+_{\theta_0}(r)$. Recalling \eqref{eq : tilde p unif}, we deduce that $\tilde{p}_\theta(r)$ admits a derivative at $\theta=\theta_0$ given by
				\begin{align} \nonumber
					\frac{\partial}{\partial \theta} \tilde{p}_{\theta_0}(r)
					&= -\frac{1}{\theta_0^2}\int_0^{\theta_0} r(x)dx+\frac{e^+_{\theta_0}(r)}{\theta_0} 
					\\ \label{eq : p tilde dot unif}
					&= \frac{1}{\theta_0^2}\int_0^{\theta_0} (e^+_{\theta_0}(r)-r(x))dx.
				\end{align}
				We recall the definition of $F^+_r = r^{-1} \left( \{ e^\alpha \} \right),$
				and that \( r(x) = 1 + (e^{\alpha} - 1) \, \indi{F^+_r}(x) \).  
				Substituting this expression into the definition of \( e^+_{\theta_0}(r) \), we obtain  
				\[
				e^+_{\theta_0}(r) = 1 + (e^{\alpha} - 1) \, e^+_{\theta_0} \left( \indi{F^+_r} \right).
				\]
				It implies
				$e^+_{\theta_0}(r)-r(x)=(e^{\alpha}-1)(e^+_{\theta_0}(\indi{F^+_r})-\indi{F^+_r}(x))$.
				Inserting in \eqref{eq : p tilde dot unif}, we deduce that
				\begin{equation*}
					\frac{(\frac{\partial}{\partial \theta} \tilde{p}_{\theta_0}(r))^2}{\tilde{p}_{\theta_0}(r)}
					=  \frac{(e^\alpha-1)^2}{\theta_0^4}\frac{\left( \int_0^{\theta_0} 
						(e^+_{\theta_0}(\indi{F^+_r})-\indi{F^+_r}(x))dx    \right)^2}{\widetilde{p}_{\theta_0}(r)}.
				\end{equation*}
				Collecting with  \eqref{eq : tilde p unif}, we deduce the expression for the Fisher information
				\begin{equation*}
					\cI_{\theta_0}\left((\widetilde{p}_\theta)_\theta\right) 
					=	\frac{(e^\alpha-1)^2}{\theta_0^3}
					\int_{\cE}
					\frac{\left( \int_0^{\theta_0} 
						(e^+_{\theta_0}(\indi{F^+_r})-\indi{F^+_r}(x))dx    \right)^2}{ \int_0^{\theta_0} r(x)dx } \mu(dr).
				\end{equation*}	
				Since $e^+_{\theta_0}(\indi{F^+_r}) \in [0,1]$ and $\indi{F^+_r}(x)\in \{0,1\}$, we deduce that 
				\begin{equation} \label{eq : F_max unif}
					\abs*{ \int_0^{\theta_0} 
						(e^+_{\theta_0}(\indi{F^+_r})-\indi{F^+_r}(x))dx} \le \theta_0. 
				\end{equation}
				From $r\ge1$, we have $\int_0^{\theta_0} r(x)dx \ge \theta_0$, which gives
				\begin{equation*}
					\cI_{\theta_0}\left((\widetilde{p}_\theta)_\theta\right) 
					\le	\frac{(e^\alpha-1)^2}{\theta_0^3}
					\int_{\cE} \frac{\theta_0^2}{\theta_0} \mu(dr) \le \frac{(e^\alpha-1)^2}{\theta_0^2},
				\end{equation*}	
				where we used that $\mu$ is a sub-probability.	
			\end{proof}
			\begin{rem}
				The upper bound given in Proposition \ref{prop : Fisher unif} is very similar
				to the upper bound in Theorem \ref{thm : equivalent alpha petit ctn} for regular models. In particular, this upper bound is obtained by searching for $F_r^{+}$ achieving the maximum of the LHS 
				in \eqref{eq : F_max unif}. We can see that \eqref{eq : F_max unif}
				becomes an equality when $F_r^+={\revar(}0,\theta_0)$ and  $F_r^+=[\theta_0,\infty)$. It suggests that the measure $\mu$, supported on the two elements $r_1=1+(e^{\alpha}-1)\indi{{\revar(}0,\theta_0)}$ and $r_2=1+(e^{\alpha}-1)\indi{[\theta_0,\infty)}$, makes the Fisher information close to the optimal one in the regime $\alpha \to 0$. 
			\end{rem}

			\subsubsection{{\revar Efficient estimator in the regime $\alpha \to 0$}} \label{Ss : efficient est unif}	
			Assume that $X_1,\dots,X_n$ is a sampling of random variables with uniform distribution on $[0,\theta]$, and $\theta>0$. We denote by $\theta_0$ the true value of the parameter.
			
			From the previous section, we expect that a Bernoulli randomization mechanism is efficient when $\alpha \to 0$. However, the construction of the Bernoulli mechanism necessitates the knowledge of $\theta_0$.
			
			Let us assume that we have a preliminary estimation of the parameter, which we call $\widehat{\theta}^{\text{p}}>0$. For simplicity of the exposition, we consider that 
			$\widehat{\theta}^{\text{p}}$ is a deterministic value. It can be a preliminary estimator independent from $X_1,\dots,X_n$, and our analysis is then conditional to $\widehat{\theta}^{\text{p}}$. Remark that we do not need $\widehat{\theta}^{\text{p}}$ to be close to $\theta_0$.
			
			We introduce the $\alpha$-LDP extremal mechanism corresponding to a two point supported measure $\mu$ on $\cE$, which we write in a more amenable way as a randomization valued on $\cZ=\{0,1\}$, by
			\begin{equation*}
				P(Z=z \mid X=x)=q_x(z)=
				\begin{cases}
					\frac{1}{1+e^\alpha}\left(e^\alpha \indi{\{1\}}(z)+\indi{\{0\}}(z) \right),
					\text{if $x \in F^+:={\revar(}0,\widehat{\theta}^{\text{p}})$}
					\\
					\frac{1}{1+e^\alpha}\left( \indi{\{1\}}(z)+e^\alpha \indi{\{0\}}(z) \right),
					\text{if $x \in (F^+)^c=[\widehat{\theta}^{\text{p}} ,\infty)$}.
				\end{cases}
			\end{equation*}
			We draw the public views $Z_1,\dots,Z_n$ according to this channel. 
			Then, we have 
			\begin{align} \nonumber
				\widetilde{p}_{\theta}(1)=P_\theta(Z=1)&=\frac{1}{\theta}\int_0^{\widehat{\theta}^{\text{p}} \wedge \theta} \frac{e^\alpha}{1+e^\alpha} dx +
				\frac{1}{\theta}
				\int_{\widehat{\theta}^{\text{p}} \wedge \theta}^{\theta} \frac{1}{1+e^\alpha} dx 
				\\&= \label{eq : tilde p unif estimateur}
				\frac{1}{1+e^\alpha} + \frac{e^\alpha-1}{1+e^\alpha} \left(\frac{\widehat{\theta}^{\text{p}}}{\theta} \wedge 1 \right).
			\end{align}
			Estimating $\widetilde{p}_{\theta}(1)$ by $\overline{Z}_n:=\frac{1}{n}\sum_{i=1}^n Z_i$ gives the following estimator for $\theta$
			\begin{equation*}
				\widehat{\theta}_n:= \widehat{\theta}^{\text{p}} \frac{e^\alpha -1}{(1+e^\alpha) \overline{Z}_n -1}.
			\end{equation*}
			\begin{prop}{\label{prop: thetahat}}
				\begin{enumerate}
					\item We have $\widehat{\theta}_n \xrightarrow[a.s.]{n\to\infty} \theta_0 \vee \widehat{\theta}^{\text{p}}$.
					\item If $\widehat{\theta}^{\text{p}} \le \theta_0$, the estimator is consistent and
					\begin{equation*}
						\sqrt{n} \left(\widehat{\theta}_n-\theta_0\right)
						\xrightarrow[\cL]{n \to \infty} \mathcal{N}\left(0,v(\theta_0,\widehat{\theta}^{\text{p}})\right),
					\end{equation*}
					where 
					\begin{equation*}
						v(\theta_0,\widehat{\theta}^{\text{p}})=
						\frac{\theta_0^4}{(\widehat{\theta}^{\text{p}})^2} \frac{1}{(e^\alpha-1)^2}
						\left[1+(e^\alpha-1) \frac{\widehat{\theta}^{\text{p}}}{\theta_0}\right]
						\left[e^\alpha-(e^\alpha-1) \frac{\widehat{\theta}^{\text{p}}}{\theta_0}\right].
					\end{equation*}
					\item {\revar If $\widehat{\theta}^{\text{p}} \ge \theta_0$, we have
						$\sqrt{n} \left(\widehat{\theta}_n-\widehat{\theta}^{\text{p}}\right)
						\xrightarrow[\cL]{n \to \infty} 
						\mathcal{N}\left( 0,
						\frac{(\widehat{\theta}^{\text{p}})^2 e^\alpha}{(e^\alpha-1)^2}\right)$.
						%
					}
				\end{enumerate}	
			\end{prop}
			\begin{proof} Point 1. This is obtained by the Law of Large Numbers, $\overline{Z}_n
				\xrightarrow[a.s.]{n\to\infty}  \widetilde{p}_{\theta_0}(1)$ together with the expression of $\widehat{\theta}_n$ and \eqref{eq : tilde p unif estimateur}.
				
				Point 2. With a few lines of algebra, we obtain:
				\[
				\sqrt{n}\left(\widehat{\theta}_n - \theta_0\right) = - \frac{\theta_0 (1 + e^\alpha)}{(1 + e^\alpha)\overline{Z}_n - 1} \sqrt{n}\left(\overline{Z}_n - \widetilde{p}_{\theta_0}(1)\right).
				\]
				Indeed, we start from the expression:
				\[
				\widehat{\theta}_n - \theta_0 = - \frac{\widehat{\theta}^{\text{p}}  (e^\alpha - 1)(1 + e^\alpha)}{\left((1 + e^\alpha)\overline{Z}_n - 1\right)\left((1 + e^\alpha)\widetilde{p}_{\theta_0}(1) - 1\right)} \left( \overline{Z}_n - \widetilde{p}_{\theta_0}(1) \right).
				\]
				Recalling that we are in the case where \( \widehat{\theta}^{\text{p}}  \leq \theta_0 \), which implies
				\[
				(1 + e^\alpha)\widetilde{p}_{\theta_0}(1) - 1 = (e^\alpha - 1) \frac{\hat{\theta}^p}{\theta_0},
				\]
				we directly obtain the desired result. Applying the Central Limit Theorem, we have:
				\[
				\sqrt{n} \left( \overline{Z}_n - \widetilde{p}_{\theta_0}(1) \right) \xrightarrow[\mathcal{L}]{n \to \infty} \mathcal{N} \left( 0, \widetilde{p}_{\theta_0}(1)\left(1 - \widetilde{p}_{\theta_0}(1)\right) \right).
				\]
				After some straightforward computations and substituting \( \widetilde{p}_{\theta_0}(1) \) as in Equation~\eqref{eq : tilde p unif estimateur}, this leads to the desired result regarding the asymptotic distribution of \( \widehat{\theta}_n \).
				
				{\revis  Point 3.  The proof of this point is similar to that of Point 2, noting that in this case, the variance term $v(\widehat{\theta}^{\text{p}},\widehat{\theta}^{\text{p}})$ simplifies to $\frac{(\widehat{\theta}^{\text{p}})^2 e^\alpha}{(e^\alpha-1)^2}$.}
			\end{proof}
			\begin{rem} 	
				{\revis In the setting of Proposition \ref{prop: thetahat}, where $\hat{\theta}^p$ is treated as a fixed preliminary value, the estimation procedure is consistent whenever $\hat{\theta}^p\leq\theta_0$. The variance of the estimator is reduced as $\hat{\theta}^p$ approaches the true value from below.}
				
				When $\alpha \to 0$ we have, {\revarna for $\widehat{\theta}^{\text{p}}\leq \theta_0$,}
				\[
				v\bigl(\theta_0,\widehat{\theta}^{\text{p}}\bigr) \sim 
				\frac{\theta_0^4}{\alpha^2 \bigl(\widehat{\theta}^{\text{p}}\bigr)^2}.
				\]
				In particular, if $\widehat{\theta}^{\text{p}} = \theta_0$, the variance of the estimator is, for $\alpha \to 0$, equivalent to 
				\[
				\frac{\theta_0^2}{\alpha^2},
				\]
				which is exactly the best achievable variance according to Proposition~\ref{prop : Fisher unif}. In practice, it is impossible to choose $\widehat{\theta}^{\text{p}} = \theta_0$, but it is possible to replace it with some preliminary consistent estimator based on a small subset of the data.
			\end{rem}
			{\revar If $\widehat{\theta}^{\text{p}}>\theta_0$, then the estimation procedure
				fails to consistently estimate  $\theta_0$. This situation, in which the initial estimation $\widehat{\theta}^{\text{p}}$ is too high, should be avoided. 
				It is possible to detect this condition from the data. According to \eqref{eq : tilde p unif estimateur}, $P(Z=1)=\widetilde{p}_{\theta_0}=\frac{e^\alpha}{1+e^\alpha}$ when $\widehat{\theta}^{\text{p}}\geq\theta_0$, while $P(Z=1)<\frac{e^\alpha}{1+e^\alpha}$ when $\widehat{\theta}^{\text{p}}<\theta_0$. Hence, it is possible to test the hypothesis $\{ \mathcal{H}_0: nZ_n \text{ is Binomial$(n,\frac{e^\alpha}{1+e^\alpha})$} \}$ against the alternative  
				$\{ \mathcal{H}_1: nZ_n \text{ is Binomial$(n,p)$ with $p<\frac{e^\alpha}{1+e^\alpha}$} \}$.	
				Rejecting $\mathcal{H}_0$ allows to conclude that $\widehat{\theta}^{\text{p}}$ was not chosen too large.}

			To see how well the estimation procedure works in practice, we perform a Monte Carlo experiment for the estimation of $\theta_0=1$, with a sample size of $n=10^3$. The value of $\widehat{\theta}^{\text{p}}$ ranges from $0.5* \theta_0$ to $1.3*\theta_0$. 
			In Figures \ref{fig : mean}--\ref{fig : std}, we draw the value of the empirical mean and the empirical standard deviation as a function of $\widehat{\theta}^{\text{p}}$. The dashed line in
			Figure \ref{fig : std} is the lower bound of the standard deviation as given by the assessment on the Fisher information of Proposition \ref{prop : Fisher unif}. The privacy level is $\alpha=0.3$ and the number of Monte Carlo iterations is $10^6$.
			
			If the value of $\widehat{\theta}^{\text{p}}$ is too small, the standard deviation of the estimator is rather high (see Figure \ref{fig : std}). This standard deviation decreases when $\widehat{\theta}^{\text{p}}$ increases to the true value of the parameter. 
			
			In practice, we see that if $\widehat{\theta}^{\text{p}} \in [0.7*\theta_0, \theta_0]$ the estimator works well and is close to the efficiency.
			If $\widehat{\theta}^{\text{p}} > \theta_0$, the estimator is completely biased as shown in Figure \ref{fig : mean}.
			
			{\revar In Figure \ref{fig : reject prop}, we plot the proportion of rejection of the hypothesis $\mathcal{H}_0$ with a risk set to $0.05$. We see that if we use $\widehat{\theta}^{\text{p}} \le 0.6$ the test procedure rejects with proportion close to $1$ the fact that $\widehat{\theta}^{\text{p}}$ is greater than $\theta_0$.
			}

			\begin{figure}
				\centering
				\begin{minipage}{0.50\textwidth}
					\centering
					\includegraphics[width=0.9\textwidth]{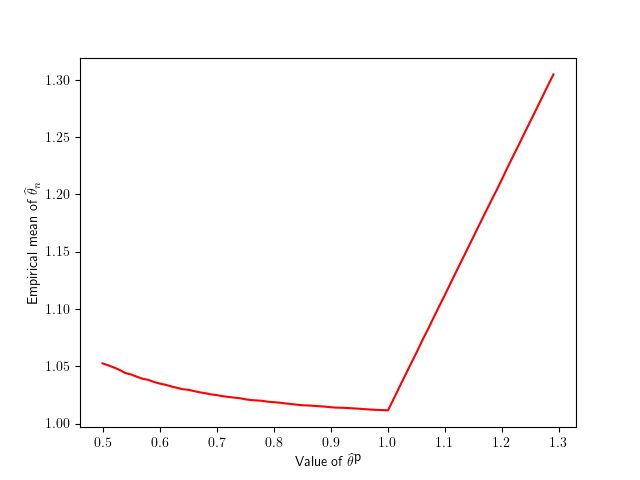} 
					\caption{Empirical Mean \label{fig : mean}}
				\end{minipage}\hfill
				\begin{minipage}{0.50\textwidth}
					\centering
					\includegraphics[width=0.9\textwidth]{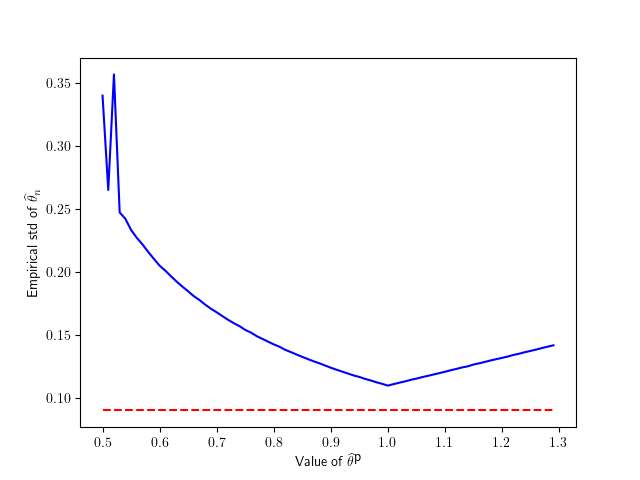} 
					\caption{Empirical Std \label{fig : std}}
				\end{minipage}
			\end{figure}
			\begin{figure}
				\centering
				\includegraphics[width=0.5\textwidth]{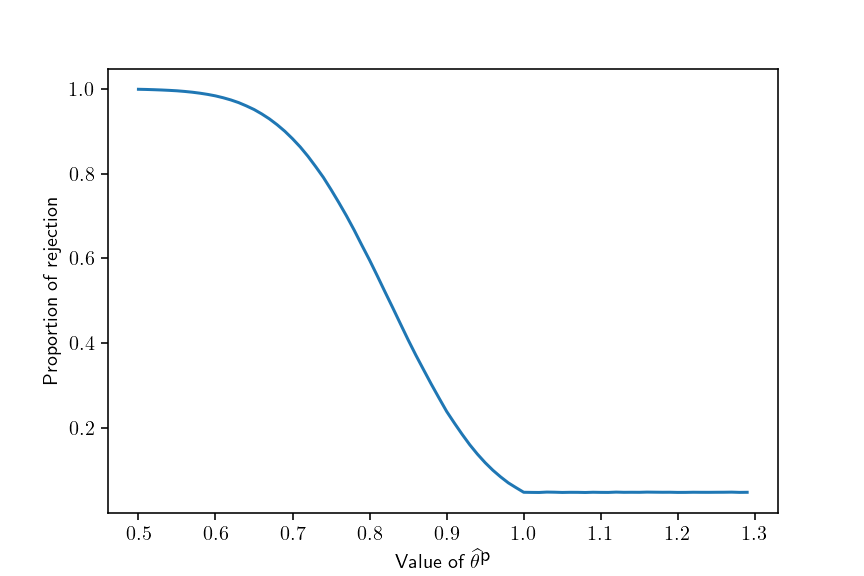} 
				\caption{Empirical rejection ratio \label{fig : reject prop}}
			\end{figure}
			%
			%
			%

			{\revis 
				
				\section{Pathways to practical implementation} \label{S : implementation pathways}
				While this paper focuses on the theoretical foundations of the optimal privacy mechanism and establishes the existence of a solution realizing the maximal Fisher information, extracting the explicit form of this measure directly from the infinite-dimensional optimization problem remains a complex task. However, the theoretical framework developed here is essential, as it provides the rigorous mathematical scaffolding required to eventually construct concrete, implementable mechanisms. 
				
				Building directly on the continuous staircase mechanisms formalized in this work, subsequent research in \cite{SymmetricPrivacy} explores a fundamental symmetry regarding the extremal values $1$ and $e^\alpha$. By leveraging this notion of symmetry, it can be shown that when the optimal measure satisfies this property, the associated Fisher information admits a closed-form expression. More generally, the Fisher information can be structurally decomposed into a symmetric component, which scales as $\alpha^2$, and an asymmetric component, which scales as $\alpha^3$ when $\alpha \to 0$. This reveals that in the high-privacy regime, that is the one we are interested in in this contribution, the asymmetric contribution is negligible, validating the structural importance of symmetric mechanisms. This theoretical symmetry bridges the gap between the abstract existence results (guaranteed by this manuscript) and practical applications, providing a tractable implementation governed by a tuning parameter. This parameter enables a smooth interpolation between the symmetric regime and a fully asymmetric regime (the latter of which recovers the full Fisher information of the non-private model as $\alpha \to \infty$). \\
				\\
				A second avenue for deriving implementable mechanisms lies in the dual formulation of our optimization problem. The optimal privacy mechanism and the maximal Fisher information are characterized by the infinite-dimensional program \eqref{eq : optimisation pb ctn} and its dual form \eqref{eq : dual nos notations}. To approach this numerically, we can follow the finite-dimensional approximation strategy of \cite{Esfahani}. By approximating $f\in L^1(\mathcal{X})$ in a finite-dimensional basis as $f=\sum_{i=1}^n \alpha_i f_i$ (where $(f_i)_{i=1,\dots,n}$ is a collection of $n\ge1$ linearly independent functions), we obtain the semi-infinite linear program:
				\begin{equation} \label{eq : dual finite reduction}
					\inf_{\alpha \in \mathbb{R}^n} \sum_{i=1}^n \alpha_i \int_\mathcal{X} f_i(x) dx, ~ \text{subject to} 
					~
					\sum_{i=1}^n \alpha_i A(f_i)(r) \ge i(r), \text{for all $r\in C$}, \norm{\alpha}_{\mathcal{R}} \le \vartheta,	
				\end{equation}
				where $\norm{\alpha}_{\mathcal{R}}$ is a norm on $\mathbb{R}^n$ ensuring the feasible set is compact, and $\vartheta>0$ and $A$ is an operator that appears naturally in our analysis, alongside with its dual $A^*$ (see Section \ref{S : dual program supp} of the Appendix for details). To formulate the dual program, we define the vector $A_n^*(\mu)=\left[\int_\mathcal{X} A^*(\mu)(x)f_1(x)dx,\dots,\int_\mathcal{X} A^*(\mu)(x)f_n(x)dx\right] \in \mathbb{R}^n$, and the dual norm $\norm{y}_{\mathcal{R}'} = \sup_{\norm{\alpha}_{\mathcal{R}}\le 1} y^\top\alpha$. The dual program of this semi-infinite approximation can then be written as:
				\begin{equation}\label{eq : primal finite reduction}
					\sup_{\mu \ge 0} \int_C i(r) \mu(dr) - \vartheta \norm{A^*_n(\mu)-\bm{c}}_{\mathcal{R}'}
				\end{equation}
				where $\bm{c}=\left[\int_\mathcal{X}f_1(x)dx,\dots,\int_Xf_n(x)dx\right] \in \mathbb{R}^n$, and the penalization term's value is zero if the normalization constraint is respected.
				
				Despite reducing $f$ to a finite-dimensional space, programs \eqref{eq : dual finite reduction} and \eqref{eq : primal finite reduction} are not fully finite-dimensional. The number of constraints in \eqref{eq : dual finite reduction} remains uncountable, and the variable $\mu$ in \eqref{eq : primal finite reduction} still resides in the space of Radon measures on $C$. While \cite{Esfahani} proposes explicit numerical schemes based on constraint randomization and regularization (see also \cite{Nesterov}), extending these theoretical guarantees to our setting is highly non-trivial. Unlike the context of \cite{Esfahani}, our constraints in \eqref{eq : dual finite reduction} are indexed by the infinite-dimensional space $C$, introducing substantial difficulties in assessing numerical approximation errors. We leave the rigorous study and computational deployment of these dual numerical methods for future research.}
			
			
			
			

		}

\begin{appendices}
	{\revis The Appendix is organized as follows. Section~\ref{S: tools} provides the necessary functional-analytic tools for the continuous framework. Section~\ref{s: proof main} presents the proofs of the main results, with two exceptions: the extension of the factorization result to a general $\mathcal{Z}$ is detailed in Section~\ref{ss: factorization Polish Z}, and the results concerning the dual program are provided in Section~\ref{S : dual program supp}. Technical results and auxiliary proofs are deferred to Section~\ref{s: proof technical}. 
}

\section{Functional analysis tools about Choquet's theorem}{\label{S: tools}}
To establish our main results in a continuous space \({X}\), our key tool is Choquet’s theorem, as stated in Theorem \ref{th: Choquet}. In order to apply it within our framework, a solid understanding of certain concepts in functional analysis is required. The purpose of this section is therefore to introduce the necessary background, with a particular focus on the weak-\(\star\) topology.

\subsection{Weak -$\star$ topology}
Let \({X}\) be a Banach space, and let \({X}'\) denote its topological dual—that is, the space of all continuous linear functionals on \({X}\). This space is itself a Banach space, equipped with the norm
\[
| | | f | | | := \sup_{x \neq 0} \frac{\langle f, x \rangle}{\| x \|},
\]
where \(\langle f, x \rangle\) denotes the dual pairing, meaning the image of \(x\) under the linear functional \(f\).

The weak topology on \({X}\) is defined as the coarsest topology that makes all functionals \(f \in {X}'\) continuous. In finite-dimensional settings, the weak and strong topologies coincide, but this is generally not the case in infinite dimensions. The weak topology becomes particularly useful when the dual space \({X}'\) is well understood—for example, in the case of \(L^p(\cX)\) spaces, where \(\cX\) is an open subset of \(\mathbb{R}^d\) equipped with the Lebesgue measure and \(p \in [1, \infty)\).

In this setting, the topological dual of \(L^p(\cX)\) can be identified with \(L^q(\cX)\), where \(q\) is the conjugate exponent of \(p\), i.e., \(\frac{1}{p} + \frac{1}{q} = 1\). Thanks to the Riesz representation theorem (see \cite{Bre83}, p.61), for any \(f \in (L^p(\cX))'\), there exists a unique \(v \in L^q(\cX)\) such that
\[
\langle f, u \rangle = \int_\cX v u, \quad \text{for all } u \in L^p(\cX).
\]

Given a Banach space \({X}\) and its topological dual \({X}'\), one can similarly define a weak topology on \({X}'\). However, we are now interested in introducing a topology on \({X}'\) that is even finer—strictly so when \({X}\) is not reflexive.

\begin{df}{\label{def: weakstar top}}
	The weak-\(\star\) topology on \({X}'\) is the coarsest topology that makes all the maps  
	\begin{align*}
		J_x : \, \, & {X}' \rightarrow \R \\
		& f \mapsto \langle f, x \rangle
	\end{align*} 
	continuous for every \(x \in {X}\).  
\end{df}

As a direct consequence of its definition, the weak-$\star$ topology satisfies the Hausdorff condition. A sequence \((f_n) \subset {X}'\) is said to converge to \(f \in {X}'\) in the weak-$\star$ topology, denoted by
\[
f_n \xrightharpoonup{\star} f \quad \Leftrightarrow \quad \forall x \in {X}, \ \langle f_n, x \rangle \to \langle f, x \rangle.
\]
It is important to observe that strong convergence implies weak convergence, which in turn implies weak-$\star$ convergence. Moreover, by the Banach–Steinhaus theorem, any sequence that converges in the weak-$\star$ topology is necessarily bounded (see p.16 of \cite{Bre83}).

Weak-$\star$ closed sets are, by definition, the complements of weak-$\star$ open sets. However, it is worth noting that sets which are closed in the strong topology — even convex ones — are not necessarily closed in the weak-$\star$ topology. For instance, if \({X}\) is not reflexive and \(\xi \in {X}'' \setminus J({X})\), where $J$ is the canonic injection, then the hyperplane \(\xi^\perp\), which is closed in both the strong and weak topologies, is not weak-$\star$ closed (see p.42 of \cite{Bre83}).

Nevertheless, closed balls in \({X}'\) are weak-$\star$ closed, and even compact in the weak-$\star$ topology (see pp.42–43 of \cite{Bre83}). This property will play a key role in our subsequent analysis. 

\begin{thm}[Banach-Alaoglu]{\label{th: BanAla}}
	The closed unit ball of the dual space of a normed vector space is compact in the weak-$\star$ topology. 
\end{thm}

Note that the classical notion of compactness, defined via the Borel–Lebesgue property, is equivalent to the Bolzano–Weierstrass property (i.e., every bounded sequence admits a convergent subsequence) only in metric spaces. Now, the dual space \({X}'\), equipped with the weak-$\star$ topology, is not metrizable if \({X}\) is infinite-dimensional. However, it is important to highlight that the closed unit ball of \({X}'\) is metrizable whenever \({X}\) is separable—that is, it admits a countable dense subset.

In that case, one can define a compatible metric on the closed unit ball of \({X}'\) by
\begin{equation}{\label{eq: d ball}}
	d(f, g) := \sum_{n = 1}^\infty 2^{-n} \left|\langle f - g, x_n \rangle\right|,   
\end{equation}
where \((x_n)_{n \in \mathbb{N}}\) is a countable dense subset of the closed unit ball of \({X}\) (see pp.48–49 of \cite{Bre83}).

The Banach–Alaoglu theorem is particularly useful in the case where \({X}\) is separable, since it implies that every bounded sequence in \({X}'\) admits a weak-$\star$ convergent subsequence. This applies, for instance, to the case \({X} = L^1(\cX)\), in which any bounded sequence in \(L^\infty(\cX)\) has a subsequence that converges in the weak-$\star$ sense.

\subsection{About Choquet's theorem}
It is well known that the two endpoints of a line segment determine all the points in between: in vector terms, the segment from \( v \) to \( w \) consists of all points of the form \( \lambda v + (1 - \lambda) w \) with \( 0 \leq \lambda \leq 1 \). A classical result by Hermann Minkowski states that in a Euclidean space, any bounded, closed, and convex set \( C \) is the convex hull of its set of extreme points \( \mathcal{E} \). This means that every point \( c \in C \) can be written as a finite convex combination of points \( e \in \mathcal{E} \). The set \( \mathcal{E} \) may be finite or infinite.

In vector terms, assigning non-negative weights \( w(e) \) to elements \( e \in \mathcal{E} \), with almost all weights being zero, any point \( c \in C \) can be represented as
\[
c = \sum_{e \in \mathcal{E}} w(e) e, \qquad \sum_{e \in \mathcal{E}} w(e) = 1,
\]
so that the weights \( w(e) \) define a probability measure supported on a finite subset of \( \mathcal{E} \). For any affine function \( f \) defined on \( C \), its value at \( c \) can then be expressed as
\[
f(c) = \int_{\mathcal{E}} f(e) \, dw(e).
\]

In infinite-dimensional settings, one aims to obtain a similar representation. This is made possible by Choquet's theorem (see p.337 of \cite{Choquet}), which we state below and will be a key tool in our approach.

\begin{thm}[Choquet]{\label{th: Choquet}}
	Let \( C \) be a convex, compact, and metrizable subset of a locally convex Hausdorff topological vector space \( U \). Then:
	\begin{itemize}
		\item The set \( \mathcal{E} \subset C \) of extreme points of \( C \) is a \( G_\delta \)-set (i.e., a countable intersection of open sets).
		\item Every point \( b \in C \) is the barycenter of at least one non-negative Radon measure supported on \( \mathcal{E} \); that is, there exists a Radon probability measure \( \mu \) on \( \mathcal{E} \) such that
		\begin{equation} \label{eq : Choquet representation}
			b = \int_{\mathcal{E}} r \, d\mu(r).
		\end{equation}
		
	\end{itemize}
\end{thm}

In the following, we will apply Choquet's theorem with \( U = L^\infty(\mathcal{X}) \), where \( \mathcal{X} \) denotes the space from which the privacy mechanisms originate and is an open subset of $\R^d$. 

As anticipated in Section \ref{ss: infinite}, in our setting it will be essential to apply Choquet’s theorem to the set
\begin{equation}{\label{eq: def B}}
	\theconvex = \left\{ v : \mathcal{X} \rightarrow \mathbb{R} \, \text{ such that } \, 1 \leq v(x) \leq e^\alpha \text{ a.e.} \right\}. 
\end{equation}
To this end, we must verify that the assumptions of Choquet's theorem are fulfilled for this subset of \( L^\infty(\mathcal{X}) \). Note that \( \theconvex \) can be interpreted as a truncated closed ball of radius \( e^\alpha \), with the unit ball removed. As previously mentioned, although the dual space is not metrizable in general if infinite-dimensional, one can construct a compatible metric on its closed unit ball using the formula in \eqref{eq: d ball}. We begin by showing that a similar construction can be applied to the set \( \theconvex \).

We introduce the notation $\widehat{B}$ for the closed ball, 
\begin{equation}\label{eq: def Ball L infty}
	\widehat{B}=B_{\|\cdot\|_\infty}(0, e^{\alpha}) =\left\{ v : \mathcal{X} \rightarrow \mathbb{R} \, \text{ such that } \, \abs{ v(x)} \leq e^\alpha \text{ a.e.} \right\}.
\end{equation}

\begin{lem}{\label{l: B metrizable}}
	The set $\theconvex$ defined in \eqref{eq: def B} equipped with the weak-$\star$ topology is metrizable by some metric denoted $d_\star$.  
\end{lem}
\begin{proof}
	Recall that \( L^1(\mathcal{X}) \) is a separable space, meaning that it admits a countable dense subset. Let us denote such a sequence by \( (x_n)_{n \in \mathbb{N}} \). Using this sequence, one can construct a compatible metric on the closed ball of radius \( e^\alpha \) in \( L^\infty(\mathcal{X}) \). Analogously to the distance \( d \) introduced in \eqref{eq: d ball} for the closed unit ball, we define:
	\begin{equation}{\label{eq: def dstar}}
		d_\star(f, g) := \sum_{n = 1}^\infty 2^{-n} \left|\langle f - g, x_n \rangle_{{\revarna L^\infty, L^1}}\right|.   
	\end{equation}
	This distance metrizes the closed ball $\widehat{B}$,
	with respect to the weak-\(\star\) topology. Since the set \( \theconvex \) is a subset of $\widehat{B}$, 
	it is metrizable as well.
\end{proof}

We now turn to the proof of the compactness of \( C \). In Lemma \ref{l: B closed}, whose proof can be found in Section \ref{s: proof technical}, we show that \( \theconvex \) is closed with respect to the weak-\(\star\) topology. The compactness of \( \theconvex \) then follows from compactness of $\widehat{B}$ by the Banach–Alaoglu Theorem, as stated in Theorem \ref{th: BanAla}.

\begin{lem}{\label{l: B closed}}
	The set $\theconvex$ defined in \eqref{eq: def B} is a compact with respect to the weak-$\star$ topology.  
\end{lem}

\begin{rem}\label{rem : E not closed}
	We show that the set $\theconvex$ is closed with respect to the weak-$\star$ topology. On the other hand, in general, we cannot expect the set $\cE$ to be closed. For instance, take $\cX = [0,1]$, and for $n \ge 1$, define  
	\[
	r_n = \sum_{k=0}^{n-1} \left( \indi{\left[\frac{2k}{2n}, \frac{2k+1}{2n}\right)} + e^\alpha \, \indi{\left[\frac{2k+1}{2n}, \frac{2k+2}{2n}\right)} \right).
	\]
	The sequence $(r_n)_n$ takes values in $\cE$. However, it can be shown that $r_n \xrightharpoonup{\star} v$, where $v(x) = \frac{1 + e^\alpha}{2}$ for all $x \in \cX$. Since $v \notin \cE$, we conclude that the set $\cE$ is not closed.
	
\end{rem}

The set \( C \), as defined in Equation \eqref{eq: def B}, is also convex. Therefore, by Lemmas \ref{l: B metrizable} and \ref{l: B closed}, we are in a position to apply Choquet’s theorem. In particular, this implies that any \( b \in C \subset L^\infty(\mathcal{X}) \) can be represented as  
\[
b = \int_{\mathcal{E}} r \, d\mu(r),
\]  
for some probability measure \( \mu \) supported on the extreme points of \( C \). This representation is an equality between functions of $L^\infty(\cX)$. 
Our next objective is to give a rigorous meaning to the pointwise representation  
\[
b(x) = \int_{\mathcal{E}} r(x) \, d\mu(r), \quad \text{for } x \in \mathcal{X},
\]  
which will be the focus of the upcoming subsection. Before doing so, let us provide extra details on the extremal points of $\theconvex$, as in the following lemma.

\begin{lem}\label{lem : extremal points}
	The element	$r \in C$ is extremal iff $r(x)\in \{1,e^\alpha\}$ $dx$ almost everywhere. 
\end{lem}
\begin{proof}
	We recall that $r$ is extremal in $\theconvex$ if any decomposition $r = \frac{r_1}{2} + \frac{r_2}{2}$ with $r_1$ and $r_2$ in $\theconvex$ implies $r_1 = r_2 = r$ a.e.. It is then clear that if $r(x) \in \{1, e^\alpha\}$ $dx$-a.e., then $r$ is extremal.  
	
	Conversely, let us show that if $r$ is extremal, it takes a.e. only the two values $1$ and $e^\alpha$. By contradiction, if this is not the case, then there exists some $\varepsilon > 0$ and $\mathcal{N} \subset \mathcal{X}$ such that for $x \in \mathcal{N}$, $r(x) \in [1+\varepsilon, e^\alpha - \varepsilon]$ and $\mathcal{N}$ has positive Lebesgue measure. Then we can split $r$ as
	\[
	r = \frac{1}{2} \bigl( r \indi{\mathcal{N}^c} + (r+\varepsilon) \indi{\mathcal{N}} \bigr)
	+ \frac{1}{2} \bigl( r \indi{\mathcal{N}^c} + (r-\varepsilon) \indi{\mathcal{N}} \bigr)
	= \frac{r_1}{2} + \frac{r_2}{2}.
	\]
	This shows that $r$ is not extremal, arriving at a contradiction.
\end{proof}

\subsection{Point evaluation operator and Bochner integral}{\label{ss:point eval}}
The aim of this subsection is to give a rigorous interpretation to the pointwise representation  
\[
b(x) = \int_{\mathcal{E}} r(x) \, d\mu(r)
\quad \text{for almost every } x \in \mathcal{X} = \R^d.
\]  
To achieve this, the notion of the Bochner integral, introduced in Definition \ref{def: Bocnher} below, will play a key role. 
The measure $\mu$ is supported on $\cE\subset \widehat{B}\subset L^\infty(\cX)$, and the RHS of \eqref{eq : Choquet representation} is an integral of a function valued in the Banach space $U=L^\infty(\cX)$ (see \cite{DieUhl_Book}, \cite{Lang_Book}  for presentation of Banach valued integrals). In particular, the integral can be seen as a Dunford integral on this Banach space (see Lemma 1 p. 52 of \cite{DieUhl_Book}). However, we will precise that this integral is also a Bochner integral for the $d_\star$ distance, meaning that an approximation as given in Definition \ref{def: Bocnher} below is valid. We will deduce a rigorous statement for the pointwise representation.


\begin{df}{\label{def: Bocnher}}
	Let \( (\Omega, \mathcal{F}, \mu) \) be a measure space, where $\mu$ is a sub-probability and let 
	$f : (\Omega, \mathcal{F}) \mapsto (\widehat{B},\Borelof{\widehat{B}})$ be a measurable map, where $\Borelof{\widehat{B}}$ is the Borel $\sigma$-field of $(\widehat{B}, d_\star)$.
	The map $f$ is said to be Bochner integrable if there exists a sequence of integrable simple functions \( (s_n) \) such that
	\[
	\lim_{n \to \infty} \int_\Omega d_\star(f,s_n) d\mu = 0,
	\]
	where the integral on the left-hand side is the usual Lebesgue integral. In this case, the Bochner integral of \( f \) is defined as
	\[
	\int_\Omega f \, d\mu := \lim_{n \to \infty} \int_\Omega s_n \, d\mu,
	\]
	where the convergence holds with the $d_\star$ distance.
\end{df}
%
%

It can be shown that the sequence \( \left\{ \int_{\Omega} s_n \, d\mu \right\}_{n=1}^\infty \) is a Cauchy sequence for $d_\star$ distance 
and hence the limit exists (see Remark \ref{rk: Cauchy} below). Moreover, this limit is independent of the choice of approximating simple functions \( \{s_n\}_{n=1}^\infty \), which ensures that the Bochner integral is well defined. 


An important property of Bochner integral is that, if \( T: U \rightarrow \mathbb{R} \) is a continuous linear operator
then \( T \circ f: \Omega \rightarrow Z \) is also Bochner integrable. In this case, integration and the application of \( T \) can be interchanged:
\begin{equation}{\label{eq: prop Bochner}}
	\int_{\Omega} T f \, d\mu = T \left( \int_{\Omega} f \, d\mu \right).   
\end{equation}
\\
%
We would like to apply the above result to the evaluation operator at a fixed point $x\in \cX$, $f\in U \mapsto f(x)\in \mathbb{R}$. However, such operator is not continuous for the $d_\star$ distance. 

We formally introduce 
as follows the evaluation operator. For any \( x \in \cX \), consider the map
\begin{align}\nonumber 
	e_x : \quad & \widehat{B} \rightarrow \mathbb{R} \\
	\label{eq : def evaluation operator}
	& b \mapsto \limsup_{{\revarna n \to \infty}} 
	\frac{1}{|B(x, {\revarna \frac{1}{n}})|} \int_{B(x, {\revarna \frac{1}{n}})\cap\cX} b(y) \, dy,
\end{align}
where \( B(x, {\revarna \frac{1}{n}}) \) denotes the ball of radius \( {\revarna \frac{1}{n}} \) centered at \( x \), and \( |B(x, {\revarna \frac{1}{n}})| := \int_{{\R}^d} \indi{B(x, {\revarna \frac{1}{n}})}(y) \, dy \) is its Lebesgue measure. 

Note that \( e_x(b) = b(x) \) whenever \( x\in\cX \) is a Lebesgue point of the function \( b \), and consequently for almost every \( x \) with respect to the Lebesgue measure. The reason why the evaluation is defined though the limit \eqref{eq : def evaluation operator} is that it enables to prove the joint measurability $(x,b) \mapsto e_x(b)$ in Proposition \ref{p: evaluation} below. We will also justify that the Bochner integral property stated in \eqref{eq: prop Bochner} applies even without continuity of the evaluation operator. It leads in particular to
\[
b(x) = e_x(b) = e_x\left( \int_\cE r \, d\mu(r) \right) = \int_\cE e_x(r) \, d\mu(r), ~ dx-a.e.,
\]
which is precisely the rigorous pointwise evaluation we aim to justify.


To establish the desired pointwise evaluation, the first step is to show that the integral over the extreme points, as given by Choquet’s theorem, is in fact a Bochner integral. To this end, it is essential to verify that \( \widehat{B} \) is separable, as stated in the following lemma. Its proof can be found in Section \ref{s: proof technical}.

\begin{lem}{\label{l: B separable}}
	The set $\widehat{B}$ defined in \eqref{eq: def Ball L infty} equipped with the weak-$\star$ topology is separable.  
\end{lem}

The separability of \( \widehat{B} \), as established above, plays a crucial role in ensuring that integrals over \( \widehat{B} \) are Bochner integrals. In particular, this means—according to Definition \ref{def: Bocnher}—that they can be approximated by integrals of simple functions.

\begin{lem}{\label{l: Bochner}}
	The set $\widehat{B}$ be the set defined in \eqref{eq: def Ball L infty}
	equipped with the weak-$\star$ topology, and corresponding Borel $\sigma$-field. Let \( \mu \) be a sub-probability measure on \( \widehat{B} \). Then,  \( \int_{\widehat{B}} f  d\mu(f) \) is a Bochner integral.
\end{lem}

The separability and metrizability of the space \( \widehat{B} \), as established
above
enable us to obtain the desired pointwise evaluation, as stated in Proposition \ref{p: evaluation} below. Its proof can be consulted in Section \ref{s: proof technical}.

\begin{prop}{\label{p: evaluation}}
	\begin{enumerate}
		\item The map 
		\begin{align*}
			&\widehat{B}\times \cX \to \R
			\\
			&(f,x) \mapsto e_x(f)
		\end{align*}
		is measurable with respect to the $\sigma$-field $\cB(\widehat{B}) \otimes \cB(\cX)$.
		\item Let $\mu$ be a sub-probability on the space $\widehat{B}$ and assume that $b=\int_{\widehat{B}} f \mu(df)$. Then,
		for almost every $x \in \cX$, 
		\begin{equation}\label{eq : evalutation rigorous}
			b(x) =e_x(b) =\int_{\widehat{B}} e_x(f) d\mu (f). 
		\end{equation}
		\item The $\mu \otimes dx$ measure of the complement of the following set is zero:
		\begin{equation} \label{eq : neg by Fubini}
			\big\{(f,x)\in\widehat{B}\times\cX : e_x(f)  = \lim_{\epsilon \to 0} \frac{1}{|B(x, \epsilon)|} \int_{B(x, \epsilon)\cap\cX} f(y) \, dy\big\}.
		\end{equation}
	\end{enumerate}	
\end{prop}


\begin{rem}{\label{rk: extension functional space}}
	It would be both natural and useful for applications to generalize the results stated for $\mathcal{X} \subset \mathbb{R}^d$ to a broader framework where $\mathcal{X}$ is a topologically complete separable metric space, equipped with a nonatomic, normalized Borel measure $\mu$. This extension is especially relevant for problems involving the statistical analysis of stochastic processes driven by Brownian motion under local differential privacy constraints.
	
	While we leave this question for future work, we conjecture that our main results extend to this more general setting. A key technical ingredient in our proofs involving the evaluation operator is the property that almost every point in $\mathcal{X}$ is a Lebesgue point for functions in a suitable class $B \subset \mathcal{X}$. This is straightforward when $\mathcal{X} = \mathbb{R}^d$ (see Step 1 in the proof of Proposition~\ref{p: evaluation}), but is less immediate in an abstract measurable space.
	
	Fortunately, Theorem 2 in \cite{Oxt} ensures the existence of a homeomorphism between a $G_\delta$ subset of $\mathcal{X}$ (with full $\mu$-measure) and the interval $(0,1)$ with the Lebesgue measure. This suggests that the analysis of the evaluation operator, and in particular of Lebesgue points, can be transferred to $(0,1)$ via this homeomorphism, since it is continuous in both directions. 
	
	We believe this intuition supports the conjectured extension, though several technical challenges remain. A detailed treatment of these issues lies beyond the scope of the present work. 
	
\end{rem}

\section{Proof of main results}{\label{s: proof main}}

This section is dedicated to the proofs of our main results, specifically the bounds on the Fisher information in both the discrete and continuous settings, as presented in Sections \ref{ss: finite} and \ref{ss: infinite}, respectively.

\subsection{Discrete case: proof of Theorem \ref{th : info cas fini}} \label{Ss: appendix discrete}

Let us begin by introducing some notation that will be useful throughout the proof.
For each $x \in \mathcal{X}$, define the index sets
\begin{equation}{\label{eq: set Qx+-}}
	Q^+_x = \{ \beta \in \{0, \dots, 2^d - 1\} \mid r_\beta(x) = e^\alpha \}, \quad 
	Q^-_x = \{ \beta \in \{0, \dots, 2^d - 1\} \mid r_\beta(x) = 1 \}.
\end{equation}
With this notation, we can express the mechanism $q^{(1)}$ as
\begin{equation}{\label{eq: mechanism q1}}
	q^{(1)}_x(\beta) = \omega_\beta r_\beta(x) = \omega_\beta \left( e^\alpha \indi{Q^+_x}(\beta) + \indi{Q^-_x}(\beta) \right).
\end{equation}

We are now ready to proceed with the proof of Theorem \ref{th : info cas fini}. The result follows directly from the upper and lower bounds stated in Propositions \ref{prop : borne sup Fisher fini} and \ref{prop : borne inf Fisher fini}, respectively.

We begin with the more involved part—establishing the upper bound—as detailed in Proposition \ref{prop : borne sup Fisher fini}.

\begin{prop}\label{prop : borne sup Fisher fini}
	Assume that $\Theta \subset \R$ and that $(p_{\theta})_\theta$ 
	is DQM at $\theta_0 \in \overset{\circ}{\Theta}$ with score function $s_{\theta_0}$. Let $q : \cX \to {\revarna\proba\left(\cZ\right)}$ be any $\alpha$-LDP mechanism taking values in a finite set $\cZ$
	and $(\widetilde{p}_\theta)_\theta$ the model for the public variable.
	Assume that $s_{\theta_0}(x)\neq 0$ for all $x\in \cX$.
	Then, there exists $\overline{\alpha}$ such that if $\alpha \le \overline{\alpha}$,
	\begin{equation} \label{eq :  uppper bound Fisher finite exacte}
		\cI_{\theta_0}((\widetilde{p}_\theta)_\theta)=\E\left[\abs*{t_{\theta_0}(Z)}^2\right] \le 
		\cI^{\text{max},\alpha}_{\theta_0}((p_\theta)_\theta),
	\end{equation}
	where we recall that
	\begin{equation*}
		\cI^{\text{max},\alpha}_{\theta_0}((p_\theta)_\theta)
		=\frac{(e^\alpha-1)^2}{4}%
		\times \frac{\E\left[\abs{s_{\theta_0}(X)}\right]^2}{\left[ (1-n_\text{max})+e^\alpha n_\text{max}  \right]
			\left[n_\text{max}+(1-n_\text{max})e^\alpha  \right]}
	\end{equation*}
	and $n_\text{max}=\int_{x: s_{\theta_0}(x)>0} p_{\theta_0}(x) {\revarna \nu_\cX}(dx)$.
	The value of $\overline{\alpha}$ is independent of the channel $q$, but depends on the model $(p_{\theta})_\theta$.
\end{prop}	

\begin{proof}
	As \( q \in \cQ_\alpha \), we use the factorisation Lemma~\ref{lem : factorisation cas fini} to write \( q = q^{(2)} \circ q^{(1)} \), where \( q^{(1)} \) is an extremal \(\alpha\)-LDP channel described in the statement of Lemma~\ref{lem : factorisation cas fini}. From \cite{steinbergerEfficiencyLocalDifferent}, and as recalled in~\eqref{eq : decroissance Fisher}, we have
	\[
	\cI_{\theta_0}\big((\widetilde{p}_\theta)_\theta\big)
	= \cI_{\theta_0}(q \circ \cP)
	= \cI_{\theta_0}\big((q^{(2)} \circ q^{(1)}) \circ \cP\big)
	= \cI_{\theta_0}\big(q^{(2)} \circ (q^{(1)} \circ \cP)\big)
	\le \cI_{\theta_0}(q^{(1)} \circ \cP).
	\]
	Hence, it is sufficient to prove~\eqref{eq :  uppper bound Fisher finite exacte} for \( q = q^{(1)} \). To lighten notation, we will denote this quantity as \( q \) in place of \( q^{(1)} \) in the remainder of the proof. Accordingly, we consider the extremal mechanism $q : \cX \to {\revarna \proba\left(\cE\right)}$  {\revarna where $\cE = \{0,\dots,2^d-1\} $,} with probability coefficients
	\begin{equation} \label{eq: qx beta 7.25}
		q_x(\beta) = \omega_\beta r_\beta(x),
	\end{equation}
	as in Lemma~\ref{lem : factorisation cas fini}. Recall from~\eqref{eq: mass omega beta} and~\eqref{eq: r beta x} that
	\[
	\sum_{\beta=0}^{2^d-1} \omega_\beta \in [e^{-\alpha}, 1],
	\quad \text{and} \quad
	r_\beta(x) = e^{\alpha} \indi{F^+_\beta}(x) + \indi{F^-_\beta}(x) = 1 + (e^\alpha - 1)\indi{F^+_\beta}(x).
	\]
	In particular, the expression for the probability \( \widetilde{p}_\theta \) on \( \cE \) is given by
	\begin{multline}
		\widetilde{p}_\theta(\beta)
		= \int_\cX p_\theta(x) q_x(\beta) \, {\revarna \nu_\cX}(dx)
		= \int_\cX p_\theta(x) \omega_\beta r_\beta(x) \, {\revarna \nu_\cX}(dx)
		\\
		= \omega_\beta \int_\cX p_\theta(x) \left[1 + (e^\alpha - 1)\indi{F^+_\beta}(x)\right] {\revarna \nu_\cX}(dx),
	\end{multline}
	where in the last equality we replaced \( r_\beta(x) \) using~\eqref{eq: r beta x}, {\revarna and we recall that $\nu_\cX$ denotes the counting measure on $\cX$}. Since \( p_\theta \) is a probability density, we have \( \int_\cX p_\theta(x) {\revarna \nu_\cX}(dx) = 1 \), so that
	\begin{equation} \label{eq: ptilde 7.5}
		\widetilde{p}_\theta(\beta) = \omega_\beta \left(1 + (e^\alpha - 1) n^+_\beta\right),
	\end{equation}
	with
	\begin{equation} \label{eq : def nplus}
		n^+_\beta = \int_{F^+_\beta} p_\theta(x) {\revarna \nu_\cX}(dx).
	\end{equation}
	
	Recall that, according to~\cite{steinbergerEfficiencyLocalDifferent} and as stated in~\eqref{eq : score public}, the score function of the model \( q \circ \cP \) is given by
	\[
	t_\theta(\beta) = \frac{\int_\cX s_\theta(x) q_x(\beta) p_\theta(x) {\revarna \nu_\cX}(dx)}{\widetilde{p}_\theta(\beta)}.
	\]
	We already obtained an expression for the denominator \( \widetilde{p}_\theta(\beta) \) in~\eqref{eq: ptilde 7.5}. For the numerator, using~\eqref{eq: qx beta 7.25} and the definition of \( r_\beta(x) \), we get
	\begin{align*}
		\int_\cX s_\theta(x) q_x(\beta) p_\theta(x) {\revarna \nu_\cX}(dx)
		&= \int_\cX s_\theta(x) \omega_\beta r_\beta(x) p_\theta(x) {\revarna \nu_\cX}(dx)
		\\&= \omega_\beta \int_\cX s_\theta(x) \left[1 + (e^\alpha - 1)\indi{F^+_\beta}(x)\right] p_\theta(x) {\revarna \nu_\cX}(dx).
	\end{align*}
	Since the score function \( s_\theta \) of \( \mathcal{P} \) is centered, we have \( \int_\cX s_\theta(x) p_\theta(x) {\revarna \nu_\cX}(dx) = 0 \). Thus,
	\[
	\int_\cX s_\theta(x) q_x(\beta) p_\theta(x) {\revarna \nu_\cX}(dx)
	= (e^\alpha - 1) \omega_\beta \int_\cX s_\theta(x) \indi{F^+_\beta}(x) p_\theta(x) {\revarna \nu_\cX}(dx).
	\]
	This yields
	\begin{equation} \label{eq: t beta 8.5}
		t_\theta(\beta) = (e^\alpha - 1)
		\frac{\int_\cX s_\theta(x) \indi{F^+_\beta}(x) p_\theta(x) {\revarna \nu_\cX}(dx)}{1 + (e^\alpha - 1) n^+_\beta},
	\end{equation}
	for all \( \beta \) such that \( \omega_\beta > 0 \). Substituting~\eqref{eq: ptilde 7.5} and~\eqref{eq: t beta 8.5}, we obtain the following expression for the Fisher information of \( q \circ \cP \):
	\begin{align*}
		\cI_\theta(q \circ \cP)
		&= \mathbb{E}[t_\theta(Y)^2]
		= \int_\mathcal{E} \widetilde{p}_\theta(\beta) \, t_\theta(\beta)^2 \, \nu(d\beta) \\
		&= (e^\alpha - 1)^2 \int_\mathcal{E} \omega_\beta
		\frac{\left(\int_\cX s_\theta(x) \indi{F^+_\beta}(x) p_\theta(x) {\revarna \nu_\cX}(dx)\right)^2}{1 + (e^\alpha - 1) n^+_\beta} \, \nu(d\beta),
	\end{align*}
	where \( \nu \) is the counting measure on \( \mathcal{E} = \{0, \dots, 2^d - 1\} \). Let us recall
	\begin{equation} \label{Eq : i beta suppA rappel}
		i_\beta := \frac{\left(\int_\cX s_\theta(x) \indi{F^+_\beta}(x) p_\theta(x) {\revarna \nu_\cX}(dx)\right)^2}{1 + (e^\alpha - 1) n^+_\beta},
	\end{equation}
	and consider the maximization problem:
	\[
	M^* = \sup_{\omega} \int_\mathcal{E} \omega_\beta i_\beta \, \nu(d\beta)
	\quad \text{subject to} \quad
	\int_\mathcal{E} \omega_\beta r_\beta(x) \, \nu(d\beta) = 1, \quad \forall x \in \cX,
	\]
	where the supremum is over sub-probability weights \( (\omega_\beta)_{\beta = 0, \dots, 2^d - 1} \) on \( \{0, \dots, 2^d - 1\} \). The constraint arises from~\eqref{eq: constraint factorisation 4.55}. We conclude that
	\[
	\cI_\theta(q \circ \cP) \le (e^\alpha - 1)^2 M^*.
	\]
	Thus, the proof of Proposition~\ref{prop : borne sup Fisher fini} is complete once we invoke Lemma~\ref{lem :  determination max fini} below, whose proof can be found in Section \ref{s: proof technical}.
	
\end{proof}

\begin{lem}\label{lem :  determination max fini}
	Under the assumptions of Proposition~\ref{prop : borne sup Fisher fini}, consider the maximization problem
	\begin{equation} \label{eq : optimization pb finite}
		M^* = \sup_{\omega \in (\R_+)^{2^d}} \int_{\mathcal{E}} \omega_\beta i_\beta \, \nu(d\beta) \quad
		\text{subject to} \quad \int_{\mathcal{E}} \omega_\beta r_\beta(x) \, \nu(d\beta) = 1, \quad \forall x \in \cX.
	\end{equation}
	Then, there exists $\overline{\alpha} > 0$ such that for all $\alpha \leq \overline{\alpha}$,
	\begin{equation} \label{eq : def M star}
		M^* = \frac{1}{4} \cdot \frac{\E\left[ |s_{\theta_0}(X)| \right]^2}{\left[(1 - n_{\max}) + e^\alpha n_{\max} \right] \left[ n_{\max} + (1 - n_{\max}) e^\alpha \right]},
	\end{equation}
	where $n_{\max}$ is as defined in~\eqref{eq: def nmax} in the statement of Theorem~1. Moreover, the maximum is attained for a vector $\overline{\omega}$ with only two non-zero components.
\end{lem}

This concludes the part of the proof dedicated to the upper bound in Theorem~\ref{th : info cas fini}.  
We now proceed with the proof of the lower bound, as stated in Proposition~\ref{prop : borne inf Fisher fini} below.  
Note that this proof heavily relies on the notation and arguments already introduced in the upper bound analysis.

\begin{prop}\label{prop : borne inf Fisher fini}
	Assume that $\Theta \subset \R$ and that $\cP=(p_{\theta})_\theta$ 
	is DQM at $\theta_0 \in \overset{\circ}{\Theta}$ with score function $s_{\theta_0}$.
	Then for all $\alpha>0$, there exists $q \in \cQ_\alpha$, such that $\cI_{\theta_0}(q\circ \cP)=\cI_{\theta_0}^{\text{max},\alpha}$.
\end{prop}
\begin{proof}
	The randomization $q$ achieving the maximum is the one discussed in Remark \ref{rk: choice Fmax} and used in the proof of Lemma \ref{lem :  determination max fini},  where the maximization problem \eqref{eq : optimization pb finite} is solved.
	For the sake of completeness, we detail it here using simplified notation.
	
	Since the support of the optimal measure $\overline{\omega}$ in Lemma~\ref{lem :  determination max fini} has cardinality $2$, we let $\cZ = \{z_1, z_2\}$. Define
	\[
	F^+ := \{x \in \cX \mid s_{\theta_0}(x) > 0\}.
	\]
	We now define the randomization:
	\begin{equation}\label{eq: optimal q 49.5}
		q_x(z) = P(Z = z \mid X = x) =
		\begin{cases}
			\displaystyle \frac{1}{1 + e^\alpha} \left( e^\alpha \indi{z_1}(z) + \indi{z_2}(z) \right) & \text{if } x \in F^+, \\
			\displaystyle \frac{1}{1 + e^\alpha} \left( \indi{z_1}(z) + e^\alpha \indi{z_2}(z) \right) & \text{if } x \notin F^+.
		\end{cases}
	\end{equation}
	After computation, for any $\theta \in \Theta$, the marginal probability of $Z$ is given by
	\begin{align*}
		\widetilde{p}_{\theta}(z_1) 
		&= P_\theta(Z = z_1) = \int_\cX q_x(z_1) p_\theta(x) {\revarna \nu_\cX}(dx) \\
		&= \frac{1}{1 + e^\alpha} \int_\mathcal{X} \left( e^\alpha \indi{F^+}(x) + \indi{(F^+)^c}(x) \right) p_\theta(x) {\revarna \nu_\cX}(dx) \\
		&= \frac{1}{1 + e^\alpha} + \frac{e^\alpha - 1}{1 + e^\alpha} \int_{F^+} p_\theta(x) {\revarna \nu_\cX}(dx),
	\end{align*}
	while $\widetilde{p}_{\theta}(z_2) = 1 - \widetilde{p}_{\theta}(z_1)$. From \eqref{eq : score public} and using the definition \eqref{eq: optimal q 49.5} of $q_x(z)$, the score function of the public data is:
	\begin{align*}
		t_{\theta_0}(z_1) 
		&= \frac{1}{\widetilde{p}_{\theta_0}(z_1)} \int_\cX s_{\theta_0}(x) \, q_x(z_1) \, p_{\theta_0}(x) {\revarna \nu_\cX}(dx) \\
		&= \frac{1}{\widetilde{p}_{\theta_0}(z_1)} \cdot \frac{1}{1 + e^\alpha} \int_\cX s_{\theta_0}(x) \left( e^\alpha \indi{F^+}(x) + \indi{(F^+)^c}(x) \right) p_{\theta_0}(x) {\revarna \nu_\cX}(dx) \\
		&= \frac{1}{\widetilde{p}_{\theta_0}(z_1)} \cdot \frac{e^\alpha - 1}{1 + e^\alpha} \int_{F^+} s_{\theta_0}(x) p_{\theta_0}(x) {\revarna \nu_\cX}(dx).
	\end{align*}
	Similarly,
	\[
	t_{\theta_0}(z_2) = \frac{1}{\widetilde{p}_{\theta_0}(z_2)} \cdot \frac{e^\alpha - 1}{1 + e^\alpha} \int_{(F^+)^c} s_{\theta_0}(x) p_{\theta_0}(x) {\revarna \nu_\cX}(dx).
	\]
	Recalling that 
	\[
	n_{\max} := \int_{F^+} p_{\theta_0}(x) {\revarna \nu_\cX}(dx), \quad 
	\int_{F^+} s_{\theta_0}(x) p_{\theta_0}(x) {\revarna \nu_\cX}(dx) = \frac{1}{2} \int_\cX \abs{s_{\theta_0}(x)} p_{\theta_0}(x) {\revarna \nu_\cX}(dx),
	\]
	we obtain, after straightforward algebra,
	\begin{align}\label{eq: end lower}
		\cI_{\theta_0}((\widetilde{p}_\theta)_\theta) 
		&= \sum_{i = 1, 2} \widetilde{p}_{\theta_0}(z_i) \, t_{\theta_0}(z_i)^2 \notag \\
		&= \frac{(e^\alpha - 1)^2}{(1 + e^\alpha)^2} \cdot \frac{1}{4} \left( \int_\cX \abs{s_{\theta_0}(x)} p_{\theta_0}(x) \mu(dx) \right)^2 \left( \frac{1}{\widetilde{p}_{\theta_0}(z_1)} + \frac{1}{\widetilde{p}_{\theta_0}(z_2)} \right) \notag \\
		&= \frac{(e^\alpha - 1)^2}{4} \cdot \frac{\E[\abs{s_{\theta_0}(X)}]^2}{(1 + e^\alpha)^2} \cdot \frac{1}{\widetilde{p}_{\theta_0}(z_1)(1 - \widetilde{p}_{\theta_0}(z_1))}.
	\end{align}
	Moreover, observe that
	\[
	\widetilde{p}_{\theta_0}(z_1)(1 - \widetilde{p}_{\theta_0}(z_1)) = 
	\frac{1}{1 + e^\alpha} \left( 1 + (e^\alpha - 1) n_{\max} \right)
	\left( 1 - \frac{1}{1 + e^\alpha} (1 + (e^\alpha - 1) n_{\max}) \right),
	\]
	which simplifies to
	\[
	\frac{1}{(1 + e^\alpha)^2} \left( 1 + (e^\alpha - 1) n_{\max} \right) \left( e^\alpha - (e^\alpha - 1) n_{\max} \right).
	\]
	Plugging this into the denominator of \eqref{eq: end lower} yields the desired expression \eqref{eq : def I max caf fini}, thus concluding the proof.
\end{proof}	

\subsection{Continuous case: proof of results in Section \ref{ss: infinite}}{\label{ss: proof cont}}

As in the proof of the discrete case, we begin by introducing some notation that will be useful throughout the proofs in this subsection.

For each $x \in \mathcal{X}$, define the sets
\[
Q^+_x := e_x^{-1}(\{e^\alpha\}) = \{ r \in \mathcal{E} \mid e_x(r) = e^\alpha \}, \quad
Q^-_x := e_x^{-1}(\{1\}) = \{ r \in \mathcal{E} \mid e_x(r) = 1 \}.
\]
Since the map $r \mapsto e_x(r)$ is measurable for each $x$, both $Q^+_x$ and $Q^-_x$ are Borel sets. These sets are the continuous counterparts of those defined in \eqref{eq: set Qx+-} in the previous subsection. With this notation, the kernel defined in \eqref{eq : def random generic extr ctn} can be written as
\[
q^{(\mu)}_x(dr) = \left( e^\alpha \indi{Q^+_x}(r) + \indi{Q^-_x}(r) \right) \mu(dr), \quad \text{for all } x \in \mathcal{X},
\]
and the normalization condition \eqref{eq : cond norm mu continue} becomes
\[
e^\alpha \mu(Q^+_x) + \mu(Q^-_x) = 1 \quad \text{for almost every } x \in \mathcal{X}.
\]
This construction plays the same role as the mechanism $q^{(1)}$ in the discrete case, defined in \eqref{eq: mechanism q1}.

With these definitions in place, we can now proceed to establish the upper and lower bounds stated in Theorem~\ref{thm : equivalent alpha petit ctn}.

\subsubsection{Proof of Theorem \ref{thm : equivalent alpha petit ctn}}
\begin{proof}
	We begin by proving the upper bound on $\cJ^{\text{max},\alpha}_{\theta_0}$. Let $q \in \cQ_\alpha$. By the factorization result in Proposition~\ref{prop : factor ctn}, any such $q$ can be written as a composition $q = q^{(2)} \circ q^{(\mu)}$, where $\mu$ is a sub-probability measure on $\cE$ satisfying the normalization condition \eqref{eq : cond norm mu continue}, and $q^{(2)}$ is a Markov kernel from $\cE$ to $\cZ$. According to \eqref{eq: fact on Fisher cont}, the Fisher information of the model $q \circ \cP$ is upper bounded by the Fisher information of $q^{(\mu)} \circ \cP$. Therefore, it is sufficient to prove the desired upper bound for $\cI_{\theta_0}(q^{(\mu)} \circ \cP)$. Using \eqref{eq : Fisher ctn general_preli}, we obtain
	\begin{equation} \label{eq : majo I in prop alpha small ctn}
		\cI_{\theta_0}(q \circ \cP) \le (e^\alpha - 1)^2
		\int_\cE \frac{\left( \int_{F^+_r} s_{\theta_0}(x) p_{\theta_0}(x) dx \right)^2}
		{\widetilde{p}_{\theta_0}(r)} \mu(dr).
	\end{equation}
	Now, define the sets $F^+_{\text{max}} := \{ x \in \mathcal{X} \mid s_{\theta_0}(x) > 0 \}$ and $F^{\prime +}_{\text{max}} := \{ x \in \mathcal{X} \mid s_{\theta_0}(x) \le 0 \}$, so that $\mathcal{X} = F^+_{\text{max}} \cup F^{\prime +}_{\text{max}}$. 
	{\revis For any Borel set $F \subset \mathcal{X}$ we have}
	\[
	\left| \int_F s_{\theta_0}(x) p_{\theta_0}(x) dx \right|
	\le \frac{1}{2} \int_{\mathcal{X}} |s_{\theta_0}(x)| p_{\theta_0}(x) dx
	= \left| \int_{F^+_{\text{max}}} s_{\theta_0}(x) p_{\theta_0}(x) dx \right|
	= \left| \int_{F^{\prime +}_{\text{max}}} s_{\theta_0}(x) p_{\theta_0}(x) dx \right|.
	\]
	Inserting this bound into \eqref{eq : majo I in prop alpha small ctn}, we find
	\begin{equation*}
		\cI_{\theta_0}(q \circ \cP) \le
		(e^\alpha - 1)^2
		\frac{ \left( \int_{\mathcal{X}} |s_{\theta_0}(x)| p_{\theta_0}(x) dx \right)^2 }{4}
		\int_\cE \frac{1}{\widetilde{p}_{\theta_0}(r)} \mu(dr).
	\end{equation*}
	Recalling from \eqref{eq : tilde p ctn} that $\widetilde{p}_{\theta_0}(r) \ge 1$ for all $r \in \cE$ (since $r(x) \ge 1$), and using that $\mu$ is a sub-probability measure, we get
	\[
	\int_\cE \frac{1}{\widetilde{p}_{\theta_0}(r)} \mu(dr) \le \mu(\cE) \le 1.
	\]
	This establishes the upper bound in \eqref{eq :  encadrement Fisher ctn small alpha}.
	
	\medskip
	
	We now turn to the lower bound. It is sufficient to construct an $\alpha$-LDP mechanism whose Fisher information exceeds the lower bound in \eqref{eq : encadrement Fisher ctn small alpha}. Define the two extremal elements
	\[
	r_1 := 1 + (e^\alpha - 1) \indi{F^+_{\text{max}}}, \qquad
	r_2 := 1 + (e^\alpha - 1) \indi{F^{\prime +}_{\text{max}}},
	\]
	which clearly belong to $\mathcal{E}$. Mimicking the construction in the discrete case (see Equation~\eqref{eq: optimal q 49.5}), we define the randomized mechanism $q^{(\mu)}$ with
	\begin{equation} \label{eq: qmu 28.5}
		\mu := \frac{1}{1 + e^\alpha} \left( \delta_{r_1} + \delta_{r_2} \right).
	\end{equation}
	We verify that this mechanism satisfies the normalization condition \eqref{eq : cond norm mu continue}. Since $e_x(r_j) = r_j(x)$ a.e.\ for $j=1,2$ (as established in Subsection~\ref{ss:point eval}), it suffices to check that
	\[
	\int_\cE e_x(r) \mu(dr)=\frac{1}{1 + e^\alpha} \sum_{j=1}^2 r_j(x) = 1 \quad \text{a.e.},
	\]
	which follows directly from the disjointness of $F^+_{\text{max}}$ and $F^{\prime +}_{\text{max}}$ and their union being the whole space $\mathcal{X}$ up to null sets. Note that for $j=1,2$,
	\[
	\int_{\mathcal{X}} s_{\theta_0}(x) r_j(x) p_{\theta_0}(x) dx
	= \int_{\mathcal{X}} s_{\theta_0}(x) p_{\theta_0}(x) dx +
	(e^\alpha - 1) \int_{F^{+}_{\text{max}} \text{ or } F^{\prime +}_{\text{max}}} s_{\theta_0}(x) p_{\theta_0}(x) dx,
	\]
	and the first term is zero by centering of $s_{\theta_0}$. The Fisher information of this mechanism therefore is, by \eqref{eq : Fisher ctn general_preli},
	\begin{align*}
		\cI_{\theta_0}(q^{(\mu)} \circ \cP)
		&= \frac{(e^\alpha - 1)^2}{1 + e^\alpha}
		\left[
		\frac{\left( \int_{F^+_{\text{max}}} s_{\theta_0}(x) p_{\theta_0}(x) dx \right)^2}
		{\widetilde{p}_{\theta_0}(r_1)} +
		\frac{\left( \int_{F^{\prime +}_{\text{max}}} s_{\theta_0}(x) p_{\theta_0}(x) dx \right)^2}
		{\widetilde{p}_{\theta_0}(r_2)}
		\right].
	\end{align*}
	Moreover, $r_j(x) \le e^\alpha$ implies $\widetilde{p}_{\theta_0}(r_j) \le e^\alpha$ for $j = 1,2$. Since the two integrals in the numerator are equal in absolute value and each equals
	\[
	\frac{1}{2} \int_{\mathcal{X}} |s_{\theta_0}(x)| p_{\theta_0}(x) dx,
	\]
	we deduce the lower bound
	\begin{equation*}
		\cI_{\theta_0}(q^{(\mu)} \circ \cP) \ge
		\frac{(e^\alpha - 1)^2}{1 + e^\alpha} \cdot \frac{2}{e^\alpha}
		\left( \frac{1}{2} \int_{\mathcal{X}} |s_{\theta_0}(x)| p_{\theta_0}(x) dx \right)^2.
	\end{equation*}
	This concludes the proof of the lower bound in the theorem.
	
\end{proof}

\subsubsection{Proof of Proposition \ref{prop : maximisation pb convex multidim}}
\begin{proof}
	Since the mapping $r \mapsto j(r)$ is bounded {\revar and $\ell$ is continuous,} 
	we deduce that $J^\star < \infty$. Moreover, the set of Radon sub-probability measures $\mu$ satisfying the constraint in \eqref{eq : optimisation pb ctn J} is non-empty. One example is $\mu$ being the Dirac mass on the function $r_0$ defined by $r_0(x)=1$ for all $x\in \cX$. 
	Hence, there exists a sequence of sub-probabilities $\mu_n$ on $\mathcal{E}$ satisfying the constraint in \eqref{eq : optimisation pb ctn J} and such that
	\[
	{\revar \ell \left( \int_\mathcal{E} j(r) \mu_n(dr) \right)} \xrightarrow{n \to \infty} J^\star.
	\]
	Since each $\mu_n$ is supported on $\mathcal{E}$, we have $\mu_n(\theconvex) = \mu_n(\mathcal{E}) \leq 1$. As $\theconvex$ is compact for the weak-$\star$ topology, the sequence $(\mu_n)_n$ is tight and thus admits a converging subsequence. Let $k : \mathbb{N} \to \mathbb{N}$ be a strictly increasing function such that $(\mu_{k(n)})_n$ converges weakly to some Radon sub-probability measure $\mu^\star$ on $\theconvex$.

	The main difficulty in the proof is that although $\mu_n$ are supported on $\mathcal{E}$, the limit $\mu^\star$ is not necessarily supported on $\mathcal{E}$, since $\mathcal{E}$ is not closed. We will show that the existence of an optimal measure $\overline{\mu}$ supported on $\mathcal{E}$ follows from a convexity argument.

	First, we show that $\mu^\star$ inherits properties from the linear constraints in \eqref{eq : optimisation pb ctn J}. Let $h \in L^1(\mathcal{X})$. Then, the map 
	\[
	r \in \theconvex \mapsto \int_\mathcal{X} h(x) r(x) dx
	\]
	is continuous for the weak-$\star$ topology and bounded. Hence, the weak convergence of $\mu_{k(n)}$ towards $ \mu^\star$ implies
	\begin{equation} \label{eq : proof closeness contraint}
		\int_\theconvex \left( \int_\mathcal{X} h(x) r(x) dx \right) \mu_{k(n)}(dr) \xrightarrow{n \to \infty} \int_\theconvex \left( \int_\mathcal{X} h(x) r(x) dx \right) \mu^\star(dr).
	\end{equation}
	Using that $e_x(r) = r(x)$ a.e., the left-hand side becomes by Fubini theorem,
	\[
	\int_\mathcal{X} h(x) \left( \int_\theconvex e_x(r) \mu_{k(n)}(dr) \right) dx.
	\]
	By the constraint $\int_\mathcal{E} e_x(r) \mu_{k(n)}(dr) = 1$ a.e.\ in $x$, this simplifies to $\int_\mathcal{X} h(x) dx$. Similarly, the right-hand side of \eqref{eq : proof closeness contraint} becomes
	\[
	\int_\mathcal{X} h(x) \left( \int_\theconvex e_x(r) \mu^\star(dr) \right) dx.
	\]
	Therefore,
	\[
	\int_\mathcal{X} h(x) dx = \int_\mathcal{X} h(x) \left( \int_\theconvex e_x(r) \mu^\star(dr) \right) dx, \quad \text{for all } h \in L^1(\mathcal{X}),
	\]
	implying that
	\[
	\int_\theconvex e_x(r) \mu^\star(dr) = 1 \quad \text{a.e.\ in } x.
	\]
	Since $r \mapsto j(r)$ is continuous and bounded, we have
	\[
	{\revar
		\lim_{n \to \infty} \int_\theconvex j(r) \mu_{k(n)}(dr) = \int_\theconvex j(r) \mu^\star(dr).
	}
	\]
	{\revar Using the continuity of $\ell$, we deduce 
		\begin{equation*}
			J^\star= 	\lim_{n \to \infty} \ell \left( \int_\theconvex j(r) \mu_{k(n)}(dr) \right) =
			\ell \left( \int_\theconvex j(r) \mu^\star(dr) \right).	
		\end{equation*}
	}
	Consequently, $\mu^\star$ is a solution to the optimization problem \eqref{eq : optimisation pb ctn J} when integrals are taken over the convex set $\theconvex$ instead of its extremal points $\mathcal{E}$.
	
	\medskip
	
	We now construct a measure $\overline{\mu}$ supported on $\mathcal{E}$. The proof uses results from the proof of Choquet’s theorem in \cite{Choquet}. Denote
	\[
	\mathfrak{M}^+ := \{ \mu \mid \text{positive Radon measures on } \theconvex \}.
	\]
	Following \cite{Choquet}, we endow $\mathfrak{M}^+$ with the partial order: for $\mu_1, \mu_2 \in \mathfrak{M}^+$, we write $\mu_1 \leq \mu_2$ if
	\begin{equation} \label{eq : def rel ordre Choquet}
		\left\{
		\begin{aligned}
			&\mu_1(\theconvex) = \mu_2(\theconvex), \\
			&\int_\theconvex r \mu_1(dr) = \int_\theconvex r \mu_2(dr), \\
			&\int_\theconvex f(r) \mu_1(dr) \leq \int_\theconvex f(r) \mu_2(dr), \quad \text{for all continuous convex } f : \theconvex \to \mathbb{R}_+.
		\end{aligned}
		\right.
	\end{equation}
	Lemma 1 in \cite{Choquet} shows that $(\mathfrak{M}^+, \leq)$ is inductive, meaning any chain admits an upper bound. Corollary 1 in \cite{Choquet}, which is a consequence of Zorn's lemma, then implies that any $\mu \in \mathfrak{M}^+$ is upper bounded by a maximal element $\nu \in \mathfrak{M}^+$. Applying this to $\mu^\star$, there exists a maximal element $\overline{\mu} \in \mathfrak{M}^+$ such that $\mu^\star \leq \overline{\mu}$. By the first point in the definition 
	$\mu^\star \leq \overline{\mu}$, we have $1 \geq \mu^\star(\theconvex) = \overline{\mu}(\theconvex)$, so $\overline{\mu}$ is a sub-probability. Moreover,
	\[
	\int_\theconvex r \mu^\star(dr) = \int_\theconvex r \overline{\mu}(dr),
	\]
	and thus, applying $e_x$ by Proposition \ref{p: evaluation},
	\[
	\int_\theconvex e_x(r) \overline{\mu}(dr) = \int_\theconvex e_x(r) \mu^\star(dr) = 1 \quad \text{a.e.\ in } x.
	\]
	{\revar Let us take $u \in \mathbb{R}^{\mathrm{d}}$ and set $f(r):=u^\top j(r) u$, which defines a convex real-valued function, since $j$ is convex and takes values in $\mathcal{S}_+$.
	}
	Applying the third property in \eqref{eq : def rel ordre Choquet} with the convex function $f$, we obtain
	\[
	{\revar
		\int_\theconvex u^\top j(r) u \mu^\star(dr) \leq \int_\theconvex u^\top j(r) u \overline{\mu}(dr).}
	\]
	{\revar Hence, for all $u\in \mathbb{R}^{\mathrm{d}}$, we have
		$	u^\top \left( \int_\theconvex  j(r)  \mu^\star(dr) \right) u \leq u^\top \left( \int_\theconvex  j(r) u \overline{\mu}(dr)\right) u$, which implies that for the Loewner order
		\begin{equation*}
			\int_\theconvex j(r) \mu^\star(dr) \leq \int_\theconvex j(r) \overline{\mu}(dr).
		\end{equation*}
		As $\ell$ is increasing, we deduce
		\begin{equation*}
			J^\star = \ell\left( \int_\theconvex j(r) \mu^\star(dr) \right) \leq  \ell
			\left(\int_\theconvex j(r) \overline{\mu}(dr)\right).
		\end{equation*}
	}
	By Lemma 5 in \cite{Choquet}, any maximal element in $\mathfrak{M}^+$ is supported on the set of extremal points of $\theconvex$, which is $\mathcal{E}$. Consequently,
	\[
	\int_\mathcal{E} e_x(r) \overline{\mu}(dr) = \int_\theconvex e_x(r) \overline{\mu}(dr) = 1 \quad \text{a.e.\ in } x,
	\]
	and
	{\revar 
		\begin{equation*}
			J^\star \leq  \ell
			\left(\int_\theconvex j(r) \overline{\mu}(dr)\right)
			=\ell \left(\int_\mathcal{E} j(r) \overline{\mu}(dr)\right).
		\end{equation*}
	}
	By definition of $J^\star$ as the supremum, we conclude
	{\revar 
		\[
		J^\star = \ell \left( \int_\mathcal{E} j(r) \overline{\mu}(dr) \right).
		\]
	}
	Hence, $\overline{\mu}$ is a solution to the optimization problem \eqref{eq : optimisation pb ctn J}.
	
	{\revar 
		The proof of the construction of $\overline{\mu}$ from ${\mu}^\star$ shows that searching for a maximum in \eqref{eq : optimisation pb ctn J} for measures $\mu$ supported  on $\theconvex$ or on $\mathcal{E}$ yields the same value.
	}
\end{proof}

\subsubsection{Proof of Theorem \ref{thm: main opt cont}}
The proof of the theorem is based on the fact that we can apply Proposition \ref{prop : maximisation pb convex  multidim} to the function \( i \) defined as in \eqref{eq : def i r ctn convex}. This is justified by the following lemma.
\begin{lem}\label{lem : i ctn cvx}
	The function 
	\begin{equation} \label{Eq : rappel expression i supp}
		i: \begin{cases} 
			& \theconvex \to \R_+ \\ 
			&{\revarna r  \mapsto i(r)  = \frac{ \left( \int_\cX  s_{\theta_0}(x) \, r(x) \, p_{\theta_0}(x) \, dx \right)^2 }{
					\int_{\cX} r(x) \, p_{\theta_0}(x) \, dx }} 
		\end{cases}
	\end{equation}
	is bounded, continuous, and convex.
\end{lem}
\begin{proof}
	Since \( s_{\theta_0} \in L^2(p_{\theta_0}) \) and \( 1 \le r \le e^{\alpha} \) for \( r \in \theconvex \), we deduce that
	\[
	i(r) \le e^{2\alpha} \left( \int_\cX \abs{s_{\theta_0}(x)} p_{\theta_0}(x) \, dx \right)^2
	\le e^{2\alpha} \norm{s_{\theta_0}}_{ L^2(p_{\theta_0})}^2, ~
	\quad \text{for all } r \in \theconvex.
	\]
	Hence, \( i \) is bounded. Also, the functions \( x \mapsto s_{\theta_0}(x) p_{\theta_0}(x) \) and \( x \mapsto p_{\theta_0}(x) \) belong to \( L^1(\cX) \), thus the mapping \( r \mapsto i(r) \) is continuous in the weak-\(\star\) topology
	recalling the expression \eqref{eq : def i r ctn convex}. The convexity of \( i \) follows from the algebraic properties:
	\begin{align*}
		i(\lambda r) &= \abs{\lambda} \, i(r), \quad \text{for all } \lambda \in \R,\, r \in \theconvex, \\
		i(r_1 + r_2) &\le i(r_1) + i(r_2), \quad \text{for all } r_1, r_2 \in \theconvex.
	\end{align*}
	These properties are verified through straightforward computations using the definition \eqref{eq : def i r ctn convex}.
\end{proof}

We can now invoke Lemma \ref{lem : i ctn cvx} to prove Theorem \ref{thm: main opt cont}, as detailed below.

\begin{proof}
	For any \( q \in \cQ_\alpha \), by the factorization result in Proposition \ref{prop : factor ctn}, we can write \( q = q^{(2)} \circ q^{(\mu)} \). Using the fact that Fisher information is reduced by randomization, we have
	\[
	\cI_{\theta_0}(q \circ \cP) \le \cI_{\theta_0}(q^{(\mu)} \circ \cP).
	\]
	Thus,
	\[
	\cJ^{\max,\alpha}_{\theta_0} \le \sup_{\mu} \cI_{\theta_0}(q^{(\mu)} \circ \cP),
	\]
	where the supremum is over all Radon sub-probabilities satisfying the condition \eqref{eq : cond norm mu continue}. 
	From \eqref{eq : Fisher ctn general_preli}, we know that
	\[
	\cI_{\theta_0}(q^{(\mu)} \circ \cP) = \int_\cE i(r) \, \mu(dr),
	\]
	where \( i \) is defined in \eqref{eq : def i r ctn convex}. By Proposition \ref{prop : maximisation pb convex multidim} {\revar (applicable by Lemma \ref{lem : i ctn cvx}  and taking $\ell$ to be the identity function),} there exists a Radon sub-probability \( \overline{\mu} \) on \( \cE \) such that
	\[
	\sup_{\mu} \int_\cE i(r) \, \mu(dr) = \int_\cE i(r) \, \overline{\mu}(dr) = \cI_{\theta_0}(q^{(\overline{\mu})} \circ \cP).
	\]
	It follows that
	\[
	\cJ^{\max,\alpha}_{\theta_0} \le \cI_{\theta_0}(q^{(\overline{\mu})} \circ \cP).
	\]
	{\revar By definition,  $\cJ^{\max,\alpha}_{\theta_0}$ is the supremum of the Fisher information 
		over any mechanisms valued on a Polish locally compact space. Since $\cE \subset \theconvex$, the mechanism $q^{(\overline{\mu})}$ is valued in $\theconvex$, which is a Polish and locally compact space.
		Thus, we deduce that there is equality: $\cJ^{\max,\alpha}_{\theta_0} = \cI_{\theta_0}(q^{(\overline{\mu})} \circ \cP).$}
\end{proof}

{\revarn 
	\section{Extension of the factorization result to non countable $\cZ$} 
	\label{ss: factorization Polish Z}
	We establish that the factorization Proposition \ref{prop : factor ctn} holds for more general spaces $\mathcal{Z}$ than countable ones. 
	
	\subsection{Recap on the notations}
	The space $\mathcal{Z}$ is a Polish space that is locally compact, 
	and endowed with  {\revarna its Borel sigma-field $\Borelof{\cZ}$} and
	a finite measure $\vartheta$. It is assumed that 
	the randomization from $\mathcal{X}$ to $\mathcal{Z}$ is given by
	\begin{equation}
		\label{eq : supp Q as q vartheta}
		\forall x\in\cX : \quad	Q(dz \mid X=x) = q_x(z) \vartheta(dz).
	\end{equation} 
	We assume that $(x,z)\mapsto q_x(z)$ is a measurable function.
	Moreover,  we have
	\begin{equation}\label{eq : supp LDP density}
		\forall z\in \cZ : \quad	e^{-\alpha} \le \frac{q_x(z)}{q_{x'}(z)} \le e^\alpha, ~  dxdx'-a.e. 
	\end{equation}
	The lemma below shows that the definition of the $\alpha$-LDP mechanism implies the existence of such measurable Radon-Nikodym derivatives $q$, since the Borel sigma-field $\Borelof{\cZ}$ is countably generated, using that $\cZ$ is a Polish space. This lemma is a slight extension of Lemma B3 in \cite{11CLDP}.
	\begin{lem}\label{eq : extension of BRS and measurable}
		Assume that $Q$ is a Markov kernel from the measurable space $(\cX,\Sigma_\cX)$ to the measurable space $(\cZ,\Sigma_\cZ)$. We assume that $\Sigma_\cZ$ is countably generated and that $e^{-\alpha} Q(A\mid X=x') \le Q(A\mid X=x) \le e^\alpha Q(A\mid X=x')$ for all $A\in\Sigma_\cZ$ and all $(x,x')\in \cX^2$. 
		Then, there exist a probability measure $\vartheta$ on $\cZ$ and a 
		measurable function 
		\begin{equation*}
			q:		\left\{
			\begin{array}{rcl}
				\cX\times\cZ &\to& [e^{-\alpha},e^\alpha]
				\\
				(x,z) & \mapsto& q_x(z)
			\end{array}
			\right.
		\end{equation*}	
		such that \eqref{eq : supp Q as q vartheta}--\eqref{eq : supp LDP density} hold true. We also have 
		\begin{equation} \label{eq : densite LDP pty base}
			\forall (x,x')\in\cX^2 :  	\quad	e^{-\alpha} \le \frac{q_x(z)}{q_{x'}(z)} \le e^\alpha, ~  \vartheta(dz)-a.e..
		\end{equation}		
	\end{lem}
	\begin{proof}
		As in the proof of Lemma B3 in \cite{11CLDP},
		we take some $x_0\in \cX$ and set $\vartheta(dz)=Q(dz\mid X=x_0)$. For $x\in\cX$, we define $\tilde{q}_{x}^{(0)}(z)$ as the Radon-Nikodym derivative $\frac{Q(dz\mid X=x_0)}{\vartheta(dz)}$. Since $\Sigma_Z$ is countably generated, by Theorem 4.44 in \cite{Cinlar} we can assume that $(x,z) \mapsto \tilde{q}_{x}^{(0)}(z)$ is measurable. Exactly as in the proof of Lemma B3 in \cite{11CLDP}, we can show that the $\alpha$-LDP condition implies that, for all $x \in \cX$, we have $\vartheta(dz)$ almost surely, $e^{-\alpha} \leq \tilde{q}_{x}^{(0)}(z) \leq e^{\alpha}$. We set for all $(x,z)\in\cX\times\cZ$,
		\begin{equation} \label{eq : supp def q^1 LDP}
			\tilde{q}_{x}^{(1)}(z)= 
			\begin{cases} 
				\tilde{q}_{x}^{(0)}(z) & \text{ if $\tilde{q}_{x}^{(0)}(z)\in[e^{-\alpha},e^\alpha]$,}
				\\
				1 & \text{ otherwise.}
			\end{cases}
		\end{equation}
		We have for all $x\in \cX$, that $\tilde{q}_{x}^{(1)}(z)=\tilde{q}_{x}^{(0)}(z)$, $\vartheta(dz)$-a.s., and consequently $\tilde{q}_{x}^{(1)}$ is a Radon-Nikodym derivative $\frac{Q(dz\mid X=x)}{\vartheta(dz)}$.
		
		Hence, for $(x,x') \in \cX^2$, the ratio $\tilde{q}_{x}^{(1)}/\tilde{q}_{x'}^{(1)}$ is a Radon-Nikodym density for $Q(dz\mid X=x)/Q(dz\mid X=x')$. Using the $\alpha$-LDP constraint, we deduce that
		$e^{-\alpha}\leq \tilde{q}_{x}^{(1)}(z)/\tilde{q}_{x'}^{(1)}(z) \leq e^{\alpha}$ holds $Q(dz\mid X=x')$ almost surely. As $Q(dz\mid X=x')$ is equivalent to $\vartheta(dz)$, it gives that the family $(\tilde{q}_{x}^{(1)})_x$ satisfies \eqref{eq : densite LDP pty base}.
		
		Let us denote $\Psi(x,x',z)=\indi{[e^{-\alpha}, e^\alpha]^c}(\tilde{q}_{x}^{(1)}(z)/\tilde{q}_{x'}^{(1)}(z))$.
		By Fubini's theorem and \eqref{eq : densite LDP pty base}, we have
		\begin{equation*}
			\int_\cZ \left(\int_{\cX^2} \Psi(x,x',z) dxdx'\right) \vartheta(dz)=
			\int_{\cX^2} 	\left(\int_\cZ  \Psi(x,x',z) \vartheta(dz) \right)dxdx'=0.
		\end{equation*}
		Define $G_\cZ=\{z\in\cZ : \int_{\cX^2} \Psi(x,x',z) \,dx\,dx' =0 \}$, which has full $\vartheta$-measure by the previous computation. We set for all $(x,z) \in \cX\times\cZ$;
		\begin{equation}  \label{eq : supp def q^2 LDP}
			q_{x}(z)= 
			\begin{cases} 
				\tilde{q}_{x}^{(1)}(z) & \text{ if $z\in G_\cZ$,}
				\\
				1 & \text{ otherwise.}
			\end{cases}
		\end{equation}
		For all $x$, we have $q_{x}(z)=	\tilde{q}_{x}^{(1)}(z)$ for $z \in G_\cZ$, that is $\vartheta(dz)$-a.s.. It entails that, for all $x\in\cX$, $q_{x}$ is a Radon-Nikodym density for $\frac{Q(dz\mid X=x)}{\vartheta(dz)}$, and that \eqref{eq : densite LDP pty base}  holds true. Moreover, the mapping $(x,z)\mapsto q_x(z)$ is valued in $[e^{-\alpha},e^{\alpha}]$ by \eqref{eq : supp def q^1 LDP}--\eqref{eq : supp def q^2 LDP}.
		
		{\revis It remains to verify \eqref{eq : supp LDP density}. Fix $z \in \cZ$. If $z \notin G_\cZ$, then $q_{x'}(z)/q_x(z) = 1$ for all $(x,x') \in \cX^2$, and \eqref{eq : supp LDP density} is clearly satisfied. If, on the other hand, $z \in G_\cZ$, then $\int_{\cX^2} \Psi(x,x',z) \,dx\,dx' = 0$, which implies $q_{x'}(z)/q_x(z) \in [e^{-\alpha}, e^\alpha]$ holds $dx\,dx'$-almost everywhere. }
	\end{proof}

	\subsection{Proof or Proposition  \ref{prop : factor ctn}}

	The goal is to construct a measure $\mu$ satisfying $\int_\cE e_x(r)\mu(dr)=1$ $dx$-a.e. and a randomization
	$q^{(2)}$ from $\mathcal{E}$ to $\mathcal{Z}$ such that
	\begin{equation} \label{Eq : facto X ctn Z Gen supplement}
		q_x(z)\vartheta(dz)=\int_\mathcal{E} q^{(2)}_r(dz)q^{(\mu)}_x(dr), \quad \text{ for $dx$-almost every $x\in\cX$.}
	\end{equation}
	%
%
%
%
%
The proof of Proposition \ref{prop : factor ctn} with a general $\cZ$ follows the structure of the proof given for a discrete $\cZ$.

The first step is to establish an analogous of \eqref{eq : q x with mu r in proof}, which is the object of the next lemma. As $\mathcal{Z}$ is uncountable, additional difficulties arise. 

In analogy with \eqref{eq: 4.4.5}, we set for $z\in \mathcal{Z}$, $\underline{q}(z)=\essinf_{x\in \mathcal{X}} q_x(z)$, which defines a measurable function as we have
$\{z : \underline{q}(z)\ge a\}=\{z : \essinf_{x\in \mathcal{X}} q(x,z) \ge a\} = \cap_{a'<a,~a'\in\mathbb{Q}} \{z : \int_\cX \indi{q_x(z)<a'} dx =0\}$.


For technical reasons, we have to introduce  $\CompleteBorelof{\cZ}$ which is the completion of the Borel sigma-field $\Borelof{\cZ}$ with respect to the null sets of $\vartheta$.

\begin{lem}
	There exists a Markov kernel 
	\begin{equation*}
		\mu(\cdot,dr) :
		\left\{ 
		\begin{array}{rcl}
			\mathcal{Z} \times \Borelof{\theconvex} &\to& [0,\infty)
			\\
			(z,A) &\mapsto &\mu(z,A)
		\end{array}
		\right. 
	\end{equation*}
	such that $dx\vartheta(dz)$-almost everywhere we have
	\begin{equation} \label{Eq : q as integral more general}
		q_x(z)= \int_{\mathcal{E}} e_x(r)  \mu(z,dr) \underline{q}(z).
	\end{equation}
	This Markov kernel is such that $z\mapsto \mu(z,A)$ is $\CompleteBorelof{\cZ}$-measurable for all $A\in\Borelof{\theconvex}$, where the latter denotes the Borel sets of $\theconvex$.
	
	Moreover, for all $z\in\mathcal{Z}$, the probability $\mu(z,dr)$ is supported on $\mathcal{E}$.
\end{lem}
\begin{proof}
	Let us set
	$u_x(z)= \frac{q_x(z)}{\underline{q}(z)}$. Then, for all $z\in\mathcal{Z}$, we have $dx$-a.e. that $u_x(z)\in[1,e^\alpha]$. By Choquet's theorem, there exists a probability measure $\mu_z$ supported on $\mathcal{E}$ such that for all $z\in\mathcal{Z}$, we have $dx$-a.e., $u_x(z)=\int_\mathcal{E} e_x(r) \mu_z(dr)$.
	However, the measurability with respect to $z$ of this procedure is not clear, as Choquet's theorem relies on Zorn's lemma and does not give an explicit construction of the measure $\mu_z$. We have to resort to a measurable selection theorem.  
	
	Let us denote 
	\begin{multline*}
		\Delta:=\Big\{ (z,\mu) \in \mathcal{Z}\times\mathfrak{M}^{+,1}(C)  ~ \text{such that, }  \\
		u_x(z) = \int_{C} e_x(r) 
		\mu(dr) ~dx-a.e., ~\mu(C)=1, \mu(\mathcal{E}^c)=0  \Big\},
	\end{multline*}
	where we recall that $\mathfrak{M}^{+,1}(C)$ denotes the set of Radon sub-probabilities on $C$. 
	We endowed $\mathfrak{M}^{+,1}(C)$ with the 
	{\revarna  weak   
		topology of measures} and the associated Borel sigma-field.	
	By definition of $\mu_z$, we have $(z,\mu_z) \in  \Delta$ for any $z\in \mathcal{Z}$.
	
	Let $\Phi :  \mathcal{Z}\times\mathfrak{M}^{+,1}(C) \to \mathbb{R} $ the mapping defined by
	\begin{equation*}\Phi(z,\mu):= \norm{u_\cdot(z) - \int_C e_{\cdot}(r)\mu(dr) }_{L^\infty(\mathcal{X})} + \abs{\mu(C)-1} + \mu(\mathcal{E}^c).
	\end{equation*}	
	We make explicit that $\norm{u_\cdot(z) - \int_C e_{\cdot}(r)\mu(dr) }_{L^\infty(\mathcal{X})}=\esssup_{x\in\mathcal{X}} \abs*{ u_x(z)- \int_C e_{x}(r)\mu(dr)}$. The function $\Phi$ is tailored for the representation $\Delta=\Phi^{-1}(\{0\})$.
	The mapping $\Phi$ is measurable if we endowed the product space $\mathcal{Z}\times\mathfrak{M}^{+,1}(C) $
	with the product sigma field of $\Borelof{\cZ}$ and of the Borel sigma-field of $\mathfrak{M}^{+,1}(C)$. Indeed, the mapping $\mu \mapsto \mu(A)$ is measurable for any $A$ Borel set of $C$. Moreover, for $g\in {C} \subset L^\infty(\mathcal{X})$, we can represent $\norm{g}_{L^\infty(\mathcal{X})}=\sup_n \int_\mathcal{X} f_n(x)g(x) dx$ 
	for some fixed sequence $(f_n)_n$ of elements of $L^1(\mathcal{X})$. 
	From this, we can check that  $(z,\mu) \mapsto  \norm{ u_\cdot(z) - \int_C e_{\cdot}(r)\mu(dr) }_{L^\infty(\mathcal{X})} $ is measurable, as a limit of measurable functions.	Consequently 
	$\Delta$ is a measurable subset of the product  $\mathcal{Z}\times\mathfrak{M}^{+,1}(C)$.
	
	Consider the projection $\Pi : \mathcal{Z} \times \mathfrak{M}^{+,1}(C) \to \mathcal{Z}$ defined by $\Pi(z, \mu) = z$. Observing that $(z, \mu_z) \in \Delta$ for all $z \in \mathcal{Z}$, we deduce that $\Pi(\Delta) = \mathcal{Z}$. Since $\mathfrak{M}^{+,1}(C)$ is a compact metric space, and therefore Polish, we can apply Corollary~8.5.4 in \cite{Cohn_Book_2013} to deduce the existence of a measurable selection $F : \mathcal{Z} \to \mathfrak{M}^{+,1}(C)$ such that $(z, F(z)) \in \Delta$ for all $z \in \mathcal{Z}$. 
	{\revarna This selection is measurable if $\cZ$ is endowed with 
		the universal completion of $\Borelof{\cZ}$, and hence it is measurable with respect to its $\vartheta$-completion
		$\CompleteBorelof{\cZ}$.
	}
	Now, we set $\mu(z,dr) := F(z) \in \mathfrak{M}^{+,1}(C)$ and by construction, for all $z\in \mathcal{Z}$ we have $dx$-a-e,
	\begin{equation} \label{eq : u as integral more general} 
		u_x(z)=\int_{C} e_x(r) \mu(z,dr)
	\end{equation} 
	where $\mu(z,dr)$ is a probability supported on $\mathcal{E}$. 
	Moreover, if $H :{C} \to \mathbb{R}$ is any bounded measurable function, then $z\mapsto \int_{C} H(r) \mu(z,dr)$ is measurable. This proves that $\mu(\cdot,dr)$ defines a randomization kernel from $\mathcal{Z}$ to $C$.

	It remains to check 
	\eqref{Eq : q as integral more general}. Let $f_1 \in L^1(\mathcal{X})$ and $f_2 \in L^\infty(\mathcal{Z})$. As $r\mapsto H(r):=\int_\mathcal{X}f_1(x)r(x)dx$ is continuous and thus measurable,
	we can consider
	\begin{align}\nonumber
		\int_\mathcal{Z} \int_{C} H(r) &\mu(z,dr) \underline{q}(z) f_2(z) \vartheta(dz) 
		\\ \label{eq : proof q is integral line 1}
		&= \int_\mathcal{Z} \int_{C}  \int_\mathcal{X} f_1(x) r(x) dx \mu(z,dr)  \underline{q}(z) f_2(z) \vartheta(dz),
		\\
		\nonumber
		&	=  \int_\mathcal{Z}  \left( \int_\mathcal{X} f_1(x) \int_{C} e_x(r)  \mu(z,dr) dx \right) \underline{q}(z) f_2(z)  \vartheta(dz), \quad
		\shortintertext{having used Fubini on the two inner integrals, with $z\in\mathcal{Z}$ fixed}
		\nonumber
		&=  \int_\mathcal{Z}  \left( \int_\mathcal{X} f_1(x) u_x(z) dx \right)\underline{q}(z) f_2(z) \vartheta(dz),
		\quad	\text{by \eqref{eq : u as integral more general}}
		\\ \label{eq : proof q is integral last_line} 
		& =  \int_\mathcal{X} \int_\mathcal{Z}   f_1(x)   q_x(z)  f_2(z) \vartheta(dz)dx, 
		\quad \text{recalling 
			$q_x(z)=u_x(z) \underline{q}(z)$}.
	\end{align}
	We use Fubini's theorem in
	\eqref{eq : proof q is integral line 1} to exchange the $dx$ and $\mu(z,dr)\vartheta(dz)$ integrals and we compare it with \eqref{eq : proof q is integral last_line} to obtain 
	\begin{align*} 
		\int_\mathcal{X} f_1(x)  
		\int_\mathcal{Z}  \int_{C}   e_x(r) \mu(z,dr)  \underline{q}(z) f_2(z) \vartheta(dz) dx &= \int_\mathcal{X} f_1(x)  \int_\mathcal{Z}    {\revarna q_x(z)}  f_2(z) \vartheta(dz)dx.
		\nonumber 
	\end{align*}
	Using that $f_1\in L^1(\mathcal{X})$ and $f_2 \in L^\infty(\mathcal{Z})$ are arbitrary elements we get
	$$q_x(z) = \int_{C}   e_x(r) \mu(z,dr)  \underline{q}(z).$$
	Then, as the support of $\mu(z,dr)$ is $\mathcal{E}$, we deduce 
	\eqref{Eq : q as integral more general}.	
\end{proof}
We are now ready to construct the measure $\mu$ and prove \eqref{Eq : facto X ctn Z Gen supplement}.

$\bullet$ Construction of $\mu$.


Generalizing \eqref{eq : def mu proof fact ctn}, we  define the measure $\mu$ on $C$, by
\begin{equation} \label{eq : def mu uncountable} 
	\mu(A):=\int_\mathcal{Z} \mu(z, A) \underline{q}(z)\vartheta(dz) 
	=\int_\mathcal{Z} \left( \int_C \indi{A}(r) \mu(z, dr) \right) \underline{q}(z)\vartheta(dz)
\end{equation} for $A$ any Borel subset of ${C}$.
The measure $\mu$ is supported in $\mathcal{E}$, using that $\mu_z(\mathcal{E}^c)=0$ for all $z\in\mathcal{Z}$. For $x\in \mathcal{X}$, we have
$\int_{\mathcal{E}} e_x(r) \mu(dr)= \int_\mathcal{Z} \int_{\mathcal{E}} e_x(r)\mu(z, dr) \underline{q}(z)\vartheta(dz) = \int_\mathcal{Z} q_x(z) \vartheta(dz)$, where the last equality holds $dx$-almost everywhere by \eqref{Eq : q as integral more general}. Since $q_x(z)\vartheta(dz)$ is a randomization mechanism, it integrates to $1$. We deduce that $ \int_{\mathcal{E}} e_x(r) \mu(dr) =1$, $dx$-almost everywhere. Hence, using that $r\mapsto e_x(r)$ is measurable for all $x\in\mathcal{X}$, we can set
\begin{equation} \label{eq : q mu as usual but very general}
	q^{(\mu)}_x(dr):=e_x(r) \mu(dr)
\end{equation}
which defines $dx$-everywhere a randomization mechanism from $\mathcal{X}$ to $\mathcal{E}$. 

{$\bullet$ Construction of $q^{(2)}$.}

It remains to construct the second part of the composition, which is the mechanism from $\mathcal{E}$ to
$\mathcal{Z}$. 	 
When $\mathcal{Z}$ is uncountable, it seems impossible to derive that $\mu(z,dr)$ is absolutely continuous with respect to $\mu$ and use the Radon-Nikodym's theorem as in the proof of Proposition \ref{prop : factor ctn} for the discrete case. The idea it to first approximate $\mathcal{Z}$ by a countable set and  take a limit.

As $\mathcal{Z}$ is a separable space, let $(z_i)_{i\ge 1}$ be a dense sequence. For $\varepsilon>0$, we consider a partition of $\mathcal{Z}$ as $\mathcal{Z}= \sqcup_{i \ge 1} D_i^{(\varepsilon)}$ where $D_i^{(\varepsilon)}$ is included 
in the ball with center $z_i$ and radius $\varepsilon$.
We define $T^{(\varepsilon)} : \mathcal{Z} \to \{z_i,i\ge1\}$, by $T^{(\varepsilon)}(z)=z_i$ is $z\in D_i^{(\varepsilon)}$. We consider the randomization obtained by consecutively applying  the randomization $q$ from $\mathcal{X}$ to $\mathcal{Z}$ and then the projection $T^{(\varepsilon)}$ on $\{z_i,~ i \ge 1\} \subset \mathcal{Z}$.
Let us denote $q_\varepsilon(x,dz)$ this randomization. Its law is given by a sum of Dirac masses
\begin{equation} 
	\label{eq : q varpsilon as dirac}
	q_\varepsilon(x,dz) := \sum_{i=1}^\infty \left(\int_{D_i^{(\varepsilon)} } q_x(z)\vartheta(dz) \right) 
	\delta_{\{z_i\}} (dz).
\end{equation}
The randomization mechanism $q_\varepsilon(x,dz)$ approximates $q_x(z)\vartheta(dz)$ for $\varepsilon\to0$. Indeed, if $f : \mathcal{Z} \to \mathbb{R}$ is continuous with compact support, we have
\begin{multline} \label{eq : approx randomise rendu denombrable}
	\abs*{ \int_{\mathcal{Z}} f(z) q_x(z) \vartheta(dz) - \int_{\mathcal{Z}} f(z) q_\varepsilon(x,dz)
	}
	\\
	=
	\abs*{ \int_{\mathcal{Z}} f(z) q_x(z) \vartheta(dz) - \int_{\mathcal{Z}} f(T^{\varepsilon}(z)) q_x(z) \vartheta(dz)
	} \xrightarrow{\varepsilon\to0}0,
\end{multline}
using the fact that the distance between $T^{\varepsilon}(z)$ and $z$ is less than $\varepsilon$.

Now, we define for $i \ge 1$, a measure on $\mathcal{E}$, by 
$\mu_i^{(\varepsilon)}(\cdot)=\int_{D_i^{(\varepsilon)}} \mu(z,\cdot) \underline{q}(z) \vartheta(dz)$. As $\mathcal{Z}=\sqcup_{i \ge 1} D_i^{(\varepsilon)}$, and recalling \eqref{eq : def mu uncountable},	 
we have $\mu=\sum_{i=1}^\infty \mu_i^{(\varepsilon)}$.
This implies that, for each $i\geq1$, $\mu_i^{(\varepsilon)}$ is absolutely continuous with respect to $\mu$. Let us denote $f_i^{(\varepsilon)}(r):=\frac{\mu_i^{(\varepsilon)}(dr)}{\mu(dr)}$. From $\mu=\sum_{i=1}^\infty \mu_i^{(\varepsilon)}$, we deduce that $\sum_{i=1}^\infty f_i^{(\varepsilon)}(r)=1$, $\mu(dr)$-almost everywhere. Let us define for $i\ge1$, $r\in\mathcal{E}$,
\begin{equation*}
	\overline{f}_i^{(\varepsilon)}(r) = 
	\begin{cases}
		f_i^{(\varepsilon)}(r) &\text{ if $\sum_{j=1}^\infty f_j^{(\varepsilon)}(r)=1$},
		\\
		\indi{\{1\}}(i) & \text{ if $\sum_{j=1}^\infty f_j^{(\varepsilon)}(r)\neq1$}.
	\end{cases}
\end{equation*}
For all $i\ge 1$, the functions $\overline{f}_i^{(\varepsilon)}$  and $f_i^{(\varepsilon)}$ are $\mu$ almost everywhere equal. By construction, we have for all $r\in \mathcal{E}$, 	
$\sum_{i=1}^\infty \overline{f}_i^{(\varepsilon)}(r)=1$. This enables us to define a randomization from {\revarna $\mathcal{E}$ to $\mathcal{Z}$} by
\begin{equation} \label{eq : q 2 epsilon general}
	q^{(2,\varepsilon)}_r(dz) = \sum_{i=1}^\infty  \overline{f}_i^{(\varepsilon)}(r) \delta_{\{z_i\}}(dz), \quad \text{ \revarna for $r\in\cE$.}
\end{equation}

The next step of the proof is to check that the composition of $q^{(\mu)}_x(dr)=e_x(r)\mu(dr)$ with $q^{(2,\varepsilon)}_r(dz)$ gives $q_{\varepsilon}(x,dz)$ for almost every $x$. Having in mind \eqref{eq : q varpsilon as dirac}, we compute
\begin{align*}
	\int_{D_i^{(\varepsilon)} } q_x(z)\vartheta(dz)  
	&= 	\int_{D_i^{(\varepsilon)} } 
	\int_\mathcal{E} e_x(r) \mu(z,dr) \underline{q}(z)
	\vartheta(dz),\quad dx-a.e., \quad \text{by \eqref{Eq : q as integral more general},}
	\\
	&= 	\int_{\mathcal{E}} 
	e_x(r) \mu^{(\varepsilon)}_i(dr), \quad \text{by def of $\mu^{(\varepsilon)}_i(dr)$,}
	\\
	&= \int_{\mathcal{E}  } 
	e_x(r) f^{(\varepsilon)}_i(r) \mu(dr), \quad \text{by definition of $f^{(\varepsilon)}_i$ as a Radon Nikodym derivative,}
	\\
	&=  \int_{\mathcal{E}  } 
	e_x(r) \overline{f}^{(\varepsilon)}_i(r) \mu(dr),  \quad 
	\text{as $\overline{f}^{(\varepsilon)}_i(r)={f}^{(\varepsilon)}_i(r)$, $ \mu(dr)$ a.e.}
\end{align*}
Thus, we can write \eqref{eq : q varpsilon as dirac} as 
\begin{align}\nonumber
	q_\varepsilon(x,dz)& = \sum_{i=1}^\infty \left( \int_{\mathcal{E}  } 
	e_x(r) \overline{f}^{(\varepsilon)}_i(r) \mu(dr) \right) 
	\delta_{\{z_i\}} (dz)
	\\ \label{eq : facto mu general Z rendu denom}
	& =\int_\mathcal{E} q^{(2,{\varepsilon})}_r(dz) q^{(\mu)}_x(dr), \quad \text{by \eqref{eq : q mu as usual but very general} and \eqref{eq : q 2 epsilon general}}.
\end{align}

The final step of the proof is to take the limit as $\varepsilon\to0$ in the previous construction. 

Let us denote $\nu^{(\varepsilon)}$ the measure on $\mathcal{E}\times\mathcal{Z}$ defined by 
$\nu^{(\varepsilon)}=\mu(dr) q^{(2,\varepsilon)}_r(dz)$. Since $\mathcal{Z}$ is locally compact, one can consider $\widehat{\mathcal{Z}}$ its Alexandroff compactification (see Theorem 8.4 in Chapter XI of \cite{dugundjiTopology1966}), which is also a metric space using that $\mathcal{Z}$ is a Polish space (see Theorem 5.6 in Chapter IX and Theorem 8.6 in Chapter XI of \cite{dugundjiTopology1966}). The measures $(\nu^{(\varepsilon)})_{\varepsilon>0}$ are sub-probabilities on the compact set ${C}\times\widehat{\mathcal{Z}}$. One can find $(\varepsilon_n)_n$ a sequence converging to $0$ and $\nu^{(0)}$ a sub-probability on  ${C}\times\widehat{\mathcal{Z}}$ which is the limit of $\nu^{(\varepsilon_n)}$. It means that for any $H : C \times \widehat{\mathcal{Z}} \to \mathbb{R}$ that is continuous and bounded, we have
\begin{equation} \label{eq : def convergence weak produit}
	\lim_n \int_C \int_{\widehat{\mathcal{Z}}} H(r,z)  q^{(2,\epsilon_n)}_r( dz)\mu(dr) =
	\int_C \int_{\widehat{\mathcal{Z}}} H(r,z)   \nu^{(0)}(dr,dz).
\end{equation} 

We use the disintegration of the measure $\nu^{(0)}$ on $ C \times \widehat{\mathcal{Z}}  $ with respect to its projection on $C$ (see Theorem 1 in \cite{Chang_Pollard97}. It allows to write $\nu^{(0)}(dr,dz)= \nu^{(0,1)}(dr) \nu^{(0,2)}(r,dz)$, with $\nu^{(0)}(A \times \widehat{\mathcal{Z}})=\nu^{(0,1)}(A)$ for $A$ any Borel subset of $C$. The latter implies that, $\nu^{(0,1)}(dr)$-almost everywhere, we have $\nu^{(0,2)}(r,\widehat{\mathcal{Z}})=1$.

We now prove that $\nu^{(0,2)}$ defines a randomization from $\mathcal{E}$ to $\widehat{\mathcal{Z}}$ such that $q$ is the composition of $q^{(\mu)}$ with this second randomization, and in particular that this second randomization takes values in $\mathcal{Z}$ with probability one. 

We apply \eqref{eq : def convergence weak produit} to a bounded continuous function $H : C \to \mathbb{R}$ that does not depend on $z$. Since $q^{(2,\varepsilon)}_r(dz)$ is a probability measure for every $r \in C$, the left-hand side of Equation~\eqref{eq : def convergence weak produit} simplifies to
\begin{equation*} 
	\lim_{n \to \infty} \int_C \int_{\widehat{\mathcal{Z}}} H(r) \, q^{(2,\epsilon_n)}(r, dz) \mu(dr) = \int_C H(r) \mu(dr).
\end{equation*} 
By \eqref{eq : def convergence weak produit}, this limit also equals 
\begin{align*} 
	\int_C \int_{\widehat{\mathcal{Z}}} H(r) \, \nu^{(0)}(dr, dz) 
	&= \int_C \int_{\widehat{\mathcal{Z}}} H(r) \, \nu^{(0,1)}(dr) \nu^{(0,2)}_r(dz) \\
	&= \int_C H(r) \, \nu^{(0,1)}(dr),	
\end{align*} 
where the last equality follows from the fact that $\nu^{(0,2)}(r, \widehat{\mathcal{Z}}) = 1$. Since this holds for any such $H$, we conclude that $\nu^{(0,1)} = \mu$.

Now, we take $f_1 \in L^1(\mathcal{X})$ and $f_2 : \mathcal{Z} \to \mathbb{R}$ continuous with compact support. We apply again \eqref{eq : def convergence weak produit}, with the continuous and bounded function
\begin{equation*}
	H :\left\{ 
	\begin{array}{ccl}
		C \times \widehat{\mathcal{Z}} &\to& \mathbb{R}\\
		(r,z)& \mapsto& \left(\int_{\mathcal{X}} f_1(x)r(x) dx\right) \times f_2(z) .
	\end{array}
	\right.
\end{equation*}
Using $\nu^{(0)}=\mu(dr) \nu^{(0,2)} (r, dz)$, it gives
\begin{multline} \label{eq : final cv in the generalized fact prop}
	\lim_n \int_C \int_{\widehat{\mathcal{Z}}}  \left(\int_{\mathcal{X}} f_1(x)r(x) dx\right) \times f_2(z)  q^{(2,\epsilon_n)}_r(dz)\mu(dr) \\=
	\int_C \int_{\widehat{\mathcal{Z}}} \left(\int_{\mathcal{X}} f_1(x)r(x) dx\right) \times f_2(z)  
	\nu^{(0,2)}(r,dz) \mu(dr).
\end{multline} 
Recalling that the support of $\mu$ is $\mathcal{E}$, and using Fubini's theorem,
the left hand side of \eqref{eq : final cv in the generalized fact prop} is 
\begin{align}\nonumber
	&\lim_n \int_\mathcal{X} f_1(x) \int_\mathcal{E} \int_{\widehat{\mathcal{Z}}}  e_x(r) f_2(z)  q^{(2,\epsilon_n)}_r(dz)\mu(dr) dx
	\\\nonumber
	\quad&=
	\lim_n \int_\mathcal{X} f_1(x) \int_\mathcal{E} \int_{\widehat{\mathcal{Z}}}  f_2(z) q^{(2,\epsilon_n)}_r( dz)q^{(\mu)}_x(dr)dx, \quad \text{by \eqref{eq : q mu as usual but very general}}
	\\ \label{eq : identify correct composition}
	\quad&= 	\lim_n \int_\mathcal{X} f_1(x) \int_{\mathcal{Z}} f_2(z) q_{\varepsilon_n}(x,dz) dx,
\end{align}
where we used \eqref{eq : facto mu general Z rendu denom} and that $f_2$ has compact support.
For almost every $x$, we have $ \int_{\mathcal{Z}} f_2(z) q_{\varepsilon_n}(x,dz) \xrightarrow{n\to\infty}
\int_\mathcal{Z} f_2(z) q_x(z)\vartheta(dz)$, by \eqref{eq : approx randomise rendu denombrable}. As $f_2$ is bounded and $f_1 \in L^1(\mathcal{X})$, we can apply the dominated convergence theorem. It yields
that \eqref{eq : identify correct composition} is equal to
\begin{equation*}
	\int_\mathcal{X}f_1(x)\int_\mathcal{Z} f_2(z) q_x(z)\vartheta(dz)dx,
\end{equation*}
which, we recall, is thus also the value of 
the  left hand side of \eqref{eq : final cv in the generalized fact prop}.	

Comparing with the right hand side of \eqref{eq : final cv in the generalized fact prop}, we deduce
\begin{equation*}
	\int_\mathcal{X}f_1(x)\int_\mathcal{Z} f_2(z) q_x(z)\vartheta(dz)dx
	=	\int_{\mathcal{X}} f_1(x) \int_\mathcal{E} \int_{\mathcal{Z}} e_x(r) f_2(z)  
	\nu^{(0,2)}(r,dz) \mu(dr) dx
\end{equation*}
where we used again Fubini's theorem with $\mu(\mathcal{E}^c)=0$ and the fact that $f_2$ has compact support. Recalling the notation $q^{(\mu)}_x(dr)=e_x(r) \mu(dr)$,
and since $f_1 \in L^1(\mathcal{X})$ is arbitrary and $f_2$ is any continuous function with compact support, we deduce that the two measures  $	q_x(z)\vartheta(dz)dx $ and $\int_\mathcal{E} \nu^{(0,2)}(r,dz) q^{(\mu)}_x(dr) dx$ coincide on $\mathcal{X}\times K $ for $K$ any compact subset of $\mathcal{Z}$. Since $\mathcal{Z}$ is Polish and locally compact, it is a countable union of compact sets (see Theorem 5.6 in Chapter IX and Theorem 7.2 in Chapter XI of \cite{dugundjiTopology1966}), and we deduce the equality between the measures on $\mathcal{X}\times\mathcal{Z}$
\begin{equation} \label{eq : composition cas non denom def a un ps pres}
	q_x(z)\vartheta(dz)dx  = \int_\mathcal{E} \nu^{(0,2)}(r,dz) q^{(\mu)}_x(dr) dx.
\end{equation}
Integrating with respect to $z$ on $\mathcal{Z}$, we deduce that $1=\int_\mathcal{E} \nu^{(0,2)}(r,\mathcal{Z}) q^{(\mu)}_x(dr)$ for almost every $x$. 
We know that $\nu^{(0,2)}(r,
\widehat{\mathcal{Z}})=1$ $\mu(dr)$-a.e. and thus $q^{(\mu)}_x(dr)$ almost everywhere for any $x$
as $q^{(\mu)}_x(dr)=e_x(r) \mu(dr) \ll  \mu(dr)$.
Then,   
taking $x$ for which $q^{(\mu)}_x(dr)$ is a probability shows that $q^{(\mu)}_x(dr)$-a.e., we must have
$\nu^{(0,2)}(r,\mathcal{Z})=1$. So, $\mu(dr)$-a.e., we have $ \nu^{(0,2)}(r,\mathcal{Z})=1$. 
The second step of the factorization is the probability transition defined
\begin{equation*}
	q^{(2)}_r(dz)=\begin{cases}
		\nu^{(0,2)}(r,dz) & \text{ if $\nu^{(0,2)}(r,\mathcal{Z})=1$,}
		\\
		\delta_{\{z_1\}}(dz) & \text{ otherwise.}
	\end{cases}
\end{equation*}
Then, $q^{(2)}$ is a randomization defined from $\mathcal{E}$ to $\mathcal{Z}$. Using 
\eqref{eq : composition cas non denom def a un ps pres} and the fact that, $\mu(dr)$-almost everywhere,
$q^{(2)}_r(dz)=\nu^{(0,2)}(r,dz)$, we deduce {\revarna the following equality between measures on $\cX\times\cZ$:}
%
\begin{equation*}
	q_x(z)\vartheta(dz)dx  = \int_\mathcal{E} q^{(2)}_r(dz) q^{(\mu)}_x(dr) dx.
\end{equation*}
This is the factorization formula we wanted to prove.	 
\qed
}

{\revis	
\section{Dual optimization problem} \label{S : dual program supp}
In this section, we prove that the optimal Fisher information in dimension $1$ can be obtained as the solution of a minimization problem, dual to the maximization described in Section \ref{Ss : Fisher cont}. In particular, we prove Proposition \ref{prop : dual no gap}.

We recall the notation of the primal problem,
\begin{equation} \label{eq : primal nos notations}
	I^*	:=\sup_{\mu} \int_C i(r) \mu(dr),\quad \text{subject to}
	\int_{C} e_x(r)\mu(dr)=1, ~ \text{for $dx$-almost every $x$}
\end{equation}
where the supremum is taken over all Radon sub-probability measures $\mu$ on $C$ and
$i(r)=\frac{\left(\int_{\mathcal{X}} r(x) s_{\theta_0}(x) dx\right)^2  }{\int_{\mathcal{X}} r(x) p_{\theta_0}(x) dx}$ and that of the dual program, which reads
\begin{equation} \label{eq : dual nos notations supp}
	\tilde{I}^*=\inf_{f \in L^1(\mathcal{X})} \int_{\mathcal{X}} f(x) dx, \quad \text{subject to}
	\int_{\mathcal{X}} r(x) f(x) dx \ge i(r), ~ \text{for all $r \in {C}$.}
\end{equation}
We know by Theorem \ref{thm: main opt cont} and Point 2 of Proposition \ref{prop : maximisation pb convex multidim}, 	
that the optimal Fisher information is solution of the primal program.

{\revis We proceed to show that there is no duality gap between the primal and dual programs, as claimed in Point 1 of Proposition \ref{prop : dual no gap}. Following this, we prove the existence of a minimizer $f^* \in L^1(\cX)$ for the dual program, which corresponds to Point 2 of Proposition \ref{prop : dual no gap}.} 



\subsection{Proof of the Point 1 in Proposition \ref{prop : dual no gap}}

%
	%
First, we introduce some notation which allows us to write the primal and dual problem in a more compact form. 

Let $A : L^1(\mathcal{X}) \to \mathcal{C}(C,\mathbb{R})$ be the operator defined for $f\in{L}^1(\mathcal{X})$ by 
\begin{equation*}
	A(f): \left\{
	\begin{array}{ll}
		C & \to \mathbb{R}\\
		r & \mapsto A(f)(r):=\int_\mathcal{X} r(x)f(x) dx,
	\end{array}
	\right.
\end{equation*}
we stress that the continuity of $r \mapsto A(f)(r)$ is with respect to the weak-$\star$ topology. We also introduce the operator $A^*$ defined on the space
$\mathfrak{M}^+(C)$ of non-negative radon measure and taking values in $L^\infty(\mathcal{X})$, by
\begin{equation*}
	A^*(\mu): \left\{
	\begin{array}{ll}
		\mathcal{X} & \to \mathbb{R}\\
		x & \mapsto A^*(\mu)(x):=\int_{C} e_x(r) \mu(dr),
	\end{array}
	\right.
\end{equation*}
for $\mu \in \mathfrak{M}^+(C)$. Using the Fubini theorem, and the measurability of the evaluation map $(x,f) \to e_x(f)$ established in Proposition \ref{p: evaluation}, we have the duality relation
\begin{align*}
	\int_C A(f)(r) \mu(dr) &=
	\int_C \int_\mathcal{X} r(x)f(x) dx \mu(dr)\\&=
	\int_\mathcal{X}\int_C e_x(r)f(x) \mu(dr) dx 
	=\int_\mathcal{X} f(x) A^*(\mu)(x) dx.
\end{align*}
for any $f\in L^1(\mathcal{X})$ and $\mu \in \mathfrak{M}^+(C)$.  
Let us define for $(f,\mu)\in L^1(\mathcal{X}) \times \mathfrak{M}^+(C)$,
\begin{align} 
	\label{eq : def G saddle function A}
	G(f,\mu)&=\int_{\mathcal{X}} f(x) dx -  \int_C \left[ A(f)(r) - i(r) \right] \mu(dr)
	\\&=
	\label{eq : def G saddle function Astar}
	\int_{\mathcal{X}} f(x)\left[1  - A^*(\mu)(x)\right]  dx +  \int_C i(r) \mu(dr),
\end{align}
where we used the duality relation.

The point 1) of Proposition \ref{prop : dual no gap} is a consequence of Lemmas \ref{lem : I great Itilde}--\ref{lem : Itilde great I} below.

\begin{lem} \label{lem : I great Itilde} 
	We have $\tilde{I}^* \ge {I}^*$.	
\end{lem}
\begin{proof}
	The dual program \eqref{eq : dual nos notations supp} reads $\tilde{I}^\star=\inf_{f\in L^1(\mathcal{X})} \int_{\mathcal{X}}f(x)dx$ subject to $A(f)(r)-i(r)\ge 0$ for all $r\in C$. 
	From \eqref{eq : def G saddle function A}, we can write the constraint as $\sup_{\mu \in \mathfrak{M}^+(C)}  G(f,\mu) < \infty$. We deduce
	\begin{equation} \label{Eq : tilde I as inf sup}
		\tilde{I}^*=\inf_{f\in L^1(\mathcal{X})}
		\sup_{\mu \in  \mathfrak{M}^+(C)} G(f,\mu).
	\end{equation}
	The primary program \eqref{eq : primal nos notations} is 
	$I^*=\sup_{\mu \in \mathfrak{M}^+(C)} \int_C i(r) \mu(dr)$ subject to $A^*(\mu)(x)-1=0$ $dx$-almost every $x$, as the condition of $\mu$ being a sub-probability is implied by $A^*(\mu)=1$ together with $\mu \ge 0$. Then, \eqref{eq : def G saddle function Astar} entails
	\begin{equation*}
		I^*=\sup_{\mu \in  \mathfrak{M}^+(C)} \inf_{f\in L^1(\mathcal{X})}G(f,\mu).
	\end{equation*}
	The lemma follows from the algebraic comparison as $\inf \sup \ge \sup \inf$. 
\end{proof}
The delicate part is to prove the reverse inequality.
\begin{lem}\label{lem : Itilde great I}
	We have $\tilde{I}^* \le {I}^*$.	
\end{lem}
\begin{proof}
	We want to apply Sion's minimax theorem to deduce the equality of the two programs. We need to introduce some compactness property. For $M\ge 1$, we set 
	$\mathfrak{M}^{+,M}(C):=\{ \mu \in \mathfrak{M}^{+}(C)  \mid \mu(C) \le M\}$, which is a compact set for the weak topology of measures, as $C$ is compact.
	We introduce, for $M\ge 1$, the modified optimization programs,
	\begin{equation*}
		I^{*}_M=\sup_{\mu \in  \mathfrak{M}^{+,M}(C)} \inf_{f\in L^1(\mathcal{X})}G(f,\mu),\quad
		\tilde{I}^*_M=\inf_{f\in L^1(\mathcal{X})}
		\sup_{\mu \in  \mathfrak{M}^{+,M}(C)} G(f,\mu).
	\end{equation*}
	The function $f \mapsto G(f,\mu)$ is linear and continuous with respect to the $L^1(\mathcal{X})$ norm, using the representation 
	\eqref{eq : def G saddle function Astar}. Moreover, $\mu  \mapsto G(f,\mu)$ is continuous for the weak topology of measure by the expression \eqref{eq : def G saddle function A}. As $\mathfrak{M}^{+,M}(C)$ is compact, we can apply 
	Sion's minimax theorem (see e.g. \cite{komiya88}). We deduce
	that $I^{*}_M=\tilde{I}^{*}_M$.
	
	The primary program \eqref{eq : primal nos notations} admits a solution by Theorem \ref{thm: main opt cont} and the program is unchanged if we omit the condition $\mu(C) \le 1$. We deduce that $I^*_M = I^*$ for all $M \ge 1$.   
	
	The difficulty is the comparison between $\tilde{I}^*_M$ and the dual program \eqref{eq : dual nos notations supp}, as the constraint $A(f)(r)- i(r) \ge 0$ is not implied by  $\sup_{\mu \in  \mathfrak{M}^{+,M}(C)} G(f,\mu)<\infty$.
	
	Let $\varepsilon>0$ and $M > 1$. By the definition of $\tilde{I}^*_M$, let us consider $f_{\varepsilon,M} \in L^1(\mathcal{X})$, with
	\begin{equation*} 
		\tilde{I}_M^* \le \sup_{\mu \in  \mathfrak{M}^{+,M}(C)} G(f_{\varepsilon,M},\mu)
		\le \tilde{I}_M^*+\varepsilon.
	\end{equation*}
	As we know that $\tilde{I}_M^*=	I_M^*=I^*$, it yields
	\begin{equation} \label{eq : sup G f ve M}
		I^* \le \sup_{\mu \in  \mathfrak{M}^{+,M}(C)} G(f_{\varepsilon,M},\mu)
		\le I^*+\varepsilon.
	\end{equation}
	The function $r\in C \mapsto A(f_{\varepsilon,M})(r) - i(r)$ is continuous for the weak-$\star$ topology. By compactness of $C$, let $r_{\varepsilon,M}=
	\argmin_{r \in C}  A(f_{\varepsilon,M})(r) - i(r)$. 
	
	The proof will be different depending on whether or not $\min_{r \in C}  A(f_{\varepsilon,M})(r) - i(r) \ge 0$. Let us start considering the case where $\min_{r \in C}  A(f_{\varepsilon,M})(r) - i(r) \ge 0$. Then, 
	\begin{align*}
		\sup_{\mu \in  \mathfrak{M}^{+}(C)} G(f_{\varepsilon,M},\mu)
		&=\sup_{\mu \in  \mathfrak{M}^{+}(C)} 
		\int_{\mathcal{X}} f_{\varepsilon,M}(x) dx -  \int_C \left[ A(f_{\varepsilon,M})(r) - i(r) \right] \mu(dr)
		\\
		&=	\int_{\mathcal{X}} f_{\varepsilon,M}(x) dx,
	\end{align*}
	as the supremum is reached for $\mu=0$ by positivity of the integrand. For the same reason, the value of the supremum on
	$\mathfrak{M}^{+,M}(C)$ is also reached at
	$\mu=0$. We deduce that $	\sup_{\mu \in  \mathfrak{M}^{+}(C)} G(f_{\varepsilon,M},\mu) = 	\sup_{\mu \in  \mathfrak{M}^{+,M}(C)} G(f_{\varepsilon,M},\mu) \le I^* + \varepsilon$, by \eqref{eq : sup G f ve M}.
	
	From \eqref{Eq : tilde I as inf sup}, we deduce, 
	\begin{equation} \label{eq : tilde I le I plus eps}
		\tilde{I}^* \le \sup_{\mu \in  \mathfrak{M}^+(C)} G(f_{\varepsilon,M},\mu)
		\le  I^* + \varepsilon,
	\end{equation}
	which concludes the proof in the first case. Let us now consider the case where $\min_{r \in C}  A(f_{\varepsilon,M})(r) - i(r) = A(f_{\varepsilon,M})(r_{\varepsilon,M}) - i(r_{\varepsilon,M})  < 0$. We define the measure $\mu_{\varepsilon,M}$ supported on the two elements of $C$, $r=
	\indi{\mathcal{X}}$ and $r=r_{\varepsilon,M}$, by
	\begin{equation*}
		\mu_{\varepsilon,M}(dr)=\delta_{\{\indi{\mathcal{X}}\}}(dr)+\delta_{\{r_{\varepsilon,M}\}}(dr)(M-1).
	\end{equation*}
	The total mass of this measure is $M$ and by \eqref{eq : def G saddle function A}, we have
	\begin{multline*}
		G(f_{\varepsilon,M},\mu_{\varepsilon,M})
		=  \int_{\mathcal{X}}f_{\varepsilon,M}(x) dx  - \left[ A(f_{\varepsilon,M})(\indi{\mathcal{X}})  - i (\indi{\mathcal{X}}) \right] \\
		- (M-1) \left[ A(f_{\varepsilon,M})(r_{\varepsilon,M})  - 
		i (r_{\varepsilon,M}) \right].
	\end{multline*}
	Now, $A(f_{\varepsilon,M})(\indi{\mathcal{X}})=
	\int_{\mathcal{X}} f_{\varepsilon,M}(x) \indi{\mathcal{X}}(x) dx$ and 
	$i(\indi{\mathcal{X}})=0$ by the definition of $i$, as $\int_{\mathcal{X}} s_{\theta_0}(x) p_{\theta_0}(x)dx=0$.
	We deduce
	\begin{equation*}
		G(f_{\varepsilon,M},\mu_{\varepsilon,M})
		=  - (M-1) \left[ A(f_{\varepsilon,M})(r_{\varepsilon,M})  - 
		i (r_{\varepsilon,M}) \right].
	\end{equation*}
	As $	G(f_{\varepsilon,M},\mu_{\varepsilon,M}) \le \sup_{\mu \in \mathfrak{M}^{+,M}(C)} G(f_{\varepsilon,M},\mu) \le I^* + \varepsilon$, we get
	\begin{equation*}
		- (M-1) \left[ A(f_{\varepsilon,M})(r_{\varepsilon,M})  - 
		i (r_{\varepsilon,M}) \right] \le I^* + \varepsilon.
	\end{equation*}
	Using $A(f_{\varepsilon,M})(r_{\varepsilon,M}) - i(r_{\varepsilon,M})  < 0$, it gives 
	\begin{equation*}
		0<
		i (r_{\varepsilon,M}) - A(f_{\varepsilon,M})(r_{\varepsilon,M})   \le \frac{I^* + \varepsilon}{M-1}.
	\end{equation*}
	By definition of $r_{\varepsilon,M}$ maximizes $i(\cdot)- A(f_{\varepsilon,M})(\cdot)$, and we deduce that for all $r\in C$,
	\begin{equation} \label{eq : approx constraint}	 
		i (r) - A(f_{\varepsilon,M})(r)   \le \frac{I^* + \varepsilon}{M-1}.
	\end{equation}
	Now, we take $g : \mathcal{X} \to \mathbb{R}_+$ any non negative measurable function with $\int_{\mathcal{X}} g(x)dx=1$ and set 
	$g_{\varepsilon,M}= f_{\varepsilon,M} +  
	\frac{ I^* + \varepsilon }{ M-1 } g$. We prove that $g_{\varepsilon,M}$ satisfies the constraint of the dual program. Indeed, for $r \in C$,
	\begin{align*}
		\int_{\mathcal{X}} r(x) g_{\varepsilon,M}(x) dx
		&=\int_{\mathcal{X}} r(x)  \left[ f_{\varepsilon,M}(x)  + \frac{ I^* + \varepsilon }{ M-1 } g(x) \right] dx
		\\ &\ge \int_{\mathcal{X}} r(x) f_{\varepsilon,M}(x) dx  + 
		\frac{ I^* + \varepsilon }{ M-1 }  
		\int_{\mathcal{X}} g(x)  dx 
		\text{ as $r \ge 1$ and $g \ge 0$}.
	\end{align*}
	Using $\int_{\mathcal{X}} g =1$ and the definition of $A$, it is
	\begin{equation*}
		\int_{\mathcal{X}} r(x) g_{\varepsilon,M}(x) dx
		\ge A(f_{\varepsilon,M})(r) + 	\frac{ I^* + \varepsilon }{ M-1 }  .
	\end{equation*}
	Using \eqref{eq : approx constraint}, we deduce that for all $r\in C$,
	\begin{equation*}
		\int_{\mathcal{X}} r(x) g_{\varepsilon,M}(x) dx
		\ge i(r).
	\end{equation*}
	The definition of the dual program \eqref{eq : dual nos notations supp} implies that 
	\begin{equation*}
		\tilde{I}^* \le \int_{\mathcal{X}} g_{\varepsilon,M}(x) dx.
	\end{equation*}
	From the expression of $g_{\varepsilon,M}$, we deduce 
	\begin{equation*}
		\tilde{I}^* \le \int_{\mathcal{X}} f_{\varepsilon,M}(x) dx 
		+ \frac{ I^* + \varepsilon }{ M-1 } \int_{\mathcal{X}} g(x) dx.
	\end{equation*}
	Using again that $g$ is a probability density, 
	\begin{align} \nonumber
		\tilde{I}^* &\le \int_{\mathcal{X}} f_{\varepsilon,M}(x) dx 
		+ \frac{ I^* + \varepsilon }{ M-1 }
		\\\nonumber
		& = G( f_{\varepsilon,M}(x),0 )
		+ \frac{ I^* + \varepsilon }{ M-1 }
		\\ \nonumber & \le \sup_{\mu \in \mathfrak{M}^{+,M}(\mathcal{X})} 
		G( f_{\varepsilon,M}(x),\mu )
		+ \frac{ I^* + \varepsilon }{ M-1 }
		\\ \label{eq : tilde I le I plus eps plus inv M} & \le  I^* + \varepsilon + \frac{ I^* + \varepsilon }{ M-1 },
	\end{align}
	where we used \eqref{eq : sup G f ve M} in the last line.
	Gathering \eqref{eq : tilde I le I plus eps} and \eqref{eq : tilde I le I plus eps plus inv M}, and using that $\varepsilon>0$ and $M>1$ are arbitrary constants, the desired result follows.
\end{proof}
\subsection{Proof of the Point 2 in Proposition \ref{prop : dual no gap}}
{\revis In order to prove Point 2 of Proposition~\ref{prop : dual no gap}, we first show that any approximate minimizer of \eqref{eq : dual nos notations supp} is necessarily bounded in $L^1(\cX)$. This is established in Lemma~\ref{lem : constraint implies L1 bound} below. The desired result then follows as an immediate consequence of Lemma~\ref{prop: dual suppA recast lem}, whose statement and proof are given below.}

\begin{lem}\label{lem : constraint implies L1 bound}
	Assume that $f\in L^1(\mathcal{X})$ satisfies $\tilde{I}^* \le \int_\mathcal{X}f(x)dx \le \tilde{I}^*+1$ and $\int_{\mathcal{X}}f(x)r(x)dx \ge i(r)$ for all $r \in C$.
	Then, $\norm{f}_{L^1(\mathcal{X)}} \le (\tilde{I}^*+1)\frac{e^\alpha + 1}{e^\alpha-1}$.
\end{lem}
\begin{proof}
	Let $r(x)=1+(e^\alpha-1)\indi{\{f(x) \le 0\}}$, which defines an element of $C$. We have
	\begin{equation*}
		\int_\mathcal{X} f(x) r(x) dx =\int_\mathcal{X} f(x) dx + \left(e^\alpha-1\right)
		\int_\mathcal{X} f(x) \indi{\{f(x) \le 0\}} dx \ge i(r).
	\end{equation*}
	Since $i(r)\ge0$, we deduce that $-\int_\mathcal{X} f(x) \indi{\{f(x) \le 0\}} dx
	\le \frac{\int_\mathcal{X} f(x) dx}{e^\alpha-1} \le \frac{\tilde{I}^*+1}{e^\alpha-1}$.
	Now, we use that 
	\begin{align*}
		\norm{f}_{L^1(\mathcal{X)}}&=	
		\int_\mathcal{X} f(x) \indi{\{f(x) > 0\}} dx
		-\int_\mathcal{X} f(x) \indi{\{f(x) \le 0\}}\\
		& = 	\int_\mathcal{X} f(x) dx
		-	2 \int_\mathcal{X} f(x) \indi{\{f(x) \le 0\}} \\
		& \le \tilde{I}^* + 1 + 2\left(\frac{\tilde{I}^* + 1}{e^\alpha - 1}\right).
	\end{align*}
	It entails
	\begin{equation*}
		\norm{f}_{L^1(\mathcal{X})} \le (\tilde{I}^*+1) \left(1+\frac{2}{e^\alpha-1}\right).
	\end{equation*}
	The lemma is proved.
\end{proof}

{\revis With these preliminary results in place, Point 2 of Proposition~\ref{prop : dual no gap} follows directly from the following lemma.}
\begin{lem}{\label{prop: dual suppA recast lem}}
	There exists $f \in L^1(\mathcal{X})$ such that $\int_\mathcal{X} f(x)dx = \tilde{I}^*$ and
	$A(f)(r)\ge i(r)$ for all $r \in C$.
\end{lem}
\begin{proof}
	From Lemma \ref{lem : constraint implies L1 bound}, there exists a sequence $(f_n)_{n\ge 1}$ bounded in $L^1(\mathcal{X})$ satisfying the constraint $A(f_n)(\cdot) \ge i(\cdot)$ for all $n\ge 1$ and with $\int_{\mathcal{X}} f_n(x)dx \xrightarrow{n\to\infty} \tilde{I}^*$.  
	As balls of $L^1(\mathcal{X})$ are nor strongly, nor weakly compact, we cannot directly extract a subsequence converging to some element of $L^1(\mathcal{X})$. We need to embed the sequence into a compact set of measures. 
	Let us denote $\widehat{\mathcal{X}}=\mathcal{X} \sqcup \{x_\infty\}$ the one point Alexandroff compactification of $\mathcal{X}$. Since $\mathcal{X}$ is an open set of
	$\mathbb{R}^d$, its compactification is a metric space and let us denote $\hat{d}$ the associated metric (see \cite{mandelkern} or Chapter XI in \cite{dugundjiTopology1966}).  
	
	For $n\ge 1$, we define $\nu_n(dx)=f_n(x)dx$ the signed measure on $\mathcal{X}$, that we consider as a measure on $\widehat{\mathcal{X}}$ with zero mass on $\{x_\infty\}$.  Using that $(f_n)_n$ is bounded in $L^1(\mathcal{X})$, the sequence $\nu_n$ is bounded in total variation. As the space $\widehat{\mathcal{X}}$ is compact, $(\nu_n)_n$ is a tight sequence of measures on $\widehat{\mathcal{X}}$. We can extract a subsequence converging weakly. For notational simplicity, and as $(f_n)_n$ is any minimizing sequence of the dual program, we omit the subsequence step and consider that the initial sequence is such that $\nu_n$ is weakly converging. Hence, there exists $\nu$ a measure on $\widehat{\mathcal{X}}$ 
	such that 
	\begin{equation*}
		\int_{\mathcal{X}} g(x) f_n(x) dx = \int_{\widehat{\mathcal{X}}} g(x) \nu_n(dx) \xrightarrow{n\to\infty}
		\int_{\widehat{\mathcal{X}}} g(x) \nu(dx) ,
	\end{equation*}
	for any function $g : \widehat{\mathcal{X}} \to \mathbb{R}$ bounded and continuous. 
	Taking $g \equiv 1$, and using that $(f_n)_n$ is a minimizing sequence, we deduce  
	\begin{equation} \label{eq : nu is minimizer}
		\tilde{I}^*=\lim_{n\to\infty}\int_{\mathcal{X}} f_n(x) dx = \int_{\widehat{\mathcal{X}}}  \nu(dx) .
	\end{equation}
	The idea is now to split the measure $\nu$ into its absolutely continuous part and a singular part. By dispatching the mass of the singular part into a density, we will construct some function $f \in L^1(\mathcal{X})$ solution of the dual program.
	
	By the Radon-Nikodym-Lebesgue theorem, we write $\nu(dx)=\nu_0(x)dx+\nu_1(dx)$ where $\nu_0$ is some integrable function and $\nu_1$ is a measure on $\widehat{\mathcal{X}}$ singular with respect to the Lebesgue measure.  Using Hahn-Jordan's theorem, we decompose the singular part into its positive and negative part. It yields $\nu(dx)=\nu_0(x)dx+\nu_{1,+}(dx)-\nu_{1,-}(dx)$. The three measures are mutually singular so that we have the partition $\widehat{\mathcal{X}}=\widehat{\mathcal{X}}_0 \sqcup \widehat{\mathcal{X}}_{1,+}\sqcup \widehat{\mathcal{X}}_{1,-}$ and 
	$\nu_0(x) dx= \nu_0(x) \indi{\widehat{\mathcal{X}}_0}(x)dx$, 
	$\nu_{1,+}(\cdot)=\nu_{1,+}(\cdot \cap \widehat{\mathcal{X}}_{1,+})$
	$\nu_{1,-}(\cdot)=\nu_{1,-}(\cdot \cap \widehat{\mathcal{X}}_{1,-})$. Furthermore, $\widehat{\mathcal{X}}_{1,+}$ and $\widehat{\mathcal{X}}_{1,-}$ have zero Lebesgue measure.
	
	Let $\phi : \mathcal{X}\to\mathbb{R}_+$ be any probability density and we set for $x\in \mathcal{X}$, $f(x) = \nu_0(x) +  \nu_{1,+}(\widehat{\mathcal{X}}_{1,+}) \phi(x) - \nu_{1,-}(\widehat{\mathcal{X}}_{1,-}) \phi(x)$. As $\int_\mathcal{X} \phi =1$, we have
	\begin{equation}  \label{eq : f is minimzer}
		\int_{\mathcal{X}} f(x)dx=\int_{\mathcal{X}} \nu_0(x)dx+ \nu_{1,+}(\widehat{\mathcal{X}}_{1,+}) -  \nu_{1,-}(\widehat{\mathcal{X}}_{1,-}) 
		=\nu(\widehat{\mathcal{X}}) = \tilde{I}^*,	
	\end{equation}
	where we used \eqref{eq : nu is minimizer}.
	
	It remains to check the constraint $A(f)(r)\ge i(r)$ for all $r\in C$.
	Let $K$ be any compact subset of $\mathcal{X}$ and set $r_K = 1 + (e^{\alpha}-1) \indi{K}$, and
	let us temporarily admit the following equation, which is proved in Lemma \ref{lem : ineg crucial existence f} below:
	\begin{equation}\label{eq : constraint K compact the more neg}
		\tilde{I}^* + (e^\alpha-1) \int_\mathcal{X} \nu_0(x)\indi{K}(x)dx - (e^\alpha-1) \nu_{1,-}(\widehat{\mathcal{X}}_{1,-}) \ge i(r_K).
	\end{equation} 
	
	
	Then,
	\begin{align*}
		A(f)(r_K)&=\int_{\mathcal{X}}f(x) [1+(e^\alpha-1)\indi{K}(x) ] dx	
		\\ &= \tilde{I}^*+ (e^\alpha-1) \int_{K} f(x) dx,\quad \text{where we used \eqref{eq : f is minimzer}},
		\\ 
		&=\tilde{I}^*+ (e^\alpha-1)  \left[ 
		\int_{K} \nu_0(x) dx +  \nu_{1,+}(\widehat{\mathcal{X}}_{1,+}) \int_K \phi(x) dx
		-\nu_{1,-}(\widehat{\mathcal{X}}_{1,-}) \int_K \phi(x) dx \right],				   
	\end{align*}
	by the definition of $f$. Since $\phi\ge0$, and $\int_\mathcal{X} \phi =1$, we deduce that
	$A(f)(r_K) \ge \tilde{I}^{*} +  (e^\alpha-1)\left[	\int_{K} \nu_0(x) dx  - \nu_{1,-}(\widehat{\mathcal{X}}_{1,-})  \right]$. Then, \eqref{eq : constraint K compact the more neg} implies $	A(f)(r_K) \ge i(r_K)$. The constraint is also valid for $r_F=1 + (e^{\alpha}-1) \indi{F}$ where $F$ is a borel set of $\mathcal{X}$. Indeed, by regularity of the Lebesgue measure it is possible to approximate $\int_F f(x)dx$, for $F$ Borel set, by 
	$\int_K f(x)dx$, with $K$ being compact. It makes $	A(f)(r_F)$ approximated, arbitrarily close, by $	A(f)(r_K)$ for some compact $K$. In the same way, 
	{\revarna recalling the expression of $i(r)$ in \eqref{Eq : rappel expression i supp},} 
	we see that $i(r_F)$ can be approximated by $i(r_K)$. This shows that for all $r\in\mathcal{E}$, we have $	A(f)(r) \ge i(r)$.
	
	Now, in order to show that the constrain is satisfied for $r\in C$, we use the Choquet theorem (see Theorem \ref{th: Choquet}) and write $r=\int_\mathcal{E} r' \mu(dr')$ for $\mu$ some probability supported on $\mathcal{E}$. Then,
	\begin{align*}
		A(f)(r)&=\int_{\mathcal{X}} f(x)  r(x) dx =\int_{\mathcal{X}}f(x)  e_x(r) dx
		\\&= \int_{\mathcal{X}} f(x)  \int_{\mathcal{E}} e_x(r') \mu(dr') dx, \quad \text{using \eqref{eq : evalutation rigorous}} 
		\\&= \int_{\mathcal{E}} \int_{\mathcal{X}} f(x)  e_x(r') dx   \mu(dr'), \quad \text{by Fubini's theorem} 
		\\&= \int_{\mathcal{E}} \int_{\mathcal{X}} f(x)  r'(x) dx\mu(dr') 
		= \int_{\mathcal{E}}  A(f)(r') \mu(dr').
	\end{align*}
	Using $A(f)(r')\ge i(r')$ for $r'\in\mathcal{E}$, it yields 
	$	A(f)(r) \ge  \int_{\mathcal{E}}  i(r') \mu(dr')$. As $i$ is a convex function, we deduce
	$A(f)(r) \ge  i\left( \int_{\mathcal{E}}  r' \mu(dr') \right) = i(r)$. Hence, the function $f$ satisfies the constraint. The lemma is proved.
\end{proof}
\begin{lem} \label{lem : ineg crucial existence f}
	Let $K$ be a compact subset of $\mathcal{X}$ and set $r_K = 1 + (e^{\alpha}-1) \indi{K}$, then
	\begin{equation*}
		\tilde{I}^* + (e^\alpha-1) \int_\mathcal{X} \nu_0(x)\indi{K}(x)dx - (e^\alpha-1) \nu_{1,-}(\widehat{\mathcal{X}}_{1,-}) \ge i(r_K).
	\end{equation*} 
\end{lem}
\begin{proof}
	
	First, we utilize the regularity of Borel measures on $\widehat{\mathcal{X}}$. For $\varepsilon>0$, we can find $I_\varepsilon^- \subset \mathcal{X}_{1,-} \subset U_\varepsilon^-$ and $I_\varepsilon^+ \subset \mathcal{X}_{1,+} \subset U_\varepsilon^+$, where $I_\varepsilon^-$,  $I_\varepsilon^+$ are compact sets and $U_\varepsilon^-$,  $U_\varepsilon^+$ open sets, such that $\eta(U_\varepsilon^- \setminus  I_\varepsilon^-) \le \varepsilon$ and $\eta(U_\varepsilon^+ \setminus  I_\varepsilon^
	+) \le \varepsilon$ when $\eta$ is any of the measures in the set $\{\abs{\nu_0}(x)dx,\nu_{1,+}(dx),\nu_{1,-}(dx),dx\}$. 
	
	Let $K$ be a compact subset of $\mathcal{X}$ and thus of $\widehat{\mathcal{X}}$. We set
	$K_0=K \cap \left(U_\varepsilon^+ \cup U_\varepsilon^-  \right)^c$. The set $K_0$ is compact and disjoint from $I_{\varepsilon}^-$. Hence, the distance between these two sets is strictly positive. We can construct $\indi{K_0}^{\text{ctn}}(x)$  which is a continuous approximation of $\indi{K_0}(x)$ by setting  $\indi{K_0}^{\text{ctn}}(x)=(1-M \hat{d}(x,K_0))_+$ for $M>0$ large. It satisfies $0\le \indi{K_0}^{\text{ctn}} \le 1$, $\indi{K_0}^{\text{ctn}}(x)=1$ for $x \in K_0$ and $\indi{K_0}^{\text{ctn}}(x)=0$ for $x$ at distance larger than $1/M$ of $K_0$. We introduce also the approximation of $\indi{I_\varepsilon^-}$ given by
	$\indi{I_\varepsilon^-}^{\text{ctn}}(x) = (1-M \hat{d}(x,I_\varepsilon^-))_+$. Choosing $M$ large enough the support of $\indi{K_0}^{\text{ctn}}$ and $\indi{I_\varepsilon^-}^{\text{ctn}}$ are disjoint.
	We let 
	\begin{equation}{\label{eq: r epsilon m}}
		r_{\epsilon,M}(x)=1 + (e^\alpha-1) \left(\indi{K_0}^{\text{ctn}}(x)+ \indi{I_\varepsilon^-}^{\text{ctn}}(x)\right).  
	\end{equation}
	Then, $r_{\epsilon,M}$ is continuous and takes values in $[1,e^\alpha]$. Using that $(f_n)_n$ satisfies the constraint of the dual program and $r_{\epsilon,M} \in C$, we can write
	\begin{equation*}
		\int_{\widehat{\mathcal{X}}} r_{\epsilon,M}(x) \nu_n(dx)=
		\int_{\widehat{\mathcal{X}}} r_{\epsilon,M}(x) f_n(x) dx \ge i(r_{\varepsilon,M}).
	\end{equation*}
	As $\nu_n$ converges weakly to $\nu$, and $r_{\varepsilon,M}$ is continuous, we can take the limit for $n\to\infty$ and deduce
	\begin{equation*}
		\int_{\widehat{\mathcal{X}}} r_{\epsilon,M}(x) \nu(dx) \ge i(r_{\varepsilon,M}).
	\end{equation*}
	Recalling \eqref{eq : nu is minimizer}, using \eqref{eq: r epsilon m} and the decomposition $\nu(dx)=\nu_0(x)dx+\nu_{1,+}(dx)-\nu_{1,-}(dx)$ it gives
	\begin{multline*}
		\tilde{I}^*+ \left(e^\alpha-1\right) \int_{\widehat{\mathcal{X}}}  \indi{K_0}^{\text{ctn}}(x) \left[\nu_0(x)dx+\nu_{1,+}(dx)-\nu_{1,-}(dx)\right]
		\\
		+
		\left(e^\alpha-1\right) \int_{\widehat{\mathcal{X}}}  \indi{I_\varepsilon^-}^{\text{ctn}}(x) \left[\nu_0(x)dx+\nu_{1,+}(dx)-\nu_{1,-}(dx)\right]
		\ge i(r_{\varepsilon,M}).
	\end{multline*}
	Now, 
	using that $K_0$ and $I_\varepsilon^-$ are compact set we have for all $x\in \widehat{\mathcal{X}}$, $\indi{K_0}^{\text{ctn}}(x)\xrightarrow{M\to\infty}
	\indi{K_0}(x)$ and 
	$\indi{I_\varepsilon^-}^{\text{ctn}}(x)\xrightarrow{M\to\infty}
	\indi{I_\varepsilon^-}(x)$. We deduce
	\begin{multline*}
		\tilde{I}^*+ \left(e^\alpha-1\right) \int_{\widehat{\mathcal{X}}}  \indi{K_0}(x) \left[\nu_0(x)dx+\nu_{1,+}(dx)-\nu_{1,-}(dx)\right]
		\\
		+
		\left(e^\alpha-1\right) \int_{\widehat{\mathcal{X}}}  \indi{I_\varepsilon^-}(x) \left[\nu_0(x)dx+\nu_{1,+}(dx)-\nu_{1,-}(dx)\right]
		\ge i\left( r_{\varepsilon} \right),
	\end{multline*}
	where $r_\varepsilon=1 + (e^\alpha-1) \left(\indi{K_0}+\indi{I_\varepsilon^-}\right)$.
	
	By definition, $K_0 \subset \left(U_\varepsilon^+ \cup U_\varepsilon^-  \right)^c$ does not intersect the support of the measures $\nu_{1,+}$ and $\nu_{1,-}$. As $I_{\varepsilon,-} \subset \mathcal{X}_{1,-}$, it does not intersect the support of {\revarna $\nu_0(x)dx$} or $\nu_{1,+}(dx)$. 
	
	It yields the simplification
	\begin{equation*}
		\tilde{I}^*+ \left(e^\alpha-1\right) \int_{\mathcal{X}}  \indi{K_0}(x) \nu_0(x)dx
		-
		\left(e^\alpha-1\right) \int_{\widehat{\mathcal{X}}}  \indi{I_\varepsilon^-}(x)\nu_{1,-}(dx)
		\ge i\left( r_{\varepsilon} \right).
	\end{equation*}
	As $K_0 \subset K$, with $K \setminus K_0 \subset (U_\varepsilon^+ \cup U_\varepsilon^-)$, we have 
	\begin{multline*}
		\abs{\int_{\mathcal{X}}  \indi{K}(x) \nu_0(x)dx - \int_{\mathcal{X}}  \indi{K_0}(x) \nu_0(x)dx} \le \int_{U_\varepsilon^-}\abs{\nu_0(x)}dx + \int_{U_\varepsilon^+} \abs{\nu_0(x)}dx \le 2 \epsilon 
		\\
		+
		\int_{\widehat{\mathcal{X}}_{1,-}}\abs{\nu_0(x)}dx + \int_{\widehat{\mathcal{X}}_{1,+}}\abs{\nu_0(x)}dx=2\varepsilon,
	\end{multline*}
	where we used that $\widehat{\mathcal{X}}_{1,-}$ and $\widehat{\mathcal{X}}_{1,+}$ have zero Lebesgue measure. Also, we have $\nu_{1,-}(\widehat{\mathcal{X}}_{1,-} \setminus I_{\varepsilon}^-) \le \varepsilon$. 
	This implies
	\begin{equation} \label{eq : lem about existence of f up to epsilon}
		\tilde{I}^*+ \left(e^\alpha-1\right) \int_{\mathcal{X}}  \indi{K}(x) \nu_0(x)dx
		- \left(e^\alpha-1\right) \nu_{1,-}(\widehat{\mathcal{X}}_{1,-} )
		+ 2\left(e^\alpha-1\right)\varepsilon
		\ge i\left( r_{\varepsilon} \right),
	\end{equation}
	where $\varepsilon>0$ is arbitrary. Notice that
	\begin{equation*}
		i(r_\varepsilon)=
		(e^\alpha - 1)^2 \frac{ \left( \int_{K_0 \sqcup I_\varepsilon^- }  s_{\theta_0}(x) \, p_{\theta_0}(x) \, dx \right)^2 }{ 1 + (e^\alpha - 1)
			\int_{K_0 \sqcup I_\varepsilon^- } p_{\theta_0}(x) \, dx }.
	\end{equation*}
	As $I_\varepsilon^- $ is a set with zero Lebesgue measure, it can be removed from the above expression. Moreover, we have seen that the Lebesgue measure of $K\setminus K_0$ is lower than $2\varepsilon$. Hence, if we let $\varepsilon \to 0$, we have
	\begin{equation*}
		i(r_\varepsilon) \to i(r_K), \quad \text{with $r_K = 1 + (e^{\alpha}-1) \indi{K}$}.
	\end{equation*}
	Together with \eqref{eq : lem about existence of f up to epsilon}, and taking $\varepsilon \to 0$, it implies the lemma.
\end{proof}

}

\section{Proofs of technical results}{\label{s: proof technical}}
This section presents the detailed proofs of the properties stated in Section \ref{S: tools}, as well as the proof of some technical lemmas needed to prove our main results.
\subsection{Proofs of technical results in functional analysis}
We begin by proving that the set $\theconvex$ is compact with respect to the weak-$\star$ topology, as stated in Lemma \ref{l: B closed}. We then proceed to Lemma \ref{l: B separable}, which asserts that $\widehat{B}$, equipped with the weak-$\star$ topology, is separable. This separability is crucial to ensuring that the integral $\int_{\widehat{B}} f \, \mu(df)$, introduced in Lemma \ref{l: Bochner}, is well-defined in the sense of Bochner. Establishing this will, in turn, allow us to rigorously justify the pointwise evaluation given in Proposition \ref{p: evaluation}. The proof of this proposition, which is the most technically involved part of this section, will conclude our exposition of the technical tools.

\subsubsection{Proof of Lemma \ref{l: B closed}}
\begin{proof}
Let us begin by proving that the set \( \theconvex \) is closed with respect to the weak-\(\star\) topology. Let \( (f_n)_n \subset C \) be a sequence such that \( f_n \xrightharpoonup{\star} f \). Since \( C \subset \widehat{B}=B_{\|\cdot\|_\infty}(0, e^{\alpha}) \), and the closed ball $\widehat{B}$ 
is compact in the weak-\(\star\) topology (by the Banach–Alaoglu Theorem, as stated in Theorem~\ref{th: BanAla}, together with the fact that dilations are continuous mappings in the strong topology, hence also in the weak and weak-\(\star\) topologies), it follows that $ f \in \widehat{B}$. 
To conclude that \( f \in \theconvex \), it remains to show that \( f(x) \geq 1 \) almost everywhere on \( \mathcal{X} \). 

Recall that, by the definition of weak-\(\star\) convergence, for all \( \varphi \in L^1(\mathcal{X}) \), we have:
\[
\langle \varphi, f_n \rangle_{L^1, L^\infty} \longrightarrow \langle \varphi, f \rangle_{L^1, L^\infty}, \quad \text{as } n \to \infty.
\]
Since each \( f_n \in \theconvex \), we have \( f_n(x) \geq 1 \) almost everywhere, and thus for $\varphi \in L^1(\cX)$, $\varphi \ge0$ :
\[
\int_{\mathcal{X}} \varphi(x) \cdot 1 \, dx \leq \int_{\mathcal{X}} \varphi(x) f_n(x) \, dx.
\]
Passing to the limit, we obtain:
\[
\int_{\mathcal{X}} \varphi(x) \, dx \leq \int_{\mathcal{X}} \varphi(x) f(x) \, dx,
\]
which implies:
\[
\int_{\mathcal{X}} \varphi(x) [f(x) - 1] \, dx \geq 0, \quad \forall \varphi \in L^1(\mathcal{X}) \text{ with $\varphi \ge 0$.}
\]
This inequality shows that \( f(x) \geq 1 \) almost everywhere, as desired.

We conclude that \( f \in \theconvex \), hence \( \theconvex \) is closed in the weak-\(\star\) topology. Since \( \theconvex \) is also contained in a weak-\(\star\) compact set, it follows that \( \theconvex \) is compact in the weak-\(\star\) topology.
\end{proof}

\subsubsection{Proof of Lemma \ref{l: B separable}}
\begin{proof}
The proof of this lemma consists in explicitly constructing a countable sequence that is dense in \( \widehat{B} \). It is well known that \( L^1(\cX) \) is separable, so we may fix a countable dense sequence \( (f_n)_{n \in \mathbb{N}} \subset L^1(\cX) \). Define the truncated functions \( \tilde{f}_n := -e^\alpha \lor f_n \land e^\alpha \), which belong to \( \widehat{B} \). Our goal is to show that the sequence \( (\tilde{f}_n)_{n \in \mathbb{N}} \) is dense in \( \widehat{B} \) with respect to the weak-\(\star\) topology.

To this end, let \( g \in \widehat{B} \). Consider the truncations \( g_M := g \cdot \indi{\{|x| \le M\}} \). It satisfies $(-e^\alpha) \vee g_M \wedge e^\alpha$
$\xrightharpoonup{\star} g$  as \( M \to \infty \). Indeed, the sequence $(-e^\alpha)\vee g_M(x) \wedge e^\alpha$ converges to  $(-e^\alpha) \vee g(x) \wedge e^\alpha=g(x)$ for almost every $x \in \cX$.
Hence, the dominated convergence theorem can be applied, ensuring that for any \( f \in L^1(\mathcal{X}) \),
\[
\int_{\cX} f(x) (-e^\alpha)\vee g_M(x) \wedge e^\alpha \, dx \to \int_{\cX} f(x) g(x) \, dx \qquad \mbox{ for } M\rightarrow \infty.
\]
Thus, it suffices to approximate each \( (-e^\alpha )\vee g_M \wedge e^\alpha \) by elements of the sequence \( (\tilde{f}_n)_{n \in \mathbb{N}} \).

Observe that \( g_M \in L^1(\cX) \). Since \( (f_n) \) is dense in \( L^1(\cX) \), there exists a strictly increasing map \( j : \mathbb{N} \to \mathbb{N} \) such that \( f_{j(n)} \to g_M \) in \( L^1(\cX) \). Up to passing to a subsequence, we may also assume \( f_{j(n)}(x) \to g_M(x) \) for almost every \( x \in \cX \). 
It follows that
\[
\tilde{f}_{j(n)}(x) = (-e^\alpha) \lor f_{j(n)}(x) \land e^\alpha \to (-e^\alpha)\vee g_M(x) \wedge e^\alpha  \quad \text{a.e. in } \cX,
\]
and \( \tilde{f}_{j(n)} \in B \) for all \( n \). Applying the dominated convergence theorem once more, we conclude that for all \( \varphi \in L^1(\cX) \),
\[
\int_{\cX} \varphi(x) \tilde{f}_{j(n)}(x) \, dx \to \int_{\cX} \varphi(x)  \left( (-e^\alpha)\vee g_M(x) \wedge e^\alpha \right) \, dx,
\]
so that \( \tilde{f}_{j(n)} \xrightharpoonup{\star}  (-e^\alpha)\vee g_M \wedge e^\alpha \). This proves the lemma.  
\end{proof}

\subsubsection{Proof of Lemma \ref{l: Bochner}}
\begin{proof}

Using Lemma \ref{l: B separable}, we let \( (f_n)_{n \in \mathbb{N}} \) be a countable sequence dense in \( \widehat{B} \) for the $d_\star$ distance. Fix \( M \geq 1 \) and define a sequence of subsets of \( \widehat{B} \) as follows:
\[
\begin{aligned}
	B_0^{(M)} & := \left\{ f \in \widehat{B} \,:\, d_\star(f, f_0) < \frac{1}{M} \right\}, \\
	B_1^{(M)} & := \left\{ f \in \widehat{B} \,:\, d_\star(f, f_1) < \frac{1}{M} \right\} \setminus B_0^{(M)}, \\
	& \ \, \vdots \\
	B_n^{(M)} & := \left\{ f \in \widehat{B} \,:\, d_\star(f, f_n) < \frac{1}{M} \right\} \setminus \bigcup_{j=0}^{n-1} B_j^{(M)},
\end{aligned}
\]
and so on, for all \( n \in \mathbb{N} \). By construction, these sets are disjoint and, since \( (f_n)_{n \in \mathbb{N}} \) is dense in \( \widehat{B} \), we have \( \bigcup_{n=0}^\infty B_n^{(M)} = \widehat{B} \). 
Moreover, by compactness of $\widehat{B}$, we know that any covering of $\widehat{B}$ by open balls of radius $1/M$ admits a finite covering. Thus, for $n$ large enough, the set $B^{(M)}_n$ are empty. In turn,
the union $\bigcup_{n=0}^\infty B_n^{(M)}$ is finite partition of \( \widehat{B} \), which can be used to approximate integrals on \( \widehat{B} \) via simple functions. 

For each \( f \in \widehat{B} \), define the function
\begin{equation}{\label{eq: Id f}}
	I_d^{(M)}(f) := \sum_{n = 0}^\infty f_n \indi{B_n^{(M)}}(f),
\end{equation}
where the above sum is actually on a finite number of indexes, depending on $M>0$.
By construction, \( I_d^{(M)}(f) \in \widehat{B} \), and
\[
d_\star\left( I_d^{(M)}(f), f \right) \le \frac{1}{M}.
\]
Furthermore, the integral of the simple map \(  f \mapsto I_d^{(M)}(f) \) is 
given explicitly by a sum :
\begin{equation}{\label{eq: int Id f}}
	\int_{\widehat{B}} I_d^{(M)}(f) \, \mu(df) = \sum_{n = 0}^\infty f_n \mu\left(B_n^{(M)}\right).
\end{equation}
Since each \( f_n \in \widehat{B} \), we have
\[
\left\| \int_{\widehat{B}} I_d^{(M)}(f) \, \mu(df) \right\|_\infty \le \sup_n \|f_n\|_\infty \sum_{n = 0}^\infty \mu(B_n^{(M)}) \le e^\alpha,
\]
which implies that \( \int_{\widehat{B}} I_d^{(M)}(f) \, \mu(df) \in \widehat{B} \).

Thus, we have constructed a sequence of simple functions that approximate the integral in the sense of Bochner:
\[
\lim_{M \to \infty} \int_{\widehat{B}} d_\star\left( I_d^{(M)}(f), f \right) \, \mu(df) = 0,
\]
as required by the definition of the Bochner integral. This completes the proof.
\end{proof}
\begin{rem}{\label{rk: Cauchy}}
For $M > 0$ and $H > 0$, we have $d_\star\left( I_d^{(M)}(f), I_d^{(M+H)}(f) \right) \le \frac{2}{M}.$
From the {\revarna expression} 
of $d_\star$ {\revarna  given in \eqref{eq: d ball}} and the fact that $\mu(\widehat{B}) \le 1$, it is possible to check that
\[
d_\star \left( \int_{\widehat{B}} I_d^{(M)}(f) \, \mu(df), \int_{\widehat{B}} I_d^{(M+H)}(f) \, \mu(df) \right) \le \frac{2}{M},
\]
which shows that the sequence of approximating integrals is Cauchy. Moreover, we obtain
\begin{equation}\label{eq : approx Bochner d star}
	d_\star \left( \int_{\widehat{B}} I_d^{(M)}(f) \, \mu(df), \int_{\widehat{B}} f \, \mu(df) \right) \le \frac{2}{M},
\end{equation}
a bound that will be useful in the next proof.
\end{rem}
\subsubsection{Proof of Proposition \ref{p: evaluation}}
\begin{proof}
To prove the desired result, we proceed through a structured argument split into four key steps, which we outline below and then establish rigorously.
%
\\
{Step 1:} We will prove that the map $(f,x) \mapsto e_x(f)$ is bi-measurable, as in Point 1. 
\\
{Step 2:} We will prove the identity \eqref{eq : evalutation rigorous} for simple functions, specifically for the approximations \( I_d(M)(f) \) used to define the Bochner integral. 
\\
{Step 3:} We will pass to the limit in these approximations, applying dominated convergence to conclude the desired pointwise identity, concluding the proof of Point 2. 
\\
{Step 4:} We will obtain the negligibility of the set \eqref{eq : neg by Fubini}, as stated in Point 3, by using Fubini theorem.

\textbf{Step 1:} Let us  denote
\begin{equation} \label{eq : def approx of evaluation}
	u_n(f,x) := \frac{1}{|B(x,1/n)|} \int_{B(x,1/n)\cap\cX} f(y) \, dy.
\end{equation}
We have $e_x(f)=\limsup_{n \to \infty} u_n(f)$ for all $(f,x) \in \widehat{B}\times\cX$.  Hence, it is sufficient to prove that $(r,x) \mapsto u_n(f,x)$ is measurable.	
We analyze separately the continuity properties of the numerator \( u_n^{(1)}(r,x) \) and the denominator \( u_n^{(2)}(x)= |B(x,1/n)|\). We begin with
\[
u_n^{(1)}(f,x) = \int_{B(x, \frac{1}{n})\cap\cX} f(y) \, dy = \left\langle \mathbbm{1}_{B(x, \frac{1}{n})}, r \right\rangle_{L^1, L^\infty}.
\]
To show continuity of \( u_n^{(1)} \), consider a sequence \( (f_p, x_p) \to (r,x) \) as \( p \to \infty \), where \( f_p \to f \) in the \( d_\star \)-topology on \( \widehat{B} \), and \( x_p \to x \) in the Euclidean topology on \( \mathbb{R}^d \). We estimate
\begin{multline}
	|u_n^{(1)}(f,x) - u_n^{(1)}(f_p,x_p)| 
	\le \int_{\cX} \left| \mathbbm{1}_{B(x, \frac{1}{n})} - \mathbbm{1}_{B(x_p, \frac{1}{n})} \right| |f_p(y)| \, dy 
	\\
	+\left| \int_{B(x, \frac{1}{n})\cap\cX} [f(y) - f_p(y)] \, dy \right| 
	= I_1 + I_2.
\end{multline}
Since \( f_p \xrightharpoonup{\star} f \) in \( L^\infty \), we have
\[
I_2 = \left\langle \mathbbm{1}_{B(x, \frac{1}{n})}, f - f_p \right\rangle \longrightarrow 0
\quad \text{as } p \to \infty.
\]
To control \( I_1 \), observe that for all \( p \), \( f_p \in \widehat{B} \) and thus satisfy \( \|f_p\|_{L^\infty} \le e^{\alpha} \). Hence,
\[
I_1 \le e^{\alpha} \int_{\mathbb{R}^d} \left| \mathbbm{1}_{B(x, \frac{1}{n})} - \mathbbm{1}_{B(x_p, \frac{1}{n})} \right| \, dy,
\]
which tends to zero as \( x_p \to x \), since the symmetric difference of the balls \( B(x, \frac{1}{n}) \) and \( B(x_p, \frac{1}{n}) \) vanishes in measure.

A similar argument applies to \( u_n^{(2)}(x) := |B(x, \frac{1}{n})| \), and we omit the details. We thus conclude that both \( u_n^{(1)} \) and \( u_n^{(2)} \) are continuous, and therefore \( (f,x)\mapsto u_n(f,x) \) is continuous as a quotient of continuous functions (with strictly positive denominator).

Since \( u_n \) is continuous on \( \widehat{B} \times \cX \), it is measurable with respect to the Borel \( \sigma \)-algebra $\Borelof{\widehat{B} \times \cX}$. As both \( \widehat{B} \) and \( \cX \) are metric spaces, we have
\[
\Borelof{\widehat{B} \times \mathbb{R}^d} = \Borelof{\widehat{B}} \otimes \Borelof{\mathbb{R}^d},
\]
so \( u_n \) is measurable with respect to the product \( \sigma \)-algebra. 
This step proves the point 1) of the proposition.
\\
\\
\textbf{Step 2:}  Let $\mu$ be a sub-probability on the space $\widehat{B}$ and assume that $b=\int_{\widehat{B}} f \mu(df)$. 	
In this step, we aim to establish \eqref{eq : evalutation rigorous} but with $I_d^{(M)}(f)$, as defined in \eqref{eq: Id f} instead of $f$. 

Recall from \eqref{eq: int Id f} that
\[
\int_{\widehat{B}} I_d^{(M)}(f) \, \mu(df) = \sum_{n=0}^\infty f_n \, \mu\left(B_n^{(M)}\right),
\]
where the equality holds in \( L^\infty(\cX) \). This implies that for almost every \( x \in \cX \),
\[
\left( \int_{\widehat{B}} I_d^{(M)}(f) \, \mu(df) \right)(x) = \sum_{n=0}^\infty f_n(x) \, \mu\left(B_n^{(M)}\right).
\]
Moreover, since for almost every \( x \in \cX \) it holds that
\[
I_d^{(M)}(f)(x) = \sum_{n=0}^\infty f_n(x) \, \mathbbm{1}_{B_n^{(M)}}(f),
\]
we have
\[
\int_{\widehat{B}} I_d^{(M)}(f)(x) \, \mu(df) = \sum_{n=0}^\infty f_n(x) \, \mu\left(B_n^{(M)}\right),
\]
so that
\[
\left( \int_{\widehat{B}} I_d^{(M)}(f) \, \mu(df) \right)(x) = \int_{\widehat{B}} I_d^{(M)}(f)(x) \, \mu(df).
\]
This is precisely the result we intended to prove, but for the case of simple functions. The goal now is to extend this identity to the Bochner integral
\[
\int_{\widehat{B}} f \, \mu(df),
\]
which will be the focus of the next step. 
\\
\\
\textbf{Step 3.} Let us introduce the following notations:
\[
u^{(M)} := \int_{\widehat{B}} I_d^{(M)}(f) \, \mu(df) ,  \quad u:= \int_{\widehat{B}} f \, \mu(df).
\]
From equation \eqref{eq : approx Bochner d star} following the proof of Lemma \ref{l: Bochner}, we have $u^{(M)} \xrightharpoonup{\star} u$ as $M \rightarrow \infty$.
%
%

Thanks to this convergence, we have, for any \( h \in L^1(\cX) \cap L^\infty(\cX) \),
\begin{align}\nonumber
	\int_{\cX} h(x) u(x) \, dx 
	&= \lim_{M \rightarrow \infty} \int_{\cX} h(x) u^{(M)}(x) \, dx \\
	\nonumber
	&= \lim_{M \rightarrow \infty} \int_{\cX} h(x) \left( \int_{\widehat{B}} I_d^{(M)}(f) \, \mu(df) \right)(x) \, dx \\
	\label{eq : h scalaire u proof prop eval}
	&= \lim_{M \rightarrow \infty} \left( \int_{\cX} \int_{\widehat{B}} h(x) I_d^{(M)}(f)(x) \, \mu(df)\right) \, dx,
\end{align}
where in the last equality we used Step 2.

From the definition of \( I_d^{(M)}(f) \), it follows that the function
\[
(x, f) \mapsto h(x) I_d^{(M)}(f)(x) = \sum_{n=1}^\infty h(x) f_n(x) \, \mathbbm{1}_{B_n^{(M)}}(f)
\]
is bi-measurable and bounded. Therefore, we can apply Fubini's theorem to obtain from \eqref{eq : h scalaire u proof prop eval}:
\[
\int_{\cX} h(x) u(x) \, dx=	\lim_{M \rightarrow \infty} \int_{\widehat{B}} \left( \int_{\cX } h(x) I_d^{(M)}(f)(x) \, dx \right) \mu(df).
\]
Now, for each \( f \in \widehat{B} \), we have \( I_d^{(M)}(f) \xrightharpoonup{\star} f \) as \( M \to \infty \), which implies
\[
\int_{\cX} h(x) I_d^{(M)}(f)(x) \, dx \longrightarrow \int_{\cX} h(x) f(x) \, dx.
\]
Moreover, since \( I_d^{(M)}(f) \in \widehat{B} \), this quantity is uniformly bounded by \( \|h\|_{L^1(\cX)} e^\alpha \). On the other hand, the right-hand side above equals \( \int_{\cX} h(x) e_x(f) \, dx \), because \( f(x) = e_x(f) \) almost everywhere.
Hence, we may apply the dominated convergence theorem to deduce:
\begin{align*}
	\lim_{M \rightarrow \infty} \int_{\widehat{B}} \left( \int_{\cX} h(x) I_d^{(M)}(f)(x) \, dx \right) \mu(df) 
	&= \int_{\widehat{B}} \left( \int_{\cX} h(x) f(x) \, dx \right) \mu(df) \\
	&= \int_{\widehat{B}} \left( \int_{\cX} h(x) e_x(f) \, dx \right) \mu(df).
\end{align*}
Since the map \( (x, f) \mapsto h(x) e_x(f) \) is bi-measurable and bounded, we can again apply Fubini's theorem to conclude:
\[
\int_{\cX} h(x) u(x) \, dx = \int_{\cX} h(x) \left( \int_{\widehat{B}} e_x(f) \, \mu(df) \right) dx.
\]
As this equality holds for every \( h \in L^1(\cX) \cap 
L^\infty(\cX) \), we deduce that for almost every \( x \in  \cX \),
\[
u(x) = \int_{\widehat{B}} e_x(f) \, \mu(df),
\]
which is precisely the pointwise identity we set out to rigorously establish. The Point 2 of the proposition is obtained.
\\
\\
\\
\textbf{Step 4:} In this final step, we show the third point of the proposition.
We define the set
\[
R := \left\{ (f,x) \in \widehat{B} \times \cX : \limsup_n u_n(f,x) = \liminf_n u_n(f,x) \right\},
\]
where we recall that $u_n$ is defined in \eqref{eq : def approx of evaluation}.
The set \( R \) therefore captures all pairs \( (f,x) \) for which the sequence \( (u_n(f,x))_n \) converges, which is for instance the case when \( x \) is a Lebesgue point for the function \( f \). Using a standard characterization of convergence, we may rewrite \( R \) as
\[
R = \bigcap_{q=1}^{\infty} \bigcup_{n=1}^{\infty} \bigcap_{p_1,p_2 \geq n} \left\{ (f,x) : |u_{p_1}(f,x) - u_{p_2}(f,x)| \leq \frac{1}{q} \right\},
\]
which establishes the measurability of the set $R$ from those of $u_n$.

We analyze the integral
\[
\int_{\cX} \int_{\widehat{B}} \mathbbm{1}_{R^c}(f,x) \, \mu(df) \, dx.
\]
Thanks to the measurability,  we may apply Fubini’s theorem to exchange the order of integration:
\[
\int_{\cX} \int_{\widehat{B}} \mathbbm{1}_{R^c}(f,x) \, \mu(df) \, dx = \int_{\widehat{B}} \int_{\cX} \mathbbm{1}_{R^c}(f,x) \, dx \, \mu(df).
\]
As already discussed, for every \( f \in \widehat{B} \), almost every \( x \in \cX  \) is a Lebesgue point of \( f \). This implies
\[
\int_{\cX} \mathbbm{1}_{R^c}(f,x) \, dx \le \int_{\cX} \mathbbm{1}_{\{x \colon x \text{ is not a Lebesgue point of } f\}} \, dx = 0.
\]
Thus, we conclude that
\[
\int_{\cX} \int_{\widehat{B}}\mathbbm{1}_{R^c}(f,x) \, \mu(df) \, dx = 0.
\]
From the definition of the set $R$, this proves the third point of the proposition.
%
%
%
\medskip

\noindent This completes the proof.

\end{proof}

\subsection{Proof of other lemmas}
We conclude by proving Lemma~\ref{lem : determination max fini}, as stated in Section~\ref{s: proof main} within the proofs of the main results.

\subsubsection{Proof of Lemma \ref{lem :  determination max fini}}
\begin{proof}
By definition \eqref{eq : def nplus}, we have $n_\beta^+ \in [0,1]$ for all $\beta$, so we can upper and lower bound the denominator of $i_\beta$
{\revarn in \eqref{Eq : i beta suppA rappel},} obtaining:
\begin{equation} \label{eq : deno not so important fini}
	e^{-\alpha} \left( \int_\cX s_{\theta}(x) \indi{F^+_\beta}(x) p_\theta(x) {\revarna \nu_\cX}(dx) \right)^2
	\le i_\beta \le 
	\left( \int_\cX s_{\theta}(x) \indi{F^+_\beta}(x) p_\theta(x) {\revarna \nu_\cX}(dx) \right)^2.
\end{equation}
Regarding the numerator of $i_\beta$, let us recall
\[
F_\text{max}^+ = \{x \in \mathcal{X} \mid s_\theta(x) > 0\}, \quad
F_\text{max}^{\prime +} = \{x \in \mathcal{X} \mid s_\theta(x) < 0\}.
\]
Since $s_\theta(x) \neq 0$ for all $x \in \cX$ by assumption, we have 
\[
F_\text{max}^+ \cup F_\text{max}^{\prime +} = \cX.
\]
Recalling that the score function is centered, i.e., 
\[
\int_\cX s_\theta(x) p_\theta(x) {\revarna \nu_\cX}(dx) = 0,
\]
we deduce that:
\begin{equation} \label{eq: symmetry 12.5}
	\int_\cX s_{\theta}(x) \indi{F^+_\text{max}}(x) p_\theta(x) {\revarna \nu_\cX}(dx)
	= - \int_\cX s_{\theta}(x) \indi{F^{\prime +}_\text{max}}(x) p_\theta(x) {\revarna \nu_\cX}(dx)
	= \frac{1}{2} \int_\cX \abs{s_{\theta}(x)} p_\theta(x) {\revarna \nu_\cX}(dx).
\end{equation}
Since $\{F_\text{max}^+, F_\text{max}^{\prime +}\}$ is a partition of $\cX$, for any $F \subset \cX$, we can write:
\[
\int_F s_\theta(x) p_\theta(x) {\revarna \nu_\cX}(dx) 
= \int_{F \cap F^+_\text{max}} s_\theta(x) p_\theta(x) {\revarna \nu_\cX}(dx)
+ \int_{F \cap F^{\prime +}_\text{max}} s_\theta(x) p_\theta(x) {\revarna \nu_\cX}(dx).
\]
This implies:
\[
\int_F s_\theta(x) p_\theta(x) {\revarna \nu_\cX}(dx) \in 
\left[ -\frac{1}{2} \int_\cX \abs{s_{\theta}(x)} p_\theta(x) {\revarna \nu_\cX}(dx), 
\frac{1}{2} \int_\cX \abs{s_{\theta}(x)} p_\theta(x) {\revarna \nu_\cX}(dx) \right],
\]
and consequently:
\begin{equation} \label{eq : iF smaller iFmax}
	\left( \int_F s_\theta(x) p_\theta(x) {\revarna \nu_\cX}(dx) \right)^2
	\le \frac{1}{4} \left( \int_\cX \abs{s_{\theta}(x)} p_\theta(x) {\revarna \nu_\cX}(dx) \right)^2.
\end{equation}
Moreover, by assumption $s_\theta(x) \neq 0$ for all $x \in \cX$ and from \eqref{eq: symmetry 12.5}, equality in \eqref{eq : iF smaller iFmax} holds if and only if $F = F_\text{max}^+$ or $F = F_\text{max}^{\prime +}$.

Since the collection of sets $\{F^+_\beta\}_{\beta=0,\dots,2^d-1}$ introduced in \eqref{eq: Fbeta+ 4.8} describes all subsets of $\cX$, there exist indices $\beta_\text{max}, \beta'_\text{max}$ such that
\[
F_{\beta_\text{max}}^+ = F_\text{max}^+, \quad 
F_{\beta'_\text{max}}^+ = F_\text{max}^{\prime +}.
\]
Therefore, the function
\[
\beta \mapsto \left( \int_\cX s_\theta(x) \indi{F^+_\beta}(x) p_\theta(x) {\revarna \nu_\cX}(dx) \right)^2
\]
reaches its maximum value only at $\beta = \beta_\text{max}$ and $\beta = \beta'_\text{max}$.
{\revarna This motivates the definition 
	\begin{equation} \label{Eq : ratio i(3) imax}
		K:=\frac{
			\left( \int_{F^+_{\beta_\text{max}}} s_\theta(x) p_\theta(x) {\revarna \nu_\cX}(dx) \right)^2}{
			\sup_{\beta \notin \{\beta_\text{max}, \beta'_\text{max}\}}  \left( \int_{F^+_\beta} s_\theta(x) p_\theta(x) {\revarna \nu_\cX}(dx) \right)^2}.
	\end{equation}
	which satisfies $K>1$.
}
As a consequence, using the bounds in \eqref{eq : deno not so important fini}, for $\alpha$ sufficiently small (say, $\alpha < \overline{\alpha}$), the function $\beta \mapsto i_\beta$ reaches its maximum on $\{\beta_\text{max}, \beta'_\text{max}\}$.
An explicit constraint on $\overline{\alpha}$ is given {\revarna by $e^{\overline{\alpha}} < K$.}

Recalling the {\revarn expression \eqref{Eq : i beta suppA rappel}, where $n_\beta^+$ appears in the denominator of $i_\beta$ , we observe that:}
\begin{itemize}
	\item If $n^+_{\beta_{\text{max}}} \leq n^+_{\beta'_{\text{max}}}$, then $i_{\beta_{\text{max}}} \geq i_{\beta'_{\text{max}}}$,
	\item If $n^+_{\beta_{\text{max}}} \geq n^+_{\beta'_{\text{max}}}$, then $i_{\beta_{\text{max}}} \leq i_{\beta'_{\text{max}}}$,
\end{itemize}
since the numerators are equal for both indices. Assume without loss of generality that $n^+_{\beta_{\text{max}}} \le n^+_{\beta'_{\text{max}}}$. Then we have:
\[
i_{\beta_{\text{max}}} \ge i_{\beta'_{\text{max}}} > i_\beta, \quad \forall \beta \notin \{\beta_\text{max}, \beta'_\text{max}\}.
\]
From the definition \eqref{eq : def nplus} of $n_\beta^+$, the explicit expressions for the values at the maxima are:
\begin{equation} \label{eq : explicit i iprime}
	\begin{aligned}
		i_{\beta_{\text{max}}}
		&= \frac{\frac{1}{4} \left( \int_\cX \abs{s_{\theta}(x)} p_\theta(x) {\revarna \nu_\cX}(dx) \right)^2}
		{1 + (e^\alpha - 1) \int_{F^+_\text{max}} p_\theta(x) {\revarna \nu_\cX}(dx)}, \\
		i_{\beta'_{\text{max}}}
		&= \frac{\frac{1}{4} \left( \int_\cX \abs{s_{\theta}(x)} p_\theta(x) {\revarna \nu_\cX}(dx) \right)^2}
		{1 + (e^\alpha - 1) \int_{F^{\prime +}_\text{max}} p_\theta(x) {\revarna \nu_\cX}(dx)}.
	\end{aligned}
\end{equation}

Now, the rest of the proof consists in showing that the maximization problem \eqref{eq : optimization pb finite} has a maximum that is attained for a sub-probability measure $(\omega_\beta)_{\beta}$ supported on $\{\beta_\text{max},\beta'_\text{max}\}$.

Let us denote by $R$ the staircase pattern matrix of size $d \times 2^d$ as defined in Definition~\ref{def: staircase matrix}. It satisfies
\[
R_{j,\beta} = (r_\beta)_j = e^\alpha \indi{F^+_\beta}(x_j) + \indi{F^-_\beta}(x_j).
\]
This matrix has columns given by the vectors defined in \eqref{eq : R matrix by slice}. Then, the optimization problem \eqref{eq : optimization pb finite} can be rewritten as
\[
M^* = \max_{\omega \in \R^{2^d}_+} \omega^T i \quad \text{subject to} \quad R \omega = \bm{1},
\]
where $i = (i_\beta)_\beta$ and $\bm{1} = [1, \dots, 1]^T \in \R^d$. This problem admits a solution since the constraint defines a compact set. Let $\overline{\omega}$ be a maximizing solution. Define 
\[
q = \#\left\{ \beta \in \{0, \dots, 2^d - 1\} \mid \overline{\omega}_\beta > 0 \right\}.
\]
We distinguish three cases depending on the value of $q$.

\paragraph{$\bullet$ Case $q = 1$}
In this case, 
$\overline{\omega}$ is
supported on a single index $\bar{\beta}$. The constraint becomes:
\begin{equation}\label{eq: system q=1}
	\begin{cases}
		1 = (r_{\bar{\beta}})_1 \bar{\omega}_{\bar{\beta}} \\
		\vdots \\
		1 = (r_{\bar{\beta}})_d \bar{\omega}_{\bar{\beta}}.
	\end{cases}
\end{equation}
From the definition of $r_\beta$, for $\bar{\beta} = 0$ we have $(r_{\bar{\beta}})_j = 1$ for all $j$, and for $\bar{\beta} = 2^d - 1$, we have $(r_{\bar{\beta}})_j = e^\alpha$ for all $j$. For other values of $\bar{\beta}$, the vector $r_{\bar{\beta}}$ has mixed entries in $\{1, e^\alpha\}$. Thus, the system~\eqref{eq: system q=1} has a solution only if $\bar{\beta} \in \{0, 2^d - 1\}$. However, for these values, the objective coefficients vanish:
\[
i_\beta = \frac{\left( \int_{F^+_\beta} s_{\theta}(x) p_\theta(x){\revarna \nu_\cX}(dx)\right)^2}{1 + (e^\alpha - 1) n^+_\beta} = 0,
\]
since $F^+_0 = \emptyset$ and $F^+_{2^d - 1} = \mathcal{X}$, and the score function integrates to $0$ over the full space. Hence, $M^* = 0$. We will show that for $q = 2$, $M^* > 0$, so the maximum cannot be attained with $q = 1$.

\paragraph{$\bullet$ Case $q = 2$}
Assume there exist $0 \le \beta_1 < \beta_2 \le 2^d - 1$ such that $\overline{\omega}_{\beta_1}, \overline{\omega}_{\beta_2} > 0$ and $\overline{\omega}_\beta = 0$ for all other $\beta$. Exclude the case $\beta_1 = 0$, $\beta_2 = 2^d - 1$, since that would again imply $M^* = 0$.

Because $r_{\beta_1} \ne r_{\beta_2}$, there exists $x \in \mathcal{X}$ such that $R_{x,\beta_1} \ne R_{x,\beta_2}$. Without loss of generality, assume $R_{x,\beta_1} = 1$, $R_{x,\beta_2} = e^\alpha$. Since $\beta_1 \ne 0$, we can find $y \in \mathcal{X}$ such that $R_{y,\beta_1} = e^\alpha$ (as $r_0$ is the only vector with all entries equal to $1$). The constraint $R \overline{\omega} = \bm{1}$ yields:
\begin{equation}\label{eq: system 14.5}
	\left\{
	\begin{aligned}
		\overline{\omega}_{\beta_1} + e^\alpha \overline{\omega}_{\beta_2} &= 1 \\
		e^\alpha \overline{\omega}_{\beta_1} + R_{y,\beta_2} \overline{\omega}_{\beta_2} &= 1.
	\end{aligned}
	\right.
\end{equation}
To ensure a solution with both $\overline{\omega}_{\beta_1}, \overline{\omega}_{\beta_2} > 0$, it must be that $R_{y,\beta_2} = 1$. Then the solution is
\[
\overline{\omega}_{\beta_1} = \overline{\omega}_{\beta_2} = \frac{1}{1 + e^\alpha},
\]
and the objective function yields
\[
M^* = \frac{i_{\beta_1} + i_{\beta_2}}{1 + e^\alpha}.
\]
This value is maximal when $i_{\beta_1}$ and $i_{\beta_2}$ are the two largest among the $(i_\beta)_\beta$, hence:
\begin{equation}\label{eq : maximume q is 2}
	M^* = \frac{i_{\beta_{\max}} + i_{\beta'_{\max}}}{1 + e^\alpha}.
\end{equation}
Using the explicit form \eqref{eq : explicit i iprime} and the identity
\[
\int_{F^{\prime +}_{\max}} p_\theta(x){\revarna \nu_\cX}(dx) = 1 - \int_{F^{+}_\text{max}} p_\theta(x){\revarna \nu_\cX}(dx),
\]
we deduce that
\[
i_{\beta_{\max}} + i_{\beta'_{\max}} = \frac{1}{4} \mathbb{E}[|s_\theta(x)|]^2 \cdot \frac{1 + e^\alpha}{\left(1 + (e^\alpha - 1)n_{\max}\right)\left(1 + (e^\alpha - 1)(1 - n_{\max})\right)}.
\]
Hence,
\[
M^* = \frac{1}{4} \mathbb{E}[|s_\theta(x)|]^2 \cdot \frac{1}{\left(1 + (e^\alpha - 1)n_{\max}\right)\left(1 + (e^\alpha - 1)(1 - n_{\max})\right)},
\]
which gives the desired expression \eqref{eq : def M star}, noting that
\[
1 + (e^\alpha - 1)(1 - n_{\max}) = e^\alpha(1 - n_{\max}) + n_{\max}.
\]

\paragraph{$\bullet$ Case $q \ge 3$}
We will show that if $\alpha$ is small enough, then the maximum cannot be attained, leading to a contradiction and implying that the assumption $q \ge 3$ is false.
Let $\beta_{1}, \dots, \beta_{q} \in \{0,\dots,2^d-1\}$ be the indices such that $\overline{\omega}_{\beta_{l}} > 0$ for $l = 1,\dots,q$.

We divide the analysis into several subcases:

\subparagraph{Subcase 1:} Assume that $\{\beta_{\max}, \beta'_{\max}\} \cap \{\beta_{1},\dots,\beta_{q}\} = \emptyset$.  
That is, the support of $\overline{\omega}$ is disjoint from the two indices where $i_\beta$ is maximal. Define
\[
i_{(3)} := \max\left\{ i_{\beta} \mid \beta \in \{0,\dots,2^d-1\} \setminus \{\beta_{\max}, \beta'_{\max} \} \right\}.
\]
Since $i_{\beta_{\max}} \ge i_{\beta'_{\max}} > i_{(3)}$, we obtain
\[
\overline{\omega}^T i = \sum_{l=1}^q \overline{\omega}_{\beta_l} i_{\beta_l} \le i_{(3)} \sum_{l=1}^q \overline{\omega}_{\beta_l} \le i_{(3)},
\]
where we used that $\sum_{\beta=0}^{2^d-1} \overline{\omega}_\beta \le 1$ due to the constraint $R\overline{\omega} = \bm{1}$.  
If $\alpha$ is small enough so that
\begin{equation} \label{Eq : contrainte alpha alphabar 2}
	i_{(3)}  <  \frac{i_{\beta_{\max}} + i_{\beta'_{\max}}}{1 + e^\alpha},
\end{equation}
then this contradicts the optimality of this subcase, as the value is strictly less than the value achieved in the case $q = 2$ (see \eqref{eq : maximume q is 2}).
{\revarna Using \eqref{eq : deno not so important fini} and recalling the definition of $K$ in \eqref{Eq : ratio i(3) imax}, a sufficient condition for \eqref{Eq : contrainte alpha alphabar 2}, is $\alpha< \overline{\alpha}$ where
	$ \overline{\alpha} \left(\frac{1+e^{\overline{\alpha}}}{2}\right)<K$.}

\subparagraph{{\revarna Subcase 2:}} Assume that both $\beta_{\max}$ and $\beta'_{\max}$ belong to $\{\beta_1, \dots, \beta_q\}$.  
Without loss of generality, relabel such that $\beta_1 = \beta_{\max}$ and $\beta_2 = \beta'_{\max}$.  

Since $F^+_{\max} \cup F^{\prime +}_{\max} = \mathcal{X}$, there exist points $x \in F^+_{\max} = F^+_{\beta_1}$ and $y \in F^{\prime +}_{\max} = F^+_{\beta_2} = \mathcal{X} \setminus F^+_{\beta_1}$.  
From the definition of the sets $F^+_\beta$, we have:
\[
R_{x,\beta_1} = e^\alpha, \quad R_{x,\beta_2} = 1, \quad R_{y,\beta_1} = 1, \quad R_{y,\beta_2} = e^\alpha.
\]
Writing the corresponding rows of $R \overline{\omega} = \bm{1}$ at points $x$ and $y$ gives:
\begin{equation} \label{eq : syst 2x2 subcase3}
	\left\{
	\begin{aligned}
		e^\alpha \overline{\omega}_{\beta_1} + \overline{\omega}_{\beta_2} + \sum_{l=3}^q R_{x,\beta_l}\overline{\omega}_{\beta_l} &= 1, \\
		\overline{\omega}_{\beta_1} + e^\alpha \overline{\omega}_{\beta_2} + \sum_{l=3}^q R_{y,\beta_l}\overline{\omega}_{\beta_l} &= 1.
	\end{aligned}
	\right.
\end{equation}

Define $s_x := \sum_{l=3}^q R_{x,\beta_l}\overline{\omega}_{\beta_l}$ and $s_y := \sum_{l=3}^q R_{y,\beta_l}\overline{\omega}_{\beta_l}$.  
Solving \eqref{eq : syst 2x2 subcase3} for $\overline{\omega}_{\beta_1}$ and $\overline{\omega}_{\beta_2}$ gives:
\begin{equation} \label{eq : expression omega subcase 3}
	\overline{\omega}_{\beta_1} = \frac{1}{1 + e^\alpha} - \frac{s_x e^\alpha - s_y}{e^{2\alpha} - 1}, \quad
	\overline{\omega}_{\beta_2} = \frac{1}{1 + e^\alpha} - \frac{s_y e^\alpha - s_x}{e^{2\alpha} - 1}.
\end{equation}

Define $\sigma := \sum_{l=3}^q \overline{\omega}_{\beta_l} > 0$.  
Since all entries of $R$ are either $1$ or $e^\alpha$, we have:
\[
\sigma \le s_x \le e^\alpha \sigma, \quad \sigma \le s_y \le e^\alpha \sigma.
\]
In particular, both $s_x e^\alpha - s_y \ge 0$ and $s_y e^\alpha - s_x \ge 0$. We now compute:
\begin{align*}
	\overline{\omega}^T i &= \overline{\omega}_{\beta_1} i_{\beta_1} + \overline{\omega}_{\beta_2} i_{\beta_2} + \sum_{l=3}^q \overline{\omega}_{\beta_l} i_{\beta_l} \\
	&= \overline{\omega}_{\beta_1} i_{\beta_{\max}} + \overline{\omega}_{\beta_2} i_{\beta'_{\max}} + \sum_{l=3}^q \overline{\omega}_{\beta_l} i_{\beta_l} \\
	&= \frac{i_{\beta_{\max}} + i_{\beta'_{\max}}}{1 + e^\alpha}
	- \frac{s_x e^\alpha - s_y}{e^{2\alpha} - 1} i_{\beta_{\max}} 
	- \frac{s_y e^\alpha - s_x}{e^{2\alpha} - 1} i_{\beta'_{\max}} 
	+ \sum_{l=3}^q \overline{\omega}_{\beta_l} i_{\beta_l} \\
	&\le \frac{i_{\beta_{\max}} + i_{\beta'_{\max}}}{1 + e^\alpha}
	- \frac{s_x + s_y}{e^\alpha + 1} i_{\beta'_{\max}} 
	+ \sum_{l=3}^q \overline{\omega}_{\beta_l} i_{\beta_l}, \quad 
	\text{as $i_{\beta_{\max}} \ge i_{\beta'_{\max}}$,} \\
	&\le \frac{i_{\beta_{\max}} + i_{\beta'_{\max}}}{1 + e^\alpha}
	+ \sigma \left( i_{(3)} - 2 \frac{i_{\beta'_{\max}}}{e^\alpha + 1} \right),
\end{align*}
where we used $s_x + s_y \ge 2 \sigma$ and $i_{\beta_l} \le i_{(3)}$ for $l \ge 3$.

As $i_{(3)} < i_{\beta'_{\max}}$, choosing $\alpha$ small enough ensures that 
\begin{equation} \label{Eq : contrainte alpha alphabar 3}
	i_{(3)} - 2 \frac{i_{\beta'_{\max}}}{e^\alpha + 1} < 0,
\end{equation}
thus the final term is negative and the value of $\overline{\omega}^T i$ is strictly less than $\frac{i_{\beta_{\max}} + i_{\beta'_{\max}}}{1 + e^\alpha}$, hence the maximum is not attained in this subcase. {\revarna One can check that the condition $\alpha<\overline{\alpha}$ where 	$ \overline{\alpha} \left(\frac{1+e^{\overline{\alpha}}}{2}\right)<K$ is sufficient
	to ensure \eqref{Eq : contrainte alpha alphabar 3}.	
}

\subparagraph{Subcase 3:} The complement of the previous subcases.  
{\revarna By exclusion of the subcases 1 and 2,} exactly one of $\beta_{\max}$ or $\beta'_{\max}$ appears among $\{\beta_1, \dots, \beta_q\}$.
{\revarna Hence, there exists an index $l_1$ such that}
\[
i_{\beta_{l_1}} = \max_{\beta} i_\beta {\revarna ~ \in \{ i_{\beta_{\max}},i_{\beta'_{\max}} \}} \quad \text{and } i_{\beta_l} \le i_{(3)} \text{ for all } l \neq l_1.
\]

The proof proceeds analogously to {\revarna Subcase 2.} One extracts a $2 \times 2$ system from $R \overline{\omega} = \bm{1}$ involving $\beta_{l_1}$ and some $\beta_{l_2} \neq \beta_{l_1}$, provided the geometric condition
\begin{equation} \label{eq : subcase 4 finite cond on F}
	F^+_{\beta_{l_1}} \cap \left(F^+_{\beta_{l_2}}\right)^c \neq \emptyset, \quad 
	\left(F^+_{\beta_{l_1}}\right)^c \cap F^+_{\beta_{l_2}} \neq \emptyset,
\end{equation}
holds. This is stated and proven in Lemma~\ref{lem : for subcase 4 finite cond on F} below.

In particular, let $x \in \left( F^{+}_{\beta_{l_2}} \right)^c \cap F^+_{\beta_{l_1}}$ and $y \in F^+_{\beta_{l_2}} \cap \left( F^{+}_{\beta_{l_1}} \right)^c$.  
By the definition of the sets $F^+_\beta$, we have:
\[
R_{x,\beta_{l_1}} = e^\alpha, \quad R_{x,\beta_{l_2}} = 1, \quad R_{y,\beta_{l_1}} = 1, \quad R_{y,\beta_{l_2}} = e^\alpha.
\]

We now write the two equations of the system $R\overline{\omega} = \bm{1}$ corresponding to the rows indexed by $x$ and $y$. Similarly to \eqref{eq : syst 2x2 subcase3}, we obtain:
\begin{equation*}
	\left\{
	\begin{aligned}
		e^\alpha \, \overline{\omega}_{\beta_{l_1}} + \overline{\omega}_{\beta_{l_2}} 
		+ \sum_{\substack{l=1 \\ l \notin \{l_1,l_2\}}}^q R_{x,\beta_l} \, \overline{\omega}_{\beta_l} &= 1, \\
		\overline{\omega}_{\beta_{l_1}} + e^\alpha \, \overline{\omega}_{\beta_{l_2}} 
		+ \sum_{\substack{l=1 \\ l \notin \{l_1,l_2\}}}^q R_{y,\beta_l} \, \overline{\omega}_{\beta_l} &= 1.
	\end{aligned}
	\right.
\end{equation*}

Let us denote the remaining sums by:
\[
\sigma_x := \sum_{\substack{l=1 \\ l \notin \{l_1,l_2\}}}^q R_{x,\beta_l} \, \overline{\omega}_{\beta_l}, \qquad 
\sigma_y := \sum_{\substack{l=1 \\ l \notin \{l_1,l_2\}}}^q R_{y,\beta_l} \, \overline{\omega}_{\beta_l}.
\]

Then, as in \eqref{eq : expression omega subcase 3}, we deduce:
\begin{equation*}
	\overline{\omega}_{\beta_{l_1}} = \frac{1}{1 + e^\alpha} - \frac{\sigma_x e^\alpha - \sigma_y}{e^{2\alpha} - 1}, \qquad
	\overline{\omega}_{\beta_{l_2}} = \frac{1}{1 + e^\alpha} - \frac{\sigma_y e^\alpha - \sigma_x}{e^{2\alpha} - 1}.
\end{equation*}

We now compute:
\begin{align*}
	\overline{\omega}^T i &= \overline{\omega}_{\beta_{l_1}} i_{\beta_{l_1}} + \overline{\omega}_{\beta_{l_2}} i_{\beta_{l_2}} 
	+ \sum_{\substack{l=1 \\ l \notin \{l_1,l_2\}}}^q \overline{\omega}_{\beta_l} i_{\beta_l} \\
	&\le \overline{\omega}_{\beta_{l_1}} i_{\beta_{\max}} + \overline{\omega}_{\beta_{l_2}} i_{\beta'_{\max}} 
	+ \sum_{\substack{l=1 \\ l \notin \{l_1,l_2\}}}^q \overline{\omega}_{\beta_l} i_{\beta_l},
\end{align*}
where we used that $i_{\beta_{l_1}} \revarna{ \leq~ } i_{\beta_{\max}}$ and $i_{\beta_{l_2}} \le i_{(3)} \le i_{\beta'_{\max}}$. Then, exactly as in {\revarna Subcase~2,} we deduce that
{\revarna for $\alpha$ small enough,}
\[
\overline{\omega}^T i < \frac{i_{\beta_{\max}} + i_{\beta'_{\max}}}{1 + e^\alpha}.
\]
Consequently, the maximum $M^*$ cannot be attained in {\revarna Subcase~3.}
\medskip 

In all cases, we obtain a contradiction to the assumption that $q \ge 3$ yields the maximum. Hence, the maximum is reached only when $q = 2$.

\medskip

This concludes the proof, up to showing that the geometric condition \eqref{eq : subcase 4 finite cond on F} is satisfied for some $\beta_{l_2}$.  
This is verified in Lemma~\ref{lem : for subcase 4 finite cond on F} below.

\end{proof}

\begin{lem}\label{lem : for subcase 4 finite cond on F}
There exists some $l_2 \in \{1,\dots,q\}$, $l_2\neq l_1$ such that,
\begin{equation*} 
	\left( F^{+}_{\beta_{l_2}}\right)^c \cap  F^+_{\beta_{l_1}}  \neq \emptyset \quad \text{and} \quad 	F^+_{\beta_{l_2}} \cap \left( F^{+}_{\beta_{l_1}}\right)^c \neq \emptyset.
\end{equation*}
\end{lem}
\begin{proof}
We proceed by contradiction. Suppose the statement does not hold. Then, for all $l \neq l_1$, we must have:
\begin{equation*} 
	\left( F^{+}_{\beta_l} \right)^c \cap F^+_{\beta_{l_1}} = \emptyset 
	\quad \text{or} \quad 
	F^+_{\beta_l} \cap \left( F^+_{\beta_{l_1}} \right)^c = \emptyset,
\end{equation*}
which implies:
\begin{equation} \label{eq: F subset 18.5}
	F^+_{\beta_{l_1}} \subset F^+_{\beta_l}
	\quad \text{or} \quad
	F^+_{\beta_l} \subset F^+_{\beta_{l_1}}.
\end{equation}
Let us choose some $x \in F^+_{\beta_{l_1}}$ and $y \in \left( F^+_{\beta_{l_1}} \right)^c$.  
This is possible because $F^+_{\beta_{l_1}} \neq \emptyset$ and $F^+_{\beta_{l_1}} \neq \cX$,  
otherwise we would have {\revarna $i_{\beta_{l_1}} = 0$,} which contradicts $M^* > 0$.

By definition of the matrix $R$, we know that:
\[
R_{x,\beta_{l_1}} = e^\alpha, \qquad R_{y,\beta_{l_1}} = 1.
\]
Now fix any $l \in \{1, \dots, q\} \setminus \{l_1\}$. Then, one of the two alternatives in \eqref{eq: F subset 18.5} must occur:

- Case 1: If $F^+_{\beta_{l_1}} \subset F^+_{\beta_l}$, then since $x \in F^+_{\beta_{l_1}}$, we also have $x \in F^+_{\beta_l}$, hence:
\[
R_{x,\beta_l} = e^\alpha \ge R_{y,\beta_l} \in \{1, e^\alpha\}.
\]

- Case 2: If $F^+_{\beta_l} \subset F^+_{\beta_{l_1}}$, then since $y \notin F^+_{\beta_{l_1}}$, it follows that $y \notin F^+_{\beta_l}$, hence:
\[
R_{y,\beta_l} = 1 \le R_{x,\beta_l} \in \{1, e^\alpha\}.
\]
In both cases, we conclude:
\[
R_{y,\beta_l} \le R_{x,\beta_l} \quad \text{for all } l \in \{1, \dots, q\} \setminus \{l_1\}, \qquad
R_{y,\beta_{l_1}} = 1 < e^\alpha = R_{x,\beta_{l_1}}.
\]
This leads to a contradiction with the strict positivity of the weights $\overline{\omega}_{\beta_l}$ and the fact that:
\begin{equation*}
	\sum_{l=1}^q R_{y,\beta_l} \, \overline{\omega}_{\beta_l} = 1 = 
	\sum_{l=1}^q R_{x,\beta_l} \, \overline{\omega}_{\beta_l}.
\end{equation*}
The proof is therefore complete.
\end{proof}
{\revarna
\begin{rem}
	From the proof of Lemma \ref{lem :  determination max fini}, we see that the condition on $\overline{\alpha}$ in Theorem \ref{th : info cas fini} is
	$e^{\overline{\alpha}} \left(\frac{e^{\overline{\alpha}}+1}{2}\right)<K$ where
	$K$ is defined in \eqref{Eq : ratio i(3) imax}.
\end{rem}
}

\end{appendices}

\end{document}